%% file: HoroSarkisov.tex
\documentclass[english]{amsart} 

\usepackage{amsfonts,amsmath,amssymb,color,amscd,amsthm}
\usepackage[T1]{fontenc} 
\usepackage[english]{babel}
\usepackage{mathtools}\usepackage{gensymb}
 
\usepackage{pgf,tikz}
\usepackage{pgfplots} \usepgfplotslibrary{patchplots}
\usetikzlibrary{cd,arrows,positioning}
\tikzset{>=stealth}
\tikzcdset{arrow style=tikz}
\tikzset{link/.style={column sep=1.8cm,row sep=0.16cm}}
\tikzset{map/.style={row sep=0em, column sep=0em}}

\usepackage{soul}

\usepackage{tikz}
\usetikzlibrary{shapes,arrows}

\usepackage[curve,all]{xy}
\xyoption{matrix}
 \xyoption{curve}
 \xyoption{color}
 \xyoption{line}
\xyoption{arc}

\usepackage[shortlabels]{enumitem}
\setlist[1]{wide}
\setlist[2]{leftmargin=15mm}
\setlist[enumerate]{label=\rm{(\arabic*)}}
\setlist[enumerate,2]{label=\rm({\it\roman*}), }
\setlist[itemize]{label=\raisebox{0.25ex}{\tiny$\bullet$}}

\usepackage{thmtools}
\usepackage{thm-restate}
\usepackage[backref, colorlinks, linktocpage, citecolor = blue, linkcolor = blue]{hyperref}

\usepackage{cleveref}

\usetikzlibrary{shapes}

\theoremstyle{plain}
\newtheorem{teo}{Theorem}
\newtheorem{lem}[teo]{Lemma}
\newtheorem{prop}[teo]{Proposition}
\newtheorem{cor}[teo]{Corollary}
\newtheorem{claim}[teo]{Claim}
\theoremstyle{definition}
\newtheorem{defi}{Definition}
\newtheorem{notc}{Notation/Construction}
\newtheorem{notat}{Notation}

\newtheorem{rem}{Remark}
\newtheorem{cond}{Condition}

\newcommand{\Hom}{\operatorname{Hom}}

\newcommand{\Ker}{\operatorname{Ker}}

\renewcommand{\dim}{\operatorname{dim}}

\renewcommand{\Im}{\operatorname{Im}}
\renewcommand{\Ker}{\operatorname{Ker}}

\newcommand{\Min}{\operatorname{Min}}

\newcommand{\Cbb}{{\mathbb C}}
\newcommand{\Qbb}{{\mathbb Q}}
\newcommand{\Zbb}{{\mathbb Z}}
\newcommand{\Rbb}{{\mathbb R}}
\newcommand{\Pbb}{{\mathbb P}}
\newcommand{\Nbb}{{\mathbb N}}

\def\reseau1{
\draw[color=gray!60] (1,0) -- (1,4);
\draw[color=gray!60] (2,0) -- (2,4);
\draw[color=gray!60] (3,0) -- (3,4);

\draw[color=gray!60] (0,1) -- (4,1);
\draw[color=gray!60] (0,2) -- (4,2);
\draw[color=gray!60] (0,3) -- (4,3);
\node at (1.7,1.55) [color=gray] {0};
}
\def\reseau2{
\draw[color=gray!60] (-2,-3) -- (-2,3);
\draw[color=gray!60] (-1,-3) -- (-1,3);
\draw[color=gray!60] (0,-3) -- (0,3);
\draw[color=gray!60] (1,-3) -- (1,3);
\draw[color=gray!60] (2,-3) -- (2,3);
\draw[color=gray!60] (-3,-2) -- (3,-2);
\draw[color=gray!60] (-3,-1) -- (3,-1);
\draw[color=gray!60] (-3,0) -- (3,0);
\draw[color=gray!60] (-3,1) -- (3,1);
\draw[color=gray!60] (-3,2) -- (3,2);
\node at (-0.4,-0.5) [color=gray] {0};
}

\title[HoroSarkisov program]{A description of the Sarkisov program of horospherical varieties via moment polytopes}
\author{Enrica Floris$^1$ \and Boris Pasquier$^1$}
\date{\today\\
\indent $^1$Universit\'e de Poitiers, CNRS, LMA, Poitiers, France}

\begin{document}

\thanks{We would like to thank M.~Brion for his comments on an earlier draft of this paper and R.~Terpereau for useful conversations.\\
\indent
Both authors acknowledge support by the ANR Project FIBALGA ANR-18-CE40-0003-01}

\begin{abstract}
Let $X$ and $Y$ be horospherical Mori fibre spaces which are birational equivariantly with respect to the group action.
Then, there is a horospherical Sarkisov program from $X/S$ to $Y/T$.
\end{abstract}

\maketitle
\setcounter{tocdepth}{1}
\tableofcontents
\section{Introduction}

Horospherical varieties are examples of varieties endowed with an action of a linear algebraic group, with finitely many orbits. In particular, they have an open orbit and are therefore rational.
An important class of horospherical varieties is given by toric varieties, on which acts the group $\mathbb G_m^n$.
Similarly to toric varieties, horospherical varieties admit a combinatorial description.
Indeed, to a horospherical variety $Z$ and an ample divisor $D$ one can attach a polytope $Q_D$, called \textit{the moment polytope}, describing the geometry of the variety and of the action. 
For instance, some facets of $Q_D$ (those not contained in a wall of the dominant chamber) are in bijection with the divisors of the variety which are invariant by the group.

In \cite{HMMP} the second named author described a minimal model program for horospherical varieties completely in terms of moment polytopes by considering a one-parameter family of polytopes of the form $\{Q_{D+\epsilon K_Z}\}_{\epsilon\in\mathbb Q}$. For small values of $\epsilon$ the polytope $Q_{D+\epsilon K_Z}$ still defines the same variety $Z$, but, since $K_Z$ is not pseudoeffective, as $\epsilon$ grows, facets of the polytope start collapsing. Eventually, the dimension of the polytope drops, defining a fibration to a variety of smaller dimension.

This process, for a suitable choice of $D$, is a minimal model program and ends with a Mori fibre space $X\to S$.

The natural question arises then of the relation between two different horospherical Mori fibre spaces which are outcome of two MMP for two different ample divisors.

Two birational Mori fibre spaces, without any further structure, are connected by a \textit{Sarkisov program} by the cornerstone results of \cite{Cor95} and \cite{HMcK}.
A Sarkisov program is a sequence of diagrams, called \textit{links}, of one of the following forms

\[\begin{array}{cccc} 
I &  II & III & IV\\
\xymatrix{
X'\ar[d]\ar@{-->}[r]& Y\ar^{\psi}[d]\\
X\ar_{\phi}[d]& T\ar[ld]\\
S&
} & 
\xymatrix{
X'\ar[d]\ar@{-->}[r]& Y'\ar[d]\\
X\ar_{\phi}[d]& Y\ar^{\psi}[d]\\
S\ar@{=}[r]&T
} & 
\xymatrix{
X\ar_{\phi}[d]\ar@{-->}[r]& Y'\ar[d]\\
S\ar[rd]& Y\ar^{\psi}[d]\\
&T
} & 
\xymatrix{
X\ar_{\phi}[d]\ar@{-->}[rr]&& Y\ar^{\psi}[d]\\
S\ar[rd]&& T\ar[ld]\\
&R&
}
\end{array}\]

 where $W'\to W$ are extremal divisorial contractions, horizontal dashed arrows are isomorphisms in codimension 1
 and all the other arrows are extremal contractions.
 In type IV, we have two cases depending on wether $S\to R$ and $T\to R$ are fibrations (type IV\textsubscript m) or  small birational maps (type IV\textsubscript s).

If $X/S$ and $Y/T$ are two Mori fibre spaces carrying the action of a connected group and which are birational equivariantly with respect to the group, then by \cite{Flo20} there is an equivariant Sarkisov program, that is one in which all the arrows in the links are equivariant with respect to the group.

In this work, we aim to produce a Sarkisov program in the spirit of \cite{HMMP} and prove the following.

\begin{teo}\label{thm:mainthIntro}
Let $G$ be a connected reductive algebraic group.
Let $X$ and $Y$ be horospherical $G$-varieties which are $G$-equivariantly birational.
Assume moreover that there are Mori fibre space structures $X/S$ and $Y/T$.
Then, there is a horospherical Sarkisov program from $X/S$ to $Y/T$.
\end{teo}

What a horospherical Sarkisov program is will be made more precise in the last section where we state and prove Theorem \ref{thm:mainth}, which directly implies Theorem \ref{thm:mainthIntro}, but the rough idea is that if $X/S$ and $Y/T$ are outcomes of two MMP on the same $Z$, we consider the two-parameter family 
$$\{Q_{(1-\delta)D+\delta D'+\epsilon K_Z}\}_{(\delta,\epsilon)\in\mathbb Q^2}.$$
We consider the polyhedron $\Omega$ of dimension two of the $(\delta,\epsilon)$ such that 
$Q_{(1-\delta)D+\delta D'+\epsilon K_Z}$ is non-empty.
There is a polyhedral decomposition of $\Omega$ given by the points $(\delta,\epsilon)$
such that the corresponding moment polytope defines the same variety.
We show then that we can "read" the Sarkisov links on a part of the boundary of $\Omega$.

In the appendix are some examples of this horospherical Sarkisov program, showing how to effectively implement our methods.

\input{horovar}

\input{resolHMMP}

\section{Two-parameters families of polytopes}\label{sec:poly}

In this section, we study some two-parameters families of polytopes. This section can be read independently from the rest of the paper, except for subsection~\ref{linkwithHMMP}. 
%Note that Corollary~\ref{cor:mainpoly} is an essential tool in the proofs of Theorems~\ref{th:main0} and \ref{th:maingeneral}.

\subsection{A two-parameters family of polytopes: definitions and first properties}
Let $\mathbb K=\Qbb$ or $\Rbb$.
Let $n$ and $p$ be two positive integers. Denote by $I_0$ the set $\{1,\dots,p\}$.
Then two matrices $A\in M_{p\times n}(\Qbb)$ and $D\in M_{p\times 1}(\Qbb)$  define a polyhedron 
$$P:=\{x\in\mathbb K^n\,\mid\, Ax\geq D\}.$$
\begin{rem}\label{rem:bdd}
Suppose that $P$ is non-empty. Then $P$ is a polytope, that is a bounded polyhedron, if and only if the following condition is satisfied.
\begin{cond}\label{condBDD} There is no non-zero $x\in\Qbb^n$ satisfying $Ax\geq 0$. \end{cond}

The first implication is not difficult. Conversely, if  $P$ is not bounded, $P$ contains at least an affine half-line, and $x$ can be taken to be a generator of the direction of this half-line.

Note that Condition~\ref{condBDD} implies that $A$ is injective.
\end{rem}

To define a two-parameters family of polytopes, we fix $A$ satisfying Condition~\ref{condBDD} and we define divisors depending on two rational parameters $\delta$ and $\epsilon$.

Let $B$, $B'$ and $C$ in $M_{p\times 1}(\Qbb)$. Set $I_0:=\{1,\dots,p\}$ and define

$$
\begin{array}{rrcl}
 D\colon& \Qbb^2&\rightarrow&\Qbb^p\\
& (\delta,\epsilon)&\mapsto &(1-\delta)B+\delta B'+\epsilon C.
\end{array}
$$
 Note that $D$ is an affine map.

\begin{defi}

Given $A$, $B$, $B'$ and $C$ as above, we define for any $(\delta,\epsilon)\in\Qbb^2$:
$$P^{\delta,\epsilon}:=\{x\in\mathbb K^n\,\mid\, Ax\geq D(\delta,\epsilon)\}\cap \Qbb^n.$$
\end{defi}
We do not exclude the case where some lines of $A$ are zero. Notice that $P^{\delta,\epsilon}$ can be empty, even for all $(\delta,\epsilon)\in\Qbb^2$. 

Condition~\ref{condBDD} implies that for any $(\delta,\epsilon)\in\Qbb^2$,
the set $P^{\delta,\epsilon}$ is a polytope (possibly empty).

Now, we want to describe some equivalent classes of polytopes in this family, looking at their faces that correspond to lines of $A$. A face $F^{\delta,\epsilon}$ of $P^{\delta,\epsilon}$ is given by some equalities in $Ax\geq D(\delta,\epsilon)$ and then is associated to some $I\subseteq I_0$. We formalize this below.

For any matrix $\mathcal{M}$ and any $i\in I_0$, we denote by $\mathcal{M}_i$ the matrix consisting of the line $i$ of $\mathcal{M}$. More generally, for any subset $I$ of $I_0$ we denote by $\mathcal{M}_I$ the matrix consisting of the lines $i\in I$ of $\mathcal{M}$. For any subset $I$ of $I_0$, we can identify $D_I$ with the affine map
$$
\begin{array}{rrcl}
 D_I\colon& \Qbb^2&\rightarrow&\Qbb^{|I|}\\
& (\delta,\epsilon)&\mapsto &(1-\delta)B_I+\delta B'_I+\epsilon C_I.
\end{array}
$$

Let $(\delta,\epsilon)\in\Qbb^2$. Denote by $\mathcal{H}_i^{\delta,\epsilon}$ 
the hyperplane $\{x\in\Qbb^n\,\mid\, A_ix=D_i(\delta,\epsilon)\}$. 
For any $I\subseteq I_0$, denote by $F_I^{\delta,\epsilon}$ the face of $P^{\delta,\epsilon}$ defined by $$F_I^{\delta,\epsilon}:=(\bigcap_{i\in I}\mathcal{H}_i^{\delta,\epsilon})\cap P^{\delta,\epsilon}.$$ Note that for any face $F^{\delta,\epsilon}$ of $P^{\delta,\epsilon}$ there exists a unique maximal $I\subseteq I_0$ such that $F^{\delta,\epsilon}=F_I^{\delta,\epsilon}$ (we include the empty face and $P^{\delta,\epsilon}$ itself). \\

\begin{defi}\label{def:omega}
Let $I\subseteq I_0$. 
Define $\Omega_I$ to be the set of $(\delta,\epsilon)\in\Qbb^2$ such that $F_I^{\delta,\epsilon}$ is not empty. 
Define $\omega_I$ to be the subset of $\Omega_I$ such that, if $I'\subseteq I_0$ satisfies $F_I^{\delta,\epsilon}= F_{ I'}^{\delta,\epsilon}$, then $I'\subseteq I$.
\end{defi}

In other words, 
$$\Omega_I=\{(\delta,\epsilon)\vert\;\text{there is}\;x\in\Qbb^n\;\text{such that}\;A_Ix=D_I(\delta,\epsilon)\;\text{ and }\;Ax\geq D(\delta,\epsilon)\}$$
and
$$\omega_I=\{(\delta,\epsilon)\vert\;\text{there is}\;x\in\Qbb^n\;\text{such that}\;A_Ix=D_I(\delta,\epsilon)\;\text{ and }\;A_{\bar{I}}x>D_{\bar{I}}(\delta,\epsilon)\}.$$
%$(\delta,\epsilon)\in\Omega_I$ if and only if there exists $x\in\Qbb^n$ such that 
%$A_Ix=D_I(\delta,\epsilon)$ and $Ax\geq D(\delta,\epsilon)$.
%A pair $(\delta,\epsilon)$ belongs to $\omega_I$ if and only if there exists $x\in\Qbb^n%$ such that $A_Ix=D_I(\delta,\epsilon)$ and $A_{\bar{I}}x>D_{\bar{I}}(\delta,\epsilon)$, 
%where $\bar{I}:=I_0\backslash I$ denotes the complement of $I$ in $I_0$.

To simplify the notation, we often write $i$ instead of ${\{i\}}$, for any $i\in I_0$.

\begin{rem}
If $(\delta,\epsilon)\in\omega_{\emptyset}$, then the polytope $P^{\delta,\epsilon}$ is of dimension~$n$ (i.e. has a non-empty interior). And, for any $i\in I_0$, if $(\delta,\epsilon)\in\omega_i$ and $A_i\neq 0$, then $F_i^{\delta,\epsilon}$ is a facet of $P^{\delta,\epsilon}$.\\ 
\end{rem}

\subsection{A polyhedral partition of $\Qbb^2$}

The following lemma describes the first properties of the sets $\Omega_I$ and $\omega_I$.

\begin{lem}\label{lem:convopen}
Let $I\subseteq I_0$.
\begin{enumerate}
\item The sets $\Omega_I$ and $\omega_I$ are convex subsets of $\Qbb^2$.

\item The set $\omega_I$ is open, for the euclidean topology, inside $\{(\delta,\epsilon)\vert\; D_I(\delta,\epsilon)\in\Im A_I\}=D_I^{-1}\Im A_I$.

\item There are four cases: either $\omega_I$ is empty, or it is a point, or it is a convex part of an affine line (a segment, a half-line or a line), or it is a non-empty open set in $\Qbb^2$.

\item If $\omega_I$ is not empty, we have  $\omega_I\subset\Omega_I\subset\overline{\omega_I}$.

\end{enumerate}

\end{lem}

\begin{proof}
\begin{enumerate}
\item If $\Omega_I$ is either empty or has cardinality 1, the statement is obvious.
We assume then that $\Omega_I$ has at least two points, 
 $(\delta_1,\epsilon_1)$ and $(\delta_2,\epsilon_2)$. Then there exist $x_1$ and $x_2$ in $\Qbb^n$ such that $A_I x_i=D_I(\delta_i,\epsilon_i)$  
and $A_{\bar{I}}x_i\geq D_{\bar{I}}(\delta_i,\epsilon_i)$ for $i=1,2$.
For any rational number $t\in [0,1]$, we get 
$A_I(t x_1+(1-t)x_2)=D_I(t\epsilon_1+(1-t)\epsilon_2,t\delta_1+(1-t)\delta_2)$ 
and $A_{\bar{I}}(t x_1+(1-t)x_2)\geq D_{\bar{I}}(t\epsilon_1+(1-t)\epsilon_2,t\delta_1+(1-t)\delta_2)$, 
so that $t(\delta_1,\epsilon_1)+(1-t)(\delta_2,\epsilon_2)$ is in $\Omega_I$.

Replacing $\geq$ by $>$, we prove the convexity of $\omega_I$.

\item The inclusion  $\omega_I\subseteq D_I^{-1}\Im A_I$ follows from the definition of  $\omega_I$.

\noindent If $D_I^{-1}\Im A_I$ is a point or if $\omega_{I}$ is empty, there is nothing to prove.

\noindent Let $(\delta_0,\epsilon_0)\in \omega_I$.
Then there is $x_0$ such that $A_I x_0=D_I(\delta_0,\epsilon_0)$ and $A_{\bar I} x_0>D_{\bar I}(\delta_0,\epsilon_0)$.

Assume that $D_I^{-1}\Im A_I$ is one-dimensional. 
Let $(\delta_1,\epsilon_1)\neq (\delta_0,\epsilon_0)$ in $D_I^{-1}\Im A_I$  and let $x_1$ be such that $A_I x_1=D_I(\delta_1,\epsilon_1)$.
Then $A_I((1-t)x_0+tx_1)=D_I((1-t)\epsilon_0+t\epsilon_1,(1-t)\delta_0+t\delta_1)$ and for $t\in\Qbb$ small enough $A_{\bar{I}}((1-t)x_0+tx_1)>D_{\bar{I}}((1-t)\epsilon_0+t\epsilon_1,(1-t)\delta_0+t\delta_1)$.

Assume that $D_I^{-1}\Im A_I$ is two-dimensional. 
Let $(\delta_1,\epsilon_1)$ and $(\delta_2,\epsilon_2)$ in $D_I^{-1}\Im A_I$ such that the three  $(\delta_i,\epsilon_i)$'s are not in the same line.  Let $x_1$ and $x_2$ be such that $A_I x_1=D_I(\delta_1,\epsilon_1)$ and $A_I x_2=D_I(\delta_2,\epsilon_2)$.
Then $$A_I((1-t_1-t_2)x_0+t_1x_1+t_2x_2)=D_I((1-t_1-t_2)\epsilon_0+t_1\epsilon_1+t_2\epsilon_2,(1-t_1-t_2)\delta_0+t_1\delta_1+t_2\delta_2)$$ and for $(t_1,t_2)$ in a neighborhood of 0 in $\Qbb^2$, we have   $$A_{\bar{I}}((1-t_1-t_2)x_0+t_1x_1+t_2x_2)>D_{\bar{I}}((1-t_1-t_2)\epsilon_0+t_1\epsilon_1+t_2\epsilon_2,(1-t_1-t_2)\delta_0+t_1\delta_1+t_2\delta_2).$$

This concludes the proof of the statement.

\item It follows from the two previous statements.

\item The first inclusion is obvious. 

To prove the second one, remark that for all $(\delta_1,\epsilon_1)\in\omega_I$ and $(\delta_2,\epsilon_2)\in\Omega_I$, the segment 
$\{(1-t)(\delta_1,\epsilon_1)+t(\delta_2,\epsilon_2)|\;t\in[0,1)\}$ is contained in $\omega_I$. 
Indeed, let $x_1$ and $x_2$ be such that $A_Ix_i=D_I(\delta_i,\epsilon_i)$, $A_{\bar{I}}x_i> D_{\bar{I}}(\delta_i,\epsilon_i)$. 
Then, for any $t\in(0,1]$, we set $x_t:=t x_1+(1-t)x_2$ and $(\delta_t,\epsilon_t):=(t\epsilon_1+(1-t)\epsilon_2,t\delta_1+(1-t)\delta_2)$.
Thus $A_Ix_t=D_I(\delta_t,\epsilon_t)$ and $A_{\bar{I}}x_t> D_{\bar{I}}(\delta_t,\epsilon_t)$.

This remark implies directly that if $(\delta_2,\epsilon_2)\in\Omega_I$ then $(\delta_2,\epsilon_2)\in\overline{\omega_I}$, as soon as $\omega_I$ is not empty.
\end{enumerate}
\end{proof}

\begin{rem}

 Lemma \ref{lem:convopen} is still true if we consider polytopes in an $\Rbb$-vector space. The proof is the same after replacing $\Qbb$ by $\Rbb$ everywhere.
 
\end{rem}

\begin{lem}\label{lem:closed}
Let $I\subseteq I_0$ be such that $\omega_I$ is not empty. Then $\Omega_I=\overline{\omega_I}$ and $\Omega_I$ is polyhedral in $\Qbb^2$.
\end{lem}

\begin{proof}

If $\omega_I$ is reduced to a point, by Lemma~\ref{lem:convopen} we have nothing to prove, so we suppose that $\omega_{I}$ contains at least two points.\\

For $(\delta,\epsilon)\in\Rbb^2$ we also set
$$P^{\delta,\epsilon}:=\{x\in\Rbb^n\,\mid\, Ax\geq (1-\delta)B+\delta B'+\epsilon C\}.$$
and we define as before $F_I^{\delta,\epsilon}$.

For $I\subseteq I_0$ we set 
$$\Omega_I(\Rbb)=\{(\delta,\epsilon)\in\Rbb^2\,\mid\,F_I^{\delta,\epsilon}\neq\emptyset\}$$ 
and
$$\omega_I(\Rbb)=\{(\delta,\epsilon)\in\Rbb^2\,\mid\,\text{if }I'\subseteq I_0\text{ satisfies }F_I^{\delta,\epsilon}= F_{ I'}^{\delta,\epsilon}\text{, then }I'\subseteq I\}.$$

Since the matrices $A$, $B$, $B'$ and $C$ have rational coefficients, if $(\delta,\epsilon)\in\Qbb^2$ and 
if there exists $x\in\Rbb^n$, such that $A_Ix=D_I(\delta,\epsilon)$, then there is $x'\in\Qbb^n$, such that $A_Ix'=D_I(\delta,\epsilon)$; and $x'$ can be chosen arbitrarily close to $x$. 
In particular, if $A_{\bar{I}}x>D_{\bar{I}}(\delta,\epsilon)$, then  $x'$ can be chosen such that $A_{\bar{I}}x'>D_{\bar{I}}(\delta,\epsilon)$. 
Hence, $\omega_I=\omega_I(\Rbb)\cap\Qbb^2$. 

Moreover, $\Omega_I=\bigcup_{I\subseteq I'\subseteq I_0}\omega_{I'}$ (over $\Rbb$ and $\Qbb$), then we also have $\Omega_I(\Qbb)=\Omega_I(\Rbb)\cap\Qbb^2$.\\

We first prove that $\Omega_I(\Rbb)$ is closed.
This, together with Lemma \ref{lem:convopen} will imply that $\Omega_I(\Rbb)=\overline{\omega_I(\Rbb)}$. 

Let $(\delta_k,\epsilon_k)_{k\in\Nbb}$ be a sequence of elements in $\Omega_I(\Rbb)$ converging to $(\bar\delta,\bar\epsilon)$ in $\Rbb^2$.
The elements  $(\delta_k,\epsilon_k)$ are contained in a compact set $K$ of $\Rbb^2$. Then, for every $k\in\Nbb$, the polytope $P^{\delta_k,\epsilon_k}$ is contained in the polytope $$P^K:=\{x\in\Rbb^n\,\mid\,Ax\geq\Min_{(\delta,\epsilon)\in K}D(\delta,\epsilon)\}$$
where $\Min_{(\delta,\epsilon)\in K}D(\delta,\epsilon)$ is the vector whose $i$-th coordinate is $\Min_{(\delta,\epsilon)\in K}D(\delta,\epsilon)_i$.
The set $P_K$ is compact by Remark \ref{rem:bdd}.

By definition, for any $k\in\Nbb$, there exists $x_k\in \Rbb^n$ such that $A_Ix_k=D_I(\delta_k,\epsilon_k)$ and $A_{\bar{I}}x_k\geq D_{\bar{I}}(\delta_k,\epsilon_k)$. 
Since the $x_k\in P^K$ that is compact, 
there is $\{k_m\}_{m\in\Nbb}$ such that $k_m\to\infty$ as $m$ tends to infinity, such that $(x_{k_m})_{m\in\Nbb}$ converges to $\bar x\in\Rbb$. Then $\bar x$ satisfies $A_I\bar x=D_I(\bar\delta,\bar\epsilon)$ 
and $A_{\bar{I}}\bar x\geq D_{\bar{I}}(\bar\delta,\bar\epsilon)$. 
This implies that $(\bar\delta,\bar\epsilon)$ is also in $\Omega_I$.\\

We now prove that $\Omega_I(\Rbb)$ is polyhedral. If $\omega_I(\Rbb)$ is either empty, or a point, or a convex part of an affine line, then there is nothing to prove. 
Suppose that $\omega_I(\Rbb)$ is open in $\Rbb^2$. Then the boundary of $\Omega_I(\Rbb)$ is $$\Omega_I(\Rbb)\backslash\omega_I(\Rbb)=\bigcup_{I\subsetneq I'\subseteq I_0}\omega_{I'}(\Rbb)=\bigcup_{I\subsetneq I'\subseteq I_0}\Omega_{I'}(\Rbb).$$ 
But for every $I'$ such that $I\subsetneq I'\subseteq I_0$, the set $\Omega_{I'}(\Rbb)$ is either empty, or a point, or a closed convex part of an affine line. 
We proved that the boundary of $\Omega_I(\Rbb)$ is a finite union of closed convex parts of affine lines.
Hence, $\Omega_I(\Rbb)$ is polyhedral in $\Rbb^2$.\\

To conclude, we have to prove that the vertices of $\Omega_I(\Rbb)$ are rational and that the maximal half lines in the boundary of $\Omega_I(\Rbb)$ are rational half-lines, that is, they have a rational extremity and rational direction. 

A vertex of $\Omega_I(\Rbb)$ is $(\delta,\epsilon)$ such that there is $I'\neq I$ with $\{(\delta,\epsilon)\}=\omega_{I'}(\Rbb)$. 
By Lemma~\ref{lem:convopen}
we have $\{(\delta,\epsilon)\}=D_{I'}^{-1}\Im A_I\in\Qbb^2$.

In the same way, a component of the boundary is of the form 
$\omega_{I'}(\Rbb)$ with $I'\neq I$.
By Lemma \ref{lem:convopen}, we have $\omega_{I'}(\Rbb)\subseteq D_{I'}^{-1}\Im A_I $ open.
The latter is an affine subspace defined over $\Qbb$, thus its direction is rational.
The extremities are vertices of $\Omega_{I}(\Rbb)$ and are then rational. 
\end{proof}

We now study inclusions between the $\omega$'s, generalizing the easy following fact: for any $J\subseteq I\subseteq I_0$, we have $\Omega_I\subseteq \Omega_J$.

\begin{lem}\label{lem:inclusions}
Let $J\subseteq I\subseteq I_0$, such that $\omega_I\neq \emptyset$.
\begin{enumerate}
\item If $\omega_J$ is not empty and has the same dimension of $\omega_I$, then $\omega_I\subseteq \omega_J$.
\item If the images of $A_I$ and $A_J$ have the same codimension $k$, then  $\omega_J$ is not empty.
\item There exists $J\subseteq I$ as above and such that $\omega_I\subseteq \omega_J$ and for any $j\in J$ the codimension of the image of $A_{J\backslash\{j\}}$ is less than $k$.
\end{enumerate}
\end{lem}

\begin{proof}
\begin{enumerate}
\item If $\omega_J$ is not empty, Lemma~\ref{lem:closed} implies that $\omega_J$ is the interior part $\Omega_J$. Since $\omega_I$ is not empty, it is as well the interior part of $\Omega_I$. As $\Omega_I\subseteq \Omega_J$, we conclude that $\omega_I\subseteq \omega_J$.
\item We prove the claim by induction on $h=|I\setminus J|$.
If $h=0$ the claim is true. 
We assume now that the claim is true for $h$ and let $J$ be such that $|I\setminus J|=h+1$.
Let $i\in I\backslash J$ and $J'=J\cup\{i\}$. 
By inductive hypothesis $\omega_{J'}$ is not empty.
Let $(\delta,\epsilon)\in\omega_{J'}$ and let $x\in\Qbb^n$ such that $A_{J'}x=D_{J'}(\delta,\epsilon)$ and $A_{\bar{{J'}}}x>D_{\bar{{J'}}}(\delta,\epsilon)$. 
By the hypothesis on the images of $A_{J'}$ and $A_J$, there is 
$y\in\Ker(A_{J})\backslash\Ker(A_{J'})$. We can choose $y$ such that $A_iy>0$. 
Then for $t>0$ small enough, we have $A_{J}(x+ty)=D_{J}(\delta,\epsilon)$, $A_i(x+ty)>D_i(\delta,\epsilon)$ and $A_{\bar{J'}}(x+ty)>D_{\bar{J'}}(\delta,\epsilon)$. 
This proves that $(\delta,\epsilon)\in\omega_{J}$.
\item It is enough to take $J$ minimal such that the images of $A_I$ and $A_J$ have the same codimension, and apply the previous statement. 
\end{enumerate}
\end{proof}

\begin{lem}\label{lem:inclusionclosure}
 Let $I\subseteq I_0$ be such that $\omega_I=\{(\delta_0,\epsilon_0)\}$. Suppose that the image of $A_I$ is of codimension~2.
There is $i\in I$ such that the image of $A_{I\backslash \{i\}}$ is of codimension~1. 
For any such $i$, the point $(\delta_0,\epsilon_0)$ belongs to $\Omega_{I\backslash \{i\}}\backslash \omega_{I\backslash \{i\}}$.

%\ora{ deuxi\`eme partie du lemme enlev\'e et mis en remarque apr\`es sans preuve}

\end{lem}

\begin{proof}
%\begin{enumerate}
%\item
 Let $i\in I$ be such that there exists rational numbers $\lambda_j$ for $j\in I\backslash\{i\}$ such that $A_i=\sum_{j\in I\backslash\{i\}}\lambda_jA_j$.  Then the image of $A_{I\backslash \{i\}}$ is of codimension one.
Let $(\delta_1,\epsilon_1)\neq(\delta_0,\epsilon_0)$ such that $D_{I\backslash \{i\}}(\delta_1,\epsilon_1)$ is in $\Im(A_{I\backslash \{i\}})$. 
Let $x$ and $y$ in $\Qbb^n$ such that $A_Ix=D_I(\delta_0,\epsilon_0)$, $A_{\bar{I}}x>D_{\bar{I}}(\delta_0,\epsilon_0)$ and $A_{I\backslash\{i\}}y=D_{I\backslash\{i\}}(\delta_1,\epsilon_1)$. For any $t\in\Qbb$, define $z_t:=(1-t)x+ty$ and $(\delta_t,\epsilon_t):=(1-t)(\delta_0,\epsilon_0)+t(\delta_1,\epsilon_1)$. Then for any $t$ we have $A_{I\backslash\{i\}}z_t=D_{I\backslash\{i\}}(\delta_t,\epsilon_t)$ and for any $t$ small enough we have $A_{\bar{I}}z_t>D_{\bar{I}}(\delta_t,\epsilon_t)$. Now if $A_iy>0$ we have $A_iz_t>D_i(\delta_t,\epsilon_t)$ for any $t>0$; and if $A_iy<0$ we have $A_iz_t>D_i(\delta_t,\epsilon_t)$ for any $t<0$.
Hence, for any $t$ small enough, either positive or negative, we have $A_{\bar{I\backslash\{i\}}}z_t>D_{\bar{I\backslash\{i\}}}(\delta_t,\epsilon_t)$ and $(\delta_t,\epsilon_t)$ is in $\omega_{I\backslash \{i\}}$. We have proved that $\omega_I$ is in the closure of $\omega_{I\backslash \{i\}}$.

 But $\omega_I=\{(\delta_0,\epsilon_0)\}$ cannot be in $\omega_{I\backslash \{i\}}$, because the relation $A_i=\sum_{j\in I\backslash\{i\}}\lambda_jA_j$ implies that if $A_{I\backslash\{i\}}x=D_{I\backslash\{i\}}(\delta_0,\epsilon_0)$ then $A_Ix=D_I(\delta_0,\epsilon_0)$. Hence $\omega_I$ is in  $\Omega_{I\backslash \{i\}}\backslash \omega_{I\backslash \{i\}}$.

\end{proof}

\begin{rem} A similar result could be proved from a one-dimensional $\omega_I$. Let $I\subset I_0$ such that $\omega_I$ is an open convex part of an affine line. Suppose that the image of $A_I$ is of codimension~1.
Then there exists $i\in I$ such that $\omega_I$ is a subset of $\Omega_{I\backslash \{i\}}\backslash \omega_{I\backslash \{i\}}$  and $A_{I\backslash \{i\}}$ is surjective.
\end{rem}

\begin{lem}\label{lem:intersectionIJ}
Let $I$ and $J$ such that $\omega_I$ and $\omega_J$ are not empty. Then $\omega_{I\cap J}$ is not empty and contains the strict convex hull of any element of $\omega_I$ with any element of $\omega_J$.
\end{lem}

\begin{proof}
Let $x$ and $y$ in $\Qbb^n$ be such that $A_Ix=D_I(\delta_1,\epsilon_1)$, $A_{\bar{I}}x>D_{\bar{I}}(\delta_1,\epsilon_1)$, $A_Jx=D_J(\delta_2,\epsilon_2)$ and $A_{\bar{J}}x>D_{\bar{J}}(\delta_2,\epsilon_2)$.\\
For any $t\in]0,1[$, $z_t:=(1-t)x+ty$ and $(\delta_t,\epsilon_t):=(1-t)(\delta_1,\epsilon_1)+t(\delta_2,\epsilon_2)$ satisfy $A_{I\cap J}z_t=D_{I\cap J}(\delta_t,\epsilon_t)$, $A_{\overline{I\cap J}}z_t>D_{\overline{I\cap J}}(\delta_t,\epsilon_t)$. Hence, $\omega_{I\cap J}$ contains the strict convex hull of $(\delta_1,\epsilon_1)$ and $(\delta_2,\epsilon_2)$. In particular, $\omega_{I\cap J}$ is not empty.
\end{proof}

\begin{rem}
\begin{itemize}
If $I=\{i\}$ and $A_i\neq 0$, or more generally if $A_I$ is surjective, Lemma~\ref{lem:convopen} implies that $\omega_{I}$ is open (possibly empty). 

\noindent If $I=\{i\}$ and $A_i\neq 0$, for any $(\delta,\epsilon)\in\Omega_I\backslash\omega_I$, either the polytope $P^{\delta,\epsilon}$ can be defined without the line $i$, in other words
$$P^{\delta,\epsilon}=\{x\in\Qbb^n\,\mid\,A_{\overline{\{i\}}}x\geq D_{\overline{\{i\}}}(\delta,\epsilon)\},$$ or the polytope $P^{\delta,\epsilon}$ is not of dimension $n$. 
Indeed, if $(\delta,\epsilon)\in\Omega_{I}\backslash\omega_{I}$, and $P^{\delta,\epsilon}$ has maximal dimension, then $F^{\delta,\epsilon}_i$ is not a facet of $P^{\delta,\epsilon}$ or equals  another $F^{\delta,\epsilon}_j$ and in both cases the inequality $A_i x\geq D_i(\delta,\epsilon)$ is superfluous in the definition of $P^{\delta,\epsilon}$.

\noindent If $I$ is such that $A_I$ is invertible, then $F_I^{\delta,\epsilon}$ is either empty or a vertex of $P^{\delta,\epsilon}$. If $(\delta,\epsilon)\in\Omega_I$,
then $P^{\delta,\epsilon}$ is simple at the corresponding vertex if $(\delta,\epsilon)\in\omega_I$. But the converse is false: it can happen that $P^{\delta,\epsilon}$ is simple but $(\delta,\epsilon)\not\in\omega_I$ if an inequality $A_i x\geq D_i(\delta,\epsilon)$ with $i\not\in I$ is superfluous in the definition of $P^{\delta,\epsilon}$ and $F_I^{\delta,\epsilon}=F_{I\cup\{i\}}^{\delta,\epsilon}$.
\end{itemize}

\end{rem}

\subsection{Generality of the polarization}\label{ssec:gen}
In this section we prove some preparatory results.
Throughout the section by \textit{general} we mean \textit{in a Zariski open set}.
\begin{lem}\label{lem:generalB}
We can choose $B$ and $B'$ general such that for any $I\subseteq I_0$, we have the following.
\begin{itemize}
\item If the image of $A_I$ has codimension 1 in $\Qbb^I$, then $\omega_I$ is an open convex part of an affine line (possibly empty).
\item If the image of $A_I$ has codimension 2 in $\Qbb^I$, then $\omega_I$ is either empty or a point.
\item If the image of $A_I$ has codimension 3 in $\Qbb^I$, then $\omega_I$ is empty.
\end{itemize}
\end{lem}

\begin{proof}

The set $\{D_I(\delta,\epsilon)\,\mid\,(\delta,\epsilon)\in\Qbb^2\}$ is the affine subspace passing through $B_I$ and directed by $B_I'-B_I$ and $C_I$. It is a plane for  $B$ and $B'$ general (that is, if $B'-B$ is not colinear to $C$). Now, to have the three conditions above, it is enough to choose $B$ and $B'$ such that:
\begin{itemize}
\item if the image of $A_I$  has codimension 1 in $\Qbb^I$, then $B_I$ and $B'_I$ are not in the image of $A_I$;
\item if the image of $A_I$  has codimension 2 in $\Qbb^I$, then $B'_I-B_I$ is not in the vector subspace $\Im(A_I)+\Qbb C_I$ (of codimension at least one in $\Qbb^I$), and if $C_I$ is in $\Im(A_I)$ then  $B_I$ is not in $\Im(A_I)$;
\item  if the image of $A_I$  has codimension 3 in $\Qbb^I$, $B_I'-B_I$ is not in $\Im(A_I)$ and $B_I$ is not in the vector subspace $\Im(A_I)+\Qbb C_I+\Qbb (B_I'-B_I)$ (of codimension at least one in $\Qbb^I$).
\end{itemize}
This, together with Lemma~\ref{lem:convopen}, means that it is enough to choose $B$ and $B'$ outside finitely many proper linear subspaces of $\Qbb^p$, thus $B$ and $B'$ in an open set of $\mathbb Q^p$.
\end{proof}

From now, we assume that $B$ and $B'$ general in the sense of Lemma \ref{lem:generalB}.

\begin{cor}\label{cor:generalB}
 There is a finite union  of convex parts of affine lines $$L=\partial\Omega_{\emptyset}\,\cup\bigcup_{I\subseteq I_0\,\dim \Omega_I\leq 1} \Omega_I\subseteq\Qbb^2$$ 
 such that if $(\delta,\epsilon)\not\in L$ then $P^{\delta,\epsilon}$ is an $n$-dimensional simple polytope or it is empty.
%For general $B$ and $B'$, for general $(\delta,\epsilon)\in\Qbb^2$, the polytope $P^{\epsilon,\delta}$ is empty or simple. (also simple as a ``colored polytope'')
\end{cor}

\begin{proof}
Assume that $P^{\delta,\epsilon}$ is not empty.
If $P^{\delta,\epsilon}$ has  of dimension~$n$ and is not simple, then there exists a vertex $v$ of $P^{\delta,\epsilon}$ 
that is contained in at least $n+1$ facets. Then there is $I\subseteq I_0$ of cardinality $p\geq n+1$ such that the vertex is $F_I^{\delta,\epsilon}=\{v\}$. 
In particular, if we choose $I$ maximal among the subsets $I$ of $I_0$ such that $F_I^{\delta,\epsilon}=\{v\}$, 
we have $(\delta,\epsilon)\in\omega_I$. But $A_I$ is a $p\times n$ matrix with $p\geq n+1$ so it cannot be surjective. By Lemma~\ref{lem:generalB}, $\omega_I$ is either a segment or a point. 

If $P^{\delta,\epsilon}$ is of dimension at most $n$, then $(\delta,\epsilon)\not\in\omega_{\emptyset}$, which is open and non-empty by hypothesis. 
Then $(\delta,\epsilon)$ is in the boundary of $\Omega_{\emptyset}$, which is a finite union of convex parts of affine lines.

We set then $$L=\partial\Omega_{\emptyset}\,\cup\bigcup_{I\subseteq I_0\,\dim \Omega_I\leq 1} \Omega_I.$$
\end{proof}

\begin{rem} 
Notice that the condition we impose on $\Omega_{\emptyset}\setminus L$ is stronger than the polytope $P^{\delta,\epsilon}$ being simple.
Indeed, the vertices of $P^{\delta,\epsilon}$ for $(\delta,\epsilon)\in\Omega_{\emptyset}\setminus L $
are contained in exactly $n$ affine hyperplanes $\mathcal{H}_i:=\{x\in\Qbb^n\,\mid\,A_ix=D_i(\delta,\epsilon)\}$,
while a vertex of an affine polytope is contained in exactly $n$ facets.
%In this proof, to construct $L$ we do not use that a vertex is in at least $n+1$ facets but in at least $n+1$ affine hyperplane $\mathcal{H}_i:=\{x\in\Qbb^n\,\mid\,A_ix=D_i(\delta,\epsilon)\}$. Then $L$ could be too big.
\end{rem}

Now we prove that if $B$ and $B'$ are general, two sets $\omega_I$ and $\omega_J$ 
of dimension 0 or 1 intersect or are aligned only in specific cases.

\begin{notc}\label{def:alpha}
Let $I\subseteq I_0$ be such that  
 the image of $A_I$ is contained in a hyperplane, and such that for any $i\in I$, $A_{I\backslash\{i\}}$ is surjective. 
Then there are rational numbers $\lambda_i^I$ with $i\in I$ such that $\sum_{i\in I}\lambda_i^I A_i=0$. 
We notice that  $\lambda_i^I\neq 0$ for every $i\in I$. Indeed, if there is $j$ such that $\lambda_j^I= 0$, then $\sum_{i\in I\setminus\{j\}}\lambda_i^I X_i$ is a nontrivial equation for the lines of $A_{I\backslash\{j\}}$.
We can assume that $\sum_{i\in I}\lambda_i^IC_i$ is zero or one. We fix such numbers.

\smallskip

Let $I\subseteq I_0$ such that $\omega_I$ is not empty, the image of $A_I$ has codimension at least two and such that for any $i\in I$, 
the image of $A_{I\backslash\{i\}}$ has codimension one. In particular, the image of $A_I$ has codimension exactly two.
Then there are rational numbers $\lambda_i^I$ and $\lambda_i^{\prime I}$ with $i\in I$ such that $\sum_{i\in I}\lambda_i^I A_i=0$ and $\sum_{i\in I}\lambda_i^{\prime I} A_i=0$ are two independent relations. 
For $B$ and $B'$ general  as in Lemma~\ref{lem:generalB}, the point in $\omega_{I}$ is the only solution of the linear system 
$$
\left\lbrace
\begin{array}{ccc}
(\sum_{i\in I}\lambda_i^I C_i) \epsilon+(\sum_{i\in I}\lambda_i^I(B'_i-B_i))\delta+\sum_{i\in I}\lambda_i^IB_i&=&0\\
(\sum_{i\in I}\lambda_i^{\prime I} C_i) \epsilon+(\sum_{i\in I}\lambda_i^{\prime I}(B'_i-B_i))\delta+\sum_{i\in I}\lambda_i^{\prime I}B_i&=&0
\end{array}
\right.
$$
therefore $\sum_{i\in I}\lambda_i^IC_i$ and $\sum_{i\in I}\lambda_i^{\prime I}C_i$ cannot be simultaneously zero. After perhaps replacing $\lambda_i^{\prime I}$ with $\lambda_i^{\prime I}-\lambda_i^{I}$, we can assume that $\sum_{i\in I}\lambda_i^IC_i=1$ and $\sum_{i\in I}\lambda_i^{\prime I}C_i=0$.

Moreover, for any $i\in I$, either $\lambda_i^I$ or $\lambda_i^{\prime I}$ is not zero. Indeed, if there is $j$ such that $\lambda_j^I= 0$ and $\lambda_j^{\prime I}=0$, 
then $\sum_{i\in I\setminus\{j\}}\lambda_i^I X_i$ and $\sum_{i\in I\setminus\{j\}}\lambda_i^{\prime I} X_i$ are nontrivial linearly independent equations for the lines of $A_{I\backslash\{j\}}$.

We fix such numbers.
\end{notc}

\begin{rem}\label{rem:surjminusj} 
Let $I\subseteq I_1$ be such that
 $\Im A_I$ and $\Im A_{I_1}$ have codimension 1 and are defined by the same equation $\sum_{i\in I} \lambda_i X_i=0$. Assume moreover that $I=\{i\vert\;\lambda_i\neq 0\}$.
Then for every $j\in I$ the morphism induced by $A_{I_1\setminus\{j\}}$ is surjective.
Indeed, consider the projection $\pi\colon \mathbb Q^{I_1}\to\mathbb Q^{I_1\setminus\{j\}}$. The image of $A_{I_1\setminus\{j\}}$ is the projection of the hyperplane defined by $\sum_{i\in I} \lambda_i X_i=0$.
 Therefore $A_{I_1\setminus\{j\}}$ is surjective.
\end{rem}

\begin{rem}\label{rem:neq}
 Let $I\neq J$ be such that the images of $A_I$ and $A_J$ have codimension $d$ and  for any $i\in I$ and $j\in J$, the images of $A_{I\backslash\{i\}}$ and $A_{J\backslash\{j\}}$
 have codimension $d-1$. Then $I\not\subseteq J$ and $J\not\subseteq I$.
 Indeed if we had $I\subsetneq J$, there would be $j\in J\setminus I$ and $\Im (A_I)$ would have codimension at most the codimension of  $\Im (A_{J\backslash\{j\}})$.
\end{rem}

\begin{lem}\label{lem:segmentsgeneraux}
Let $I\neq J$ be such that $\omega_I$ and $\omega_J$ are  non-empty open convex parts of the same affine line, 
and such that for any $i\in I$ and $j\in J$, $A_{I\backslash\{i\}}$ and $A_{J\backslash\{j\}}$ are surjective. 
Then $B$ and $B'$ satisfy a quadratic or linear condition.
\end{lem}

\begin{proof}
Set $B'':=B'-B$. By remark \ref{rem:neq} we have $I\not\subseteq J$ and $J\not\subseteq I$.
Since $\omega_I$ and $\omega_J$ are non-empty open convex parts of affine lines, the images of $A_I$ and $A_J$ have codimension at least one.
Since $A_{I\backslash\{i\}}$ and $A_{J\backslash\{j\}}$ are surjective, they have codimension exactly one.

The affine line containing $\omega_{I}$ (resp. $\omega_{J}$) has as equation $\sum_{i\in I}\lambda_i^ID_i(\delta,\epsilon)=0$ (resp. $\sum_{i\in I}\lambda_j^J D_j(\delta,\epsilon)=0$), \textit{i.e.} 
$$(\sum_{i\in I}\lambda^I_iC_i)\epsilon+(\sum_{i\in I}\lambda^I_iB''_i)\delta+\sum_{i\in I}\lambda^I_iB_i=0,\;\text{resp.}\;(\sum_{j\in J}\lambda^J_jC_j)\epsilon+(\sum_{j\in J}\lambda^J_jB''_j)\delta+\sum_{j\in J}\lambda^J_jB_j=0.$$ 
Those two equations by hypothesis define the same line.
It follows from Construction \ref{def:alpha} that the coefficients of $\epsilon$ are either both zero or both one.

\medskip

 If $\sum_{i\in I}\lambda^I_iC_i=\sum_{j\in J}\lambda^J_jC_j=0$, since the two equations define the same line, 
 we have $$(\sum_{i\in I}\lambda^I_iB''_i)(\sum_{j\in J}\lambda^J_jB_j)=(\sum_{j\in J}\lambda^J_jB''_j)(\sum_{i\in I}\lambda^I_iB_i).$$ 
 This condition is nontrivial: let $i\in I$ and $j\in J$ such that $i\not\in J$ and $j\not\in I$, then the coefficient in the quadratic condition in $B_i''B_j$ is $\lambda_i^I\lambda_j^J$ and this is non zero by Construction \ref{def:alpha}.

\medskip

If $\sum_{i\in I}\lambda^I_iC_i=\sum_{j\in J}\lambda^J_jC_j=1$, then $\sum_{i\in I}\lambda^I_iB''_i=\sum_{j\in J}\lambda^J_jB''_j$ and $\sum_{i\in I}\lambda^I_iB_i=\sum_{j\in J}\lambda^J_jB_j$.
\end{proof}

\begin{cor}\label{cor:segmentsgeneraux}
Assume that $B$ and $B'$ are general. Then for any $I$ and $J$ subsets of $I_0$ such that $\omega_I$ and $\omega_J$ are nonempty open convex parts of affine lines,
$\omega_I$ and $\omega_J$ are contained in the same  
affine line  if and only if $\omega_{I\cap J}$ is a non-empty open convex part of an affine line.
Moreover in this case, we have $\omega_{I\cap J}\supseteq \omega_I\cup\omega_J$.
\end{cor}

\begin{proof}
Let $I$ and $J$ be subsets of $I_0$ such that $\omega_I$ and $\omega_J$ are non-empty open convex parts of the same affine line $L$. 
By Lemma~\ref{lem:inclusions}, there exist $I'\subseteq I$ and $J'\subseteq J$ such that $\omega_{I'}$ and $\omega_{J'}$ are non-empty open segments contained in $L$, 
and such that for every $i\in I'$ and $j\in J'$, $A_{I'\backslash\{i\}}$ and $A_{J'\backslash\{j\}}$ are surjective. 
By Lemma~\ref{lem:segmentsgeneraux}, for $B$ and $B'$ general, we must have $I'=J'$. 
In particular $I\cap J\supset I'$ and $\omega_I\subset\omega_{I\cap J}\subseteq \omega_{I'}$ so that $\omega_{I\cap J}$ is a non-empty open convex part of an affine line.

Conversely, if $\omega_I$, $\omega_J$ and $\omega_{I\cap J}$ are non-empty open convex part of an affine line,
then they are contained in the same affine line because $\omega_I\subset\omega_{I\cap J}\supset\omega_J$.
\end{proof}

\begin{lem}\label{lem:pointsgeneraux}
Let $I\neq J$ be such that $\omega_I=\omega_J$ is reduced to a point, and such that for any $i\in I$ and $j\in J$, 
the images of $A_{I\backslash\{i\}}$ and $A_{J\backslash\{j\}}$ have codimension one. 
Then $B$ and $B'$ satisfy  quadradic conditions.

\end{lem}

\begin{proof}
Set $B'':=B'-B$. By remark \ref{rem:neq} we have $I\not\subseteq J$ and $J\not\subseteq I$.
Since $\omega_I$ and $\omega_J$ are points, the images of $A_I$ and $A_J$ have codimension at least two.
Since $A_{I\backslash\{i\}}$ and $A_{J\backslash\{j\}}$ have codimension exactly one, the images of $A_I$ and $A_J$ have codimension exactly two.

Let $(\delta_0,\epsilon_0)$ be such that $\omega_I=\omega_J=\{(\delta_0,\epsilon_0)\}$. Then $(\delta_0,\epsilon_0)$ is the unique solution of the two following systems:

$$(\mathcal{S}_I)\quad\left\lbrace\begin{array}{ccc}
\sum_{i\in I}\lambda_i^ID_i(\delta,\epsilon)&=&0\\
\sum_{i\in I}\lambda_i^{\prime I}D_i(\delta,\epsilon)&=&0
\end{array}\right.\quad\mbox{and}\quad (\mathcal{S}_J)\quad\left\lbrace\begin{array}{ccc}
\sum_{j\in J}\lambda_j^JD_j(\delta,\epsilon)&=&0\\
\sum_{j\in J}\lambda_j^{\prime J}D_j(\delta,\epsilon)&=&0
\end{array}\right..$$

By Construction~\ref{def:alpha}, we can write these two systems as follows:
$$\left\lbrace\begin{array}{ccc}
\epsilon+(\sum_{i\in I}\lambda_i^IB''_i)\delta+\sum_{i\in I}\lambda_i^IB_i&=&0\\
(\sum_{i\in I}\lambda_i^{\prime I}B''_i)\delta+\sum_{i\in I}\lambda_i^{\prime I}B_i&=&0
\end{array}\right.\quad\mbox{and}\quad \left\lbrace\begin{array}{ccc}
\epsilon+(\sum_{j\in J}\lambda_j^JB''_j)\delta+\sum_{j\in J}\lambda_j^JB_j&=&0\\
(\sum_{j\in J}\lambda_j^{\prime J}B''_j)\delta+\sum_{j\in J}\lambda_j^{\prime J}B_j&=&0
\end{array}\right..$$

The equations in the second line are multiple one of the other, thus $$(\mathcal{C}_1):\,(\sum_{i\in I}\lambda^{'I}_iB''_i)(\sum_{j\in J}\lambda^{'J}_jB_j)=(\sum_{j\in J}\lambda^{'J}_jB''_j)(\sum_{i\in I}\lambda^{'I}_iB_i).$$

The equations in the first line imply $\left((\sum_{i\in I}\lambda_i^IB''_i)-(\sum_{j\in J}\lambda_j^JB''_j)\right)\delta+\sum_{i\in I}\lambda_i^IB_i-\sum_{j\in J}\lambda_j^JB_j=0$.
This equation is a multiple of $(\sum_{i\in I}\lambda_i^{\prime I}B''_i)\delta+\sum_{i\in I}\lambda_i^{\prime I}B_i=0$, therefore

$$(\mathcal{C}_2):\,\left((\sum_{i\in I}\lambda_i^IB''_i)-(\sum_{j\in J}\lambda_j^JB''_j)\right)\sum_{i\in I}\lambda_i^{\prime I}B_i-\left(\sum_{i\in I}\lambda_i^IB_i-\sum_{j\in J}\lambda_j^JB_j\right)\sum_{i\in I}\lambda_i^{\prime I}B''_i=0.$$

To conclude, it is enough to prove that the conditions $(\mathcal{C}_1)$ and $(\mathcal{C}_2)$ cannot be both trivial. 
Suppose that $(\mathcal{C}_1)$ is trivial. For any $j\in J$ such that $j\not\in I$ and for any $i\in I$, $\lambda_j^{\prime J}\lambda_i^{\prime I}=0$, 
but $(\lambda_i^{\prime I})_{i\in I}$ is not zero, so that $\lambda_j^{\prime J}=0$. 
 The condition $\lambda_j^{\prime J}=0$ implies that $\lambda_j^J\neq 0$, 
and then for any $i\in I$ the coefficient of $B_jB_i''$ in $(\mathcal{C}_2)$ is $\lambda_j^J\lambda_i^{\prime I}$, but $(\lambda_i^{\prime I})_{i\in I}$ is not zero, 
so that there exists $i\in I$ such that $\lambda_j^J\lambda_i^{\prime I}\neq 0$, and $(\mathcal{C}_2)$ is not trivial.
\end{proof}

\begin{cor}\label{cor:pointsgeneraux}
For $B$ and $B'$ general we have the following: for any $I,J\subseteq I_0$ such that $\omega_I$ and $\omega_J$ are singletons, we have
$\omega_I=\omega_J$ if and only if $\omega_{I\cap J}$ is a singleton, equal to $\omega_I=\omega_J$.
\end{cor}

\begin{proof}
Let $I$ and $J$ be subsets of $I_0$ such that $\omega_I=\omega_J=\{(\delta_0,\epsilon_0)\}$. 
By Lemma~\ref{lem:inclusions}, there exist $I'\subseteq I$ and $J'\subseteq J$ such that $\omega_{I'}=\omega_{J'}$ and for any $i\in I'$ and $j\in J'$, 
the images of $A_{I'\backslash\{i\}}$ and $A_{J'\backslash\{j\}}$ have codimension one. 
By Lemma~\ref{lem:pointsgeneraux}, for general $B$ and $B'$, we must have $I'=J'$. 
In particular $I\cap J\supset I'$ and $\omega_{I\cap J}\subseteq \omega_{I'}=\{(\delta_0,\epsilon_0)\}$. 
By Lemma \ref{lem:intersectionIJ} the set $\omega_{I\cap J}$ is not empty, then it must be $\omega_{I\cap J}=\{(\delta_0,\epsilon_0)\}$.

Conversely, if $\omega_I$, $\omega_J$ and $\omega_{I\cap J}$ are singletons then $\omega_I=\omega_{I\cap J}=\omega_J$
because $\omega_I\subset\omega_{I\cap J}\supset\omega_J$.
\end{proof}

\begin{lem}\label{lem:pointsegmentgeneraux}
Let $I$ and $J$ be such that $\omega_I=\{(\delta_0,\epsilon_0)\}$, $\omega_J$ is an open convex part of an affine line, and $\omega_I$ is contained in this affine line. 
Suppose that for every $i\in I$, the image of $A_{I\backslash\{i\}}$ has codimension one and for every $j\in J$,  $A_{J\backslash\{j\}}$ is surjective. 

Then either $J\subseteq I$ or $B$ and $B'$ satisfy a quadratic condition.
\end{lem}

\begin{proof}
Set $B''=B'-B$.

The affine line containing both $\omega_I$ and $\omega_J$ has equation $$(\sum_{j\in J}\lambda^J_jC_j)\epsilon+(\sum_{j\in J}\lambda^J_jB''_j)\delta+\sum_{j\in J}\lambda^J_jB_j=0,$$ where $\sum_{j\in J}\lambda^J_jC_j $ is zero or one and $(\delta_0,\epsilon_0)$ is the unique solution of the system $$(\mathcal{S}_I):\quad\left\lbrace\begin{array}{ccc}
\epsilon+(\sum_{i\in I}\lambda_i^IB''_i)\delta+\sum_{i\in I}\lambda_i^IB_i&=&0\\
(\sum_{i\in I}\lambda_i^{'I}B''_i)\delta+\sum_{i\in I}\lambda_i^{'I}B_i&=&0
\end{array}\right..$$
If $\sum_{j\in J}\lambda^J_jC_j =0$ we have $$(\sum_{i\in I}\lambda_i^{'I}B''_i)(\sum_{j\in J}\lambda^J_jB_j)=(\sum_{i\in I}\lambda_i^{'I}B_i)(\sum_{j\in J}\lambda^J_jB''_j) .$$ And if $\sum_{j\in J}\lambda^J_jC_j =1$, we have $$(\sum_{i\in I}\lambda_i^IB''_i)(\sum_{j\in J}\lambda^J_jB_j)=(\sum_{i\in I}\lambda_i^IB_i)(\sum_{j\in J}\lambda^J_jB''_j) .$$
In both cases, if there exists $j\in J$ that is not in $I$, then there exists $i\in I$ such that the coefficent of $B_jB''_i$ is non zero 
(it is $\lambda_i^{'I}\lambda_j^J$ in the first case and $\lambda_i^I\lambda_j^J$ is the second case) so that the condition is not trivial.
\end{proof}

\begin{cor}\label{cor:pointsegmentgeneraux}
For $B$ and $B'$ general we have the following: for any $I$ and $J$ subsets of $I_0$ 
such that $\omega_I$ is a singleton and $\omega_J$ is an open convex part of an affine line $\ell$, the affine line $\ell$ contains 
$\omega_I$ if and only if $\omega_{I\cap J}$ is a non-empty open convex part of $\ell$.
\end{cor}

\begin{proof}
By Lemma~\ref{lem:inclusions}, there exist $I'\subseteq I$ and $J'\subseteq J$ such that $\omega_{I'}=\omega_{I}$, $\omega_J\subset\omega_{J'}$ and for any $i\in I'$ and $j\in J'$, 
the image of $A_{I'\backslash\{i\}}$ has codimension one and $A_{J'\backslash\{j\}}$ is surjective. 
Then, by Lemma~\ref{lem:pointsegmentgeneraux}, , for general $B$ and $B'$, we must have $J'\subseteq I'$. 

In particular, $J'\subseteq I\cap J$ and by Lemma~\ref{lem:inclusions} we have $\omega_{I\cap J}\subset\omega_{J'}$, so that $\omega_{I\cap J}$ is either empty or an open convex part of an affine line (the same affine line as for $\omega_J$).

By with Lemma~\ref{lem:intersectionIJ}, the set $\omega_{I\cap J}$ is not empty.
\end{proof}

%\begin{rem}\label{rem:nombresrelations}
%Let $(\delta,\epsilon)$ be in $\Omega_\emptyset$. To a subset $I$ such that $\omega_I$ has dimension 1, resp. 0, and $(\delta,\epsilon)$ is an element of $\omega_I$, 
%is associated one non trivial relation, resp. two linearly independent relations as in Construction~\ref{def:alpha}. 

%Let $K$ be such that $\omega_K$ has dimension 1. Then a general point $(\delta,\epsilon)\in\omega_K$
%does not belong to any $\omega_I$ of dimension 0.
%Moreover, if there is $I$ with $\omega_I$ of dimension 1 and such that $\omega_I$ and $\omega_K$ intersect in more than one point, then $\omega_I$ and $\omega_K$ define the same relation.

%On the other hand there are only a finite number of points consisting of the intersection of two or more one-dimensional $\omega_I$.
%\end{rem} j'ai enlevé cette remarque car partiellement répétée quand on définit les U_i

We end this subsection with the following lemma in order to avoid the case where three affine lines intersect into a point except in only one situation.
\begin{lem}\label{lem:3int}
Let $I$, $J$ and $K$ be distinct subsets of $I_0$ such that the images of $A_I$, $A_J$ and $A_K$ have codimension one, and such that for any $i\in I$, $j\in J$ and 
$k\in K$, $A_{I\backslash\{i\}}$, $A_{J\backslash\{j\}}$ and $A_{K\backslash\{k\}}$ are surjective. If the three affine lines 
generated by  $\omega_I$, $\omega_J$ and $\omega_K$ intersect, then either $B$ and $B'$ satisfy some linear or quadratic condition or $I\cup J= I\cup K=J\cup K=I\cup J\cup K$.
\end{lem}

\begin{proof}

Set $B'':=B'-B$.

By Lemma~\ref{lem:segmentsgeneraux}, if the three lines are not distinct, we are  done. Suppose then that the three lines are not distinct and meet at the point $(\delta_0,\epsilon_0)$.

Let $L=I,J,K$. The affine line containing $\omega_{L}$ has equation
 $\sum_{\ell\in L}\lambda_\ell^L D_\ell(\delta,\epsilon)=0$.
If we set 
$$
\left\lbrace
\begin{array}{r}
a_L=\sum_{\ell\in L}\lambda^L_\ell C_\ell\\
b_L=\sum_{\ell\in L}\lambda^L_\ell B''_\ell\\
c_L=\sum_{\ell\in L}\lambda^L_\ell B_\ell
\end{array}
\right.
$$ 
then the equation becomes $a_L\epsilon+b_L\delta+c_L=0$.
By hypothesis the point $(a_L,b_L,c_L)$ belongs to the plane $a\epsilon_0+b\delta_0+c=0$ for $L=I,J,K$.

By Construction \ref{def:alpha} we have $a_L\in\{0,1\}$.

\noindent \underline{Assume that there is $L$ such that $a_L=0$.} Without loss of generality we can assume that $L=I$.
Notice that then only $I$ is such that $a_I=0$, because two affine lines distinct and parallel do not meet.  

 Then $b_I(c_J-c_K)=c_I(b_J-b_K)$, that is,
 
 $$(\sum_{i\in I}\lambda^I_iB''_i)(\sum_{j\in J}\lambda^J_jB_j-\sum_{k\in K}\lambda^K_kB_k)=(\sum_{i\in I}\lambda^I_iB_i)(\sum_{j\in J}\lambda^J_jB''_j-\sum_{k\in K}\lambda^K_kB''_k).$$ 

Assuming that we do not have $I\cup J= I\cup K=J\cup K=I\cup J\cup K$  we prove that this condition is not trivial.
There are three cases, either $J\cup K\neq I\cup J\cup K$, or $I\cup J\neq I\cup J\cup K$, or $I\cup K\neq I\cup J\cup K$. The last two cases are proved in the same way.

If $J\cup K\neq I\cup J\cup K$
 let $i\in I$ be such that  $i\not\in J\cup K$.
 As $J\neq K$ we can assume without loss of generality that there is $j\in J, j\not\in K$. Then the coefficient of $B_i''B_j$ in the quadratic condition is $\lambda_i^I\lambda_j^J$ which is non zero by Construction~\ref{def:alpha}.

If $I\cup J\neq I\cup J\cup K$
let $k\in K$ such that  $k\not\in I\cup J$.
 Then for every $i\in I$ the coefficient of $B_i''B_k$ in the quadratic condition is $\lambda_i^I\lambda_k^k$ which is non zero by Construction~\ref{def:alpha}.

\medskip

\noindent \underline{Assume that $a_L=1$ for every $L$.}
Without loss of generality we can assume that $I\cup J\neq I\cup J\cup K$.
We prove that the condition
$$(b_I-b_J)(c_I-c_K)=(b_I-b_K)(c_I-c_J)$$
verified by the coefficients of the two linear equations is not trivial.
By replacing $b_L$ and $c_L$ we get

\begin{align*}
(\sum_{i\in I}\lambda^I_iB''_i&-\sum_{j\in J}\lambda^J_jB''_j)(\sum_{i\in I}\lambda^I_iB_i-\sum_{k\in K}\lambda^K_kB_k)\\
&=(\sum_{i\in I}\lambda^I_iB''_i-\sum_{k\in K}\lambda^K_kB''_k)(\sum_{i\in I}\lambda^I_iB_i-\sum_{j\in J}\lambda^J_jB_j).
\end{align*} 
 
%$$(\sum_{i\in I}\lambda^I_iB''_i)(\sum_{j\in J}\lambda^J_jB_j)-(\sum_{i\in I}\lambda^I_iB_i)(\sum_{j\in J}\lambda^J_jB''_j)-(\sum_{i\in I}\lambda^I_iB''_i)(\sum_{k\in K}\lambda^K_kB_k)+(\sum_{i\in I}\lambda^I_iB_i)\sum_{k\in K}\lambda^K_kB''_k)+$$ $$(\sum_{j\in J}\lambda^J_jB''_j)(\sum_{k\in K}\lambda^K_kB_k)-(\sum_{j\in J}\lambda^J_jB_j)(\sum_{k\in K}\lambda^K_kB''_k)=0.$$

Let $k\in K$ be such that $k\not\in I\cup J$. 
 As $I\neq J$ we can assume without loss of generality that there is $i\in I, i\not\in J$.
Then the coefficient in the quadratic condition of $B_i''B_k$ is $\lambda_i^I\lambda_k^K$ which is non zero by Construction~\ref{def:alpha}.

%there exist disjoint subsets $A_0$, $A_1$, $A_2$ and $A_3$ of $I_0$ such that $I=A_0\cup A_1\cup A_2$, $J=A_0\cup A_1\cup A_3$ and $K=A_0\cup A_2\cup A_3$, and such that for any $i\in A_1$ and $j\in A_1$, $\lambda_i^I\lambda_j^J-\lambda_i^J\lambda_j^I=0$, for any $i\in A_2$ and $k\in A_2$, $\lambda_i^I\lambda_k^K-\lambda_i^K\lambda_k^I=0$, and for any $j\in A_3$ and $k\in A_3$, $\lambda_j^J\lambda_k^K-\lambda_i^K\lambda_k^J=0$. ++++ et alors???

\end{proof}

\begin{rem}\label{rem:2alpha}
Let $L$ be such that $\dim D_L^{-1}(Im A_L)=0$, $\omega_L=\{(\delta_0,\epsilon_0)\}$ and such that for every $\ell\in L$ the image of $A_{L\setminus\{\ell\}}$ has codimension 1 in $\Qbb^{|L|}$.

Then there are two linearly independent relations $\sum_{\ell\in L} \lambda_{\ell} X_{\ell}=0$ and $\sum_{\ell\in L} \mu_{\ell} X_{\ell}=0$ on the lines of $A_L$.
We set $I=\{\ell\in L\vert\;\lambda_{\ell}\neq 0 \}$ and $J=\{\ell\in L\vert\;\mu_{\ell}\neq 0 \}$. We can suppose that $I$ and $J$ are proper in $L$. 
Then, as the image of $A_{L\setminus\{\ell\}}$ has codimension 1 in $\Qbb^{|L|}$ for every $\ell\in L$, 
 the images of $A_I$ and $A_J$ have codimension 1
and we have $I\cup J=L$.

Viceversa, given proper subsets $I$ and $J$ in $L$ such that $I\cup J=L$ and the images of $A_I$ and $A_J$ have codimension 1, we have two such independent relations on the lines of $A_L$.

%\je{Moreover, let $P\subseteq\Qbb^{|L|}$ be the image of $A_L$, which has codimension 2. Let $\{e_p\}$ be a basis for $P$.
%The equations of the hyperplanes containing $P$ form a linear subspace of $\Qbb^{|L|}$, namely
% $$\mathcal H_P=\{\bar a\in\Qbb^{|L|}\vert\; \bar a\cdot e_p=0 \;\text{for all}\; p \}.$$
%We notice that the following set is non-empty
%$$\left(\mathcal H_P\times \mathcal H_P\setminus \bigcup_{\ell\in L}\{a_\ell=b_\ell\}\right)\cap\{\bar a\cdot C=1\}\cap\{\bar b\cdot C=1\}$$
%where we denote by $(\bar a,\bar b)$ an element of $\mathcal H_P\times \mathcal H_P$ and $C=C_L$.

%Therefore we can chose equations $\sum \lambda_{\ell}^L X_{\ell}=0$ and $\sum \lambda_{\ell}^{\prime L} X_{\ell}=0$ such that $\lambda_{\ell}^L\neq \lambda_{\ell}^{\prime L}$ for every $\ell$. 
%}
\end{rem}

\subsection{Polyhedral decomposition and the geography of models}\label{linkwithHMMP}

Let $G$ be a reductive group and $H\subseteq G$ be a horospherical subgroup.
We choose a basis of $M_{\mathbb Q}$ so that $M_{\mathbb Q}\cong\mathbb Q^n$.
We fix a horospherical embedding $Z$ of $G/H$ and set 
$p=r+|S\backslash R|$. Recall that, by the notation given in Section~\ref{sec:horo}, $r$ is the number of $G$-stable prime divisor of $Z$ and $|S\backslash R|$ is the number of $B$-stable prime divisor of $G/H$. In particular, $p$ is the number of $B$-stable prime divisor of $Z$.
Let $A$ be the $p\times n$ matrix associated to the linear map $\varphi(m)=(\langle m,x_i \rangle_{i=1\ldots r},\;  \langle m,\alpha_M^{\vee} \rangle_{\alpha\in S\backslash R})$.
Denote by $J_0\subseteq I_0$ the set of indices $S\backslash R$.

Let $B=(-d_1,\ldots, -d_r, (-d_{\alpha})_{\alpha\in S\backslash R})$ and $B'=(-d'_1,\ldots, -d'_r, (-d'_{\alpha})_{\alpha\in S\backslash R})$
be such that $D=\sum_{i=1}^d d_i Z_i+\sum_{\alpha\in S\backslash R} d_{\alpha} D_{\alpha}$ and $D'=\sum_{i=1}^d d'_i Z_i+\sum_{\alpha\in S\backslash R} d'_{\alpha} D_{\alpha}$ are ample divisors on two varieties $X$ and $Y$ respectively, in the sense that $\tilde{Q}_D$ and $\tilde{Q}_{D'}$ are pseudo-moment polytopes for $X$ and $Y$ respectively. Notice that it may occur that some $D_i$'s are not divisors of $X$ or $Y$ as we observe in Lemma~\ref{lem:HMMPendsXT}.
Let $C=(c_1,\ldots, c_r, (c_{\alpha})_{\alpha\in S\backslash R})$ be such that $c_i>0$ for every $i$ and 
$c_{\alpha}>0$ for every $\alpha$.

\noindent Since $X$ is projective, $\tilde{Q}_D=\{x\in M_\Qbb\,\mid\,Ax\geq B\}$ is a polytope and then Condition~\ref{condBDD} is verified by A.

From now on, following remark~\ref{rem:notpseudo}, if a polytope $Q$ is defined by $\{x\in M_\Qbb\,\mid\,Ax\geq D\}$, with $D=(-d_1,\ldots, -d_r, (-d_{\alpha})_{\alpha\in S\backslash R})$, we will say that it is the pseudo-moment polytope associated to the moment polytope $v^0+Q$ with $v^0=\sum_{\alpha\in S\backslash R}d_\alpha\varpi_\alpha$, or similarly the pseudo-moment polytope of the associated $G/H$-embedding as soon as $v^0+Q$ is a $G/H$-polytope (Definition \ref{def:G/H-eq}).

%We begin by the description of the horospherical associated to the polytopes $P^{\epsilon,\delta}$.
%Suppose that $A$, $B$, $B'$ and $C$ are coming from $G/H$-embeddings $X\longrightarrow T$ and $Y\longrightarrow S$ with a resolution $Z$ as in Section~\ref{sec:??}. In particular, $Q^{0,0}$ is a pseudo-moment poytope of $X$ and $Q^{0,1}$ is a pseudo-moment poytope of $Y$. The lines of $A$ are primitive elements of rays of the colored fan $\mathbb{F}_Z$ and the image $\alpha_M^\vee$ of colors of $G/H$. Denote by $J_0\subseteq I_0$ the set of lines coming from a color of $G/H$. We put the index $\alpha$ instead of $i$ if $A_i$ is the line of $A$ coming from the color $D_\alpha$.

\noindent We suppose that $B$ and $B'$ are general in the sense of the previous subsection, that is, that they belong to the Zariski open subset of the ample cone such that the conditions of Corollaries \ref{cor:generalB}, \ref{cor:segmentsgeneraux}, \ref{cor:pointsgeneraux}, \ref{cor:pointsegmentgeneraux} and Lemma \ref{lem:3int} are satisfied.

\smallskip

Then we define the following locally closed subsets of $\Omega_{\emptyset}$.
We denote by ${\rm Aff}(S)$ the affine space generated by a set $S$.
$$
\begin{array}{ll}
U_2=&\{(\delta,\epsilon)\in \Omega_{\emptyset}\vert\; (\delta,\epsilon)\in\omega_I\Rightarrow\dim \omega_I=2\}\\
U_1=&\{(\delta,\epsilon)\in \Omega_{\emptyset}\vert\;(\delta,\epsilon)\not\in U_2,\;(\delta,\epsilon)\in\omega_I\cap\omega_J\Rightarrow\dim {\rm Aff}\,\omega_I\cap{\rm Aff}\,\omega_J=1\}\\
U_0=&\{(\delta,\epsilon)\in \Omega_{\emptyset}\vert\;\exists I, \{(\delta,\epsilon)\}= \omega_I\}\\
U'_0=&\Omega_{\emptyset}\setminus (U_2\cup U_1 \cup U_0).
\end{array}
$$
Note that $U_0,U'_0$ are finite sets by the generality assumptions on $B$ and $B'$.

\begin{notat}
Let $(\delta,\epsilon)$ be such that 
 $P^{\delta,\epsilon}$ is not empty. 
\begin{enumerate}
\item If $(\delta,\epsilon)\in U_2$, by Definition \ref{def:G/H-eq} and Proposition \ref{prop:classprojembpoly}, 
the polytope $P^{\delta,\epsilon}$ is the pseudo-moment polytope of a $G/H$-embedding which we denote by $X^{\delta,\epsilon}$. 
\item If $(\delta,\epsilon)\in U_1$, by Lemma \ref{lem:Gmorph} there is $H'\supseteq H$ such that $P^{\delta,\epsilon}$ is the pseudo-moment polytope of a $G/H'$-embedding which we denote by $Y^{\delta,\epsilon}$.
\item If $(\delta,\epsilon)\in U_0$, by Lemma \ref{lem:Gmorph} there is $H'\supseteq H$ such that $P^{\delta,\epsilon}$ is the pseudo-moment polytope of a $G/H'$-embedding which we denote by $Z^{\delta,\epsilon}$. 
\end{enumerate}
\end{notat}

\begin{rem}\label{rem:ontheline}
If $(\delta,\epsilon)\in\omega_\emptyset$,  by Definition \ref{def:G/H-eq} and Proposition \ref{prop:classprojembpoly} the polytope $P^{\delta,\epsilon}$ is  the pseudo-moment polytope associated to a $G/H$-polytope. 
%\ora{reference in first section, forse dobbiamo puntualizzare che $U_2\subseteq \omega_{\emptyset}$}.
On the other hand, if $(\delta,\epsilon)\in\Omega_\emptyset\backslash\omega_\emptyset$, then there is $H\subsetneq H'$ such that   $P^{\delta,\epsilon}$ is  the pseudo-moment polytope associated to a $G/H'$-polytope. 
In this case, $H'$ is defined as follows: let $M'_\Qbb$ be the linear subspace generated by $P^{\delta,\epsilon}$ and $M'$ the sublattice $M'_\Qbb\cap M$ of $M$, let $R'$ be the union of $R$ with the set of simple roots $\alpha\in S\backslash R$ such that $P^{\delta,\epsilon}\subset\mathcal{H}^{\delta,\epsilon}_\alpha$, and define $P'$ to be the parabolic subgroup containing $B$ whose simple roots are $R'$; then $H'$ is the kernel of characters of $M'\subset\mathfrak{X}(P')$.

%If $(\delta,\epsilon)\in U'_0$ then there are $I,J$ such that 
%$\omega_I$ and $\omega_J$ generate two different lines.
%The set $U'_0$ is finite and we will not study the corresponding $G/H'$-embedding, because they will not appear in our further discussions.
\end{rem}

\begin{rem}\label{rem:U0Qf}
By Corollary~\ref{cor:generalB}, the generality assumption on $B$ and $B'$ and
Lemma \ref{lem:Qf} for every $(\delta,\epsilon)\in U_2$
the variety $X^{\delta,\epsilon}$ is $\Qbb$-factorial. 
In particular, by Corollary~\ref{cor:pic}, its Picard number 
is 
\begin{equation}\label{eq:pic}
\rho(X^{\delta,\epsilon})=|I_1^{\delta,\epsilon}\sqcup J_0|-n
\end{equation}

 where $I_1^{\delta,\epsilon}:=\{i\in I_0\backslash J_0\,\mid\,(\delta,\epsilon)\in\omega_i\}$. Since $(\delta,\epsilon)\in U_2$, we can replace $I_1^{\delta,\epsilon}$ by $I_2^{\delta,\epsilon}=\{i\in I_0\backslash J_0\,\mid\,(\delta,\epsilon)\in\Omega_i\}$ in the formula. 
\end{rem}

\medskip

Let $(\delta_0,\epsilon_0)\in U_2$.
Let $\ell\colon \Qbb\to \Qbb^2$ be the parametrisation of a rational affine line 
such that $\ell(0)=(\delta_0,\epsilon_0)$.
Let $\bar t>0$ be the minimum such that $\ell(\bar t)\in\Omega_{\emptyset}\setminus U_2$.
There is a $G$-equivariant morphism from $X^{\delta_0,\epsilon_0}$ to
the variety corresponding to $P^{\ell(\bar t)}$.
% some $Y^{\delta_1,\epsilon_1}$ and $Z^{\delta_2,\epsilon_2}$ as soon as, for $k=1$ or $2$, the closed segment between $(\delta_0,\epsilon_0)$ and $(\epsilon_k,\lambda_k)$ only meet a zero or one-dimensional $\omega$ in $ (\epsilon_k,\lambda_k)$. 
Indeed, for every $t\in[0,\bar t)$, 
 the polytope $P^{\ell(t)}$ is the moment polytope of a polarized variety $(X^{\ell(t)},D^{\ell(t)})$, and all the $P^{\ell(t)}$ for $t\in[0,\bar t)$ are equivalent.
 Thus the corresponding varieties are isomorphic to $X^{\delta_0,\epsilon_0}$.
The divisor $D(\ell(\bar t))$ is nef, but not ample, on $X^{\delta_0,\epsilon_0}$.
Hence, by
Lemma \ref{lem:Gmorph} 
there is a $G$ equivariant morphism from $X^{\delta_0,\epsilon_0}$ to
the variety corresponding to $P^{\ell(\bar t)}$.

\subsubsection{If $\ell(\bar t)\in U_1$}
We assume in this paragraph that $\ell(\bar t)=(\delta_1,\epsilon_1)\in U_1$.
The main result of this paragraph is Proposition \ref{prop:morphXY}, which describes
 the different sorts of $G$-equivariant morphisms we can get from $X^{\delta_0,\epsilon_0}$ to $Y^{\delta_1,\epsilon_1}$.
 We start with some preparatory lemmas and remarks.

 \begin{lem}\label{lem:conditionMori}
 Let  $(\delta,\epsilon)\in U_1$.
Let $I$ be such that  $(\delta,\epsilon)\in\omega_I$ and $\dim \omega_I=1$.
Assume moreover that $I$ is minimal with the property.
 Denote by $I_+:=\{i\in I\,\mid\, \lambda_i^I>0\}$ and $I_-:=\{i\in I\,\mid\, \lambda_i^I<0\}$. Note that $I=I_+\sqcup I_-$.
 
 Then $(\delta,\epsilon)\in\omega_\emptyset$ if and only if both $I_+$ and $I_-$ are not empty.
 \end{lem}
 
 \begin{proof}
 Let $\sum_{i\in I}\lambda_i^I X_i=0$ be a relation satified by the lines of $A_I$ as in Construction \ref{def:alpha}.
Notice that $\lambda_i^I\neq 0$ for any $i\in I$ by minimality of $I$.

 Fix $x\in\Qbb^n$ such that $A_Ix=D_I(\delta,\epsilon)$ and $A_{\bar{I}}x\geq D_{\bar{I}}(\delta,\epsilon)$.
 
Assume first that  $I_+$ and $I_-$ are not empty. 
Let $i\in I_+$, then 
 $$A_i=-\sum_{j\in I_+, j\neq i}\frac{\lambda_j^I}{\lambda_i^I} A_j+\sum_{j\in I_-}\frac{-\lambda_j^I}{\lambda_i^I} A_j.$$
By the minimality of $I$, the matrix $A_{I\backslash\{i\}}$ is surjective.
Therefore, there exists $y\in\Qbb^n$ such that $A_jy>0$ for any $j\in I\backslash\{i\}$ and $A_jy$ big enough for $j\in I_-$ so that $A_iy$ is also positive. 
 Then for any $t>0$ small enough, we have $A(x+ty)>D(\delta,\epsilon)$. 
 This means that $(\delta,\epsilon)\in \omega_\emptyset$.
 
 Conversely, suppose that $I^+=I$ and $I_-=\emptyset$. 
First, we prove that  $\sum_{i\in I}\lambda_i^I D_i(\delta,\epsilon)=0$.
Since $(\delta,\epsilon)\in\omega_I$, 
there is $x\in\Qbb^n$ such that $Ax\geq D(\delta,\epsilon)$ and $A_Ix= D_I(\delta,\epsilon)$. 
Thus $0=\sum_{i\in I}\lambda_i^I A_ix=\sum_{i\in I}\lambda_i^I D_i(\delta,\epsilon)$. 

Now, let $y\in\Qbb^n$ such that $Ay\geq  D(\delta,\epsilon)$. Then $0=\sum_{i\in I}\lambda_i^I A_iy\geq \sum_{i\in I}\lambda_i^I D_i(\delta,\epsilon)=0$, so that $A_Iy=D_I(\delta,\epsilon)$. In particular, $(\delta,\epsilon)\not\in \omega_\emptyset$.
 
 The same occurs if $I^-=I$.
 \end{proof}
 
 \begin{rem}\label{rem:IplusIminus}
 The proof implies in particular that, if both $I_+$ and $I_-$ are not empty, then there is $y\in\Qbb^n$ such that $A_Iy>0$.
 \end{rem}

\begin{lem}\label{lem:conditiondivisorielle}
 Let  $(\delta,\epsilon)\in U_1\cap\omega_\emptyset$.
Let $I$ be such that  $(\delta,\epsilon)\in\omega_I$ and $\dim \omega_I=1$.
Assume moreover that $I$ is minimal with the property.
 Let $i\in I_0\backslash J_0$. Then $(\delta,\epsilon)\in\Omega_i\backslash \omega_i$ if and only if $ I_+$ or $I_-$ equals $\{i\}$.

In particular, there is at most one $i\in I_0\backslash J_0$, such that $(\delta,\epsilon)$ is in $\Omega_i\backslash \omega_i$. 
\end{lem}

\begin{proof}
 Let $\sum_{i\in I}\lambda_i^I X_i=0$ be a relation satified by the lines of $A_I$ as in Construction \ref{def:alpha}.
Notice that $\lambda_i^I\neq 0$ for any $i\in I$ by minimality of $I$.
 
Suppose that there exists $i\in I_0\backslash J_0$ such that $(\delta,\epsilon)\in\Omega_i\backslash \omega_i$. Then there exists $J\subseteq I_0$ containing $i$ such that $(\delta,\epsilon)\in\omega_J$. Pick such $J$ minimal. By Lemma~\ref{lem:inclusions}(1), since $(\delta,\epsilon)\not\in\omega_i$, the image of $A_J$ has codimension either one or two. 
 By assumption on $(\delta,\epsilon)$, $\omega_J$ is not a point, so that $\omega_J$ is one-dimensional and then by Corollary \ref{cor:segmentsgeneraux} and the generality assumption, $I\subset J$. 
 By minimality of $J$, we have $J=I\cup\{i\}$.
 
Assume by contradiction that $i\not\in I$.
With Notation~\ref{def:alpha}, we have $\lambda_i^J=0$, that is, $J$ and $I$ give the same relation.

Since $(\delta,\epsilon)\in\omega_\emptyset$, by Lemma~\ref{lem:conditionMori} and Remark \ref{rem:IplusIminus} and \ref{rem:surjminusj}, there exists $y\in\Qbb^n$ such that $A_I y>0$ and such that $A_i y=0$. Let $x\in\Qbb^n$ such that $A_Jx=D_J(\delta,\epsilon)$ and $A_{\bar{J}}x>D_{\bar{J}}(\delta,\epsilon)$. Then for any $t>0$ small enough $A_i(x+ty)=D_i(\delta,\epsilon)$ and  $A_{\bar{i}}(x+ty)>D_{\bar{i}}(\delta,\epsilon)$.  This is a contradiction with $(\delta,\epsilon)\not\in\omega_i$. Then $i\in I$ and $I=J$.

By assumption the matrix $A_{I\backslash\{i\}}$ is surjective.
If $I\backslash\{i\}$ is not contained in either $I_+$ or $I_-$, we can find $y\in\Qbb^n$ such that $A_jy>0$ for any $j\in I\backslash\{i\}$ and $\sum_{j\in I\backslash\{i\}}\lambda_j^IA_jy=0$, which implies that $A_iy=0$. 
Let $x\in\Qbb^n$ be such that $A_Ix=D_I(\delta,\epsilon)$ and $A_{\bar{I}}x>D_{\bar{I}}(\delta,\epsilon)$. Then, for any $t>0$ small enough we have $A_i(x+ty)=D_i(\delta,\epsilon)$ and  $A_{\bar{i}}(x+ty)>D_{\bar{i}}(\delta,\epsilon)$.  This is a contradiction with $(\delta,\epsilon)\not\in\omega_i$. 
We conclude with Lemma~\ref{lem:conditionMori} and the hypothesis $(\delta,\epsilon)\in\omega_\emptyset$, that $\{i\}=I_+$ or $I_-$.\\

Conversely, suppose that $I_+=\{i\}$.  Note that $(\delta,\epsilon)\in \Omega_i$ is obvious. Since there exists $x\in\Qbb^n$ such that $A_Ix=D_I(\delta,\epsilon)$, we have $D_i(\delta,\epsilon)=\sum_{j\in I\backslash\{i\}}\frac{-\lambda_j^I}{\lambda_i^I} D_j(\delta,\epsilon)$. 

For any $y\in\Qbb^n$, such that $A_{\bar{i}}y\geq D_{\bar{i}}(\delta,\epsilon)$ we have $$A_iy=\sum_{j\in I\backslash\{i\}}\frac{-\lambda_j^I}{\lambda_i^I}A_jy \geq\sum_{j\in I\backslash\{i\}}\frac{-\lambda_j^I}{\lambda_i^I}D_j(\delta,\epsilon)=D_i(\delta,\epsilon).$$ And we have equality if and only if  $A_{I\backslash\{i\}}y=D_{I\backslash\{i\}}(\delta,\epsilon)$. Then $(\delta,\epsilon)\not\in \omega_i$.\\

For the last statment, note that if $|I|=1$ then $A_I$ is a zero line and $i\in J_0$; and if $|I|=2$ with $|I_+|=1$ then the two lines of $A_I$ generate the same ray of $N_\Qbb$ and at most one could be the primitive element of the ray, so that the other line has index in $J_0$.
\end{proof}

\begin{lem}\label{lem:conditiondivisorielleMori}
 Let  $(\delta,\epsilon)\in U_1\cap \Omega_\emptyset\backslash\omega_\emptyset$.
Let $I$ be such that  $(\delta,\epsilon)\in\omega_I$ and $\dim \omega_I=1$.
Assume moreover that $I$ is minimal with the property.
 Let $i\in I_0\backslash I$. Then either $(\delta,\epsilon)\not\in\Omega_{I\cup \{i\}}$ or $(\delta,\epsilon)\in \omega_{I\cup \{i\}}$.
\end{lem}

\begin{proof}
Let $i\in I_0\backslash I$ such that $(\delta,\epsilon)\in\Omega_{I\cup \{i\}}$. There exists $J\subseteq I_0$ containing $i$ and $I$ such that $(\delta,\epsilon)\in\omega_J$. By hypothesis on $(\delta,\epsilon)$ the set $\omega_J$ is one-dimensional and Lemma~\ref{lem:inclusions} implies that $\omega_J\subseteq \omega_{I\cup \{i\}}$. %In particular, in the same affine line as $\omega_I$. 
In particular, $(\delta,\epsilon)\in \omega_{I\cup \{i\}}$.
\end{proof}

\begin{cor}\label{cor:Bdiv}
\begin{enumerate}
Let $(\delta,\epsilon)\in U_1$.
\item\label{Bdiv1} Assume that $(\delta,\epsilon)\in \omega_{\emptyset}$. If there is $i\in I_0\backslash J_0$ such that $I_+$ or $I_-$ equals $\{i\}$, then the $G$-stable prime divisors of  $Y^{\delta,\epsilon}$ are in bijection with  $I_1^{\delta,\epsilon}=I_2^{\delta,\epsilon}\backslash\{i\}$. Otherwise 
the $G$-stable prime divisors of  $Y^{\delta,\epsilon}$ are in bijection with $I_1^{\delta,\epsilon}=I_2^{\delta,\epsilon}$.
\item\label{Bdiv2} Assume that $(\delta,\epsilon)\in \Omega_\emptyset\backslash\omega_{\emptyset}$.
Then the $G$-stable prime divisors of $Y^{\delta,\epsilon}$ are in bijection with  $$\{i\in I_0\backslash I\,\mid\,(\delta,\epsilon)\in\omega_{I\cup\{i\}}\}=\{i\in I_0\backslash I\,\mid\,(\delta,\epsilon)\in\Omega_{I\cup\{i\}}\}=I_2^{\delta,\epsilon}\backslash (I_2^{\delta,\epsilon}\cap I).$$ 
\end{enumerate}
\end{cor}
\begin{proof}
The set of $B$-stable prime divisors of $Y^{\delta,\epsilon}$ is the union of $G$-stable divisors and colors of the horospherical homogeneous space. The  $G$-stable prime divisors correspond bijectively to the facets $F_i^{\delta,\epsilon}$ with $i\in I_0\backslash J_0$ that are not equal to some $F_j^{\delta,\epsilon}$ with $j\in J_0$.

Assume that  $(\delta,\epsilon)\in \omega_{\emptyset}$.
Thus $Y^{\delta,\epsilon}$ is a $G/H$-embedding, colors are indexed by $S\backslash I$ or equivalently by $J_0$. If there are $i$ and $j\in I_0$ with $i\neq j$ and such that $F_i^{\delta_1,\epsilon_1}=F_j^{\delta_1,\epsilon_1}$ is a facet, then $A_i$ and $A_j$ are colinear, one of $i$ or $j$ is in $J_0$, and $I={i,j}$.
Lemma~\ref{lem:conditiondivisorielle} implies the claim.

Assume that  $(\delta,\epsilon)\in\Omega_\emptyset\backslash\omega_\emptyset$.
Then $Y^{\delta,\epsilon}$ is a $G/H'$-embedding with $H'\subsetneq H$ as explained in Remark~\ref{rem:ontheline}. In the proof of Lemma~\ref{lem:conditionMori}, we prove that $P^{\delta,\epsilon}$ generates the affine subspace $\{x\in M_\Qbb\,\mid\,A_Ix=D_I(\delta,\epsilon)\}$, so that $M'=\ker(A_I)\cap M$. Moreover, the colors of $G/H'$ are indexed by $J_0\backslash(J_0\cap I)$. If for some $i,j\in I_0$, $F_i^{\delta,\epsilon}=F_j^{\delta,\epsilon}$ is a facet, then $i,j\not\in I$, $A_i$ and $A_j$ are colinear modulo the vector space $N''_\Qbb$ generated by the lines of $A_I$. If $i\neq j$ it would give another relation to the one given by $I$, then $i=j$. Note also that for any $i\in I_0\backslash I$, $F_i^{\delta,\epsilon}= F_{I\cup\{i\}}^{\delta,\epsilon}$, and if it is not empty then $A_i\not\in N''_\Qbb$. Indeed, if $A_i\in N''_\Qbb$, we have another relation to the one given by $I$. 

Hence Lemma~\ref{lem:conditiondivisorielleMori} implies the claim.
\end{proof}

\begin{prop}\label{prop:morphXY}
Let $(\delta_0,\epsilon_0)\in U_2$.
Let $\ell\colon \Qbb\to \Qbb^2$ be the parametrisation of a rational affine line 
such that $\ell(0)=(\delta_0,\epsilon_0)$.
Let $\bar t>0$ be the minimum such that $\ell(\bar t)\in\Omega_{\emptyset}\setminus U_2$.
Assume  that $\ell(\bar t)=(\delta_1,\epsilon_1)\in U_1$.

Set $I$ minimal such that $(\delta_1,\epsilon_1)\in\omega_I$.

The morphism from $X^{\delta_0,\epsilon_0}$ to $Y^{\delta_1,\epsilon_1}$ is an extremal contraction and one of the following occurs.
\begin{enumerate}
\item If $I=I_+$ or $I_-$ then $\dim Y^{\delta_1,\epsilon_1}<\dim X^{\delta_0,\epsilon_0}$, $Y^{\delta_1,\epsilon_1}$ is $\Qbb$-factorial and $\rho(X)=\rho(Y)+1$. % \marginpar{terminality is missing so no Mfs}
\item If $|I|\geq 2$ and $I_+$ or $I_-$ is $\{i\}$ with $i\in I_0\backslash J_0$ then  $Y^{\delta_1,\epsilon_1}$ is $\Qbb$-factorial and $X^{\delta_0,\epsilon_0}\to Y^{\delta_1,\epsilon_1}$ is an extremal divisorial contraction or an isomorphism.
\item In the other cases  $X^{\delta_0,\epsilon_0}\to Y^{\delta_1,\epsilon_1}$ is a small extremal contraction.
\end{enumerate}
\end{prop}

\begin{proof}
\underline{Assume that $I=I_+$ or $I_-$.}
By Corollary~\ref{cor:Bdiv} we have $(\delta_1,\epsilon_1)\in\Omega_{\emptyset}\setminus\omega_\emptyset$ and by Remark \ref{rem:ontheline} we have $\dim Y^{\delta_1,\epsilon_1}<\dim X^{\delta_0,\epsilon_0}$.

By Lemma~\ref{lem:conditiondivisorielleMori}, the $B$-stable prime divisors of $Y^{\delta_1,\epsilon_1}$ are in bijection with $(I_2^{\delta_1,\epsilon_1}\sqcup J_0)\backslash I$. 
Let $F_J^{\delta_1,\epsilon_1}$ be a face of $P^{\delta_1,\epsilon_1}$, choose $J$ such that $(\delta_1,\epsilon_1)\in\omega_J$. 
Then $I\subseteq J$ and $\omega_J$ is one-dimensional, so that the restriction of $A_{J\backslash I}$ to $\ker(A_I)$ is surjective. 
This implies,   with Lemma~\ref{lem:Qf} applied to the matrix of lines of  $A_{I_0\backslash I}$ restricted to $\ker(A_I)$, that $Y^{\delta_1,\epsilon_1}$ is $\Qbb$-factorial. 
Then its Picard number is $|(I_2^{\delta_1,\epsilon_1}\sqcup J_0)\backslash I|-(n-|I|+1)=|I_2^{\delta_1,\epsilon_1}\sqcup J_0|-n-1$. 
But $I_2^{\delta_1,\epsilon_1}=I_2^{\delta_0,\epsilon_0}$ so that the relative Picard number of $X^{\delta_0,\epsilon_0}\longrightarrow Y^{\delta_1,\epsilon_1}$ is~1.

\medskip

\underline{Assume that $I_+$ or $I_-$ is $\{i\}$ with $i\in I_0\backslash J_0$ and $|I|\geq 2$.}
By Lemma \ref{lem:conditionMori} we have  $(\delta_1,\epsilon_1)\in\omega_\emptyset$.
By Remark \ref{rem:ontheline} we have $\dim Y^{\delta_1,\epsilon_1}=\dim X^{\delta_0,\epsilon_0}$.

By Corollary~\ref{cor:Bdiv}, the $B$-stable prime divisors of $Y^{\delta_1,\epsilon_1}$ are in bijection with 
$(I_2^{\delta_1,\epsilon_1}\backslash\{i\})\sqcup J_0=(I_2^{\delta_0,\epsilon_0}\backslash\{i\})\sqcup J_0$.
For the sake of shortness we denote by $K\setminus\{i\}$ the set $K\cap (I_0\setminus\{i\})$.

Let $F_J^{\delta_1,\epsilon_1}$ be a face of $P^{\delta_1,\epsilon_1}$, choose $J$ such that $(\delta_1,\epsilon_1)\in\omega_J$. 

If $A_J$ is not surjective then  $I\subseteq J$ and $A_{J\backslash\{i\}}$ is surjective. This implies,  with Lemma~\ref{lem:Qf}, 
that $Y^{\delta_1,\epsilon_1}$ is $\Qbb$-factorial. Then its Picard number is $|(I_2^{\delta_0,\epsilon_0}\backslash\{i\})\sqcup J_0|-n$. 
This is the Picard number of $X^{\delta_0,\epsilon_0}$ minus one if $i\in I_2^{\delta_0,\epsilon_0}$ and the Picard number of $X^{\delta_0,\epsilon_0}$ if not. 
In that latter case, since $P^{\delta_1,\epsilon_1}$ can be defined as well without the line $i$, it is an equivalent $G/H$-polytope to $P^{\delta_0,\epsilon_0}$, and 
then $X^{\delta_0,\epsilon_0}=Y^{\delta_1,\epsilon_1}$.

\medskip

\underline{In the other cases}, again y Lemma \ref{lem:conditionMori} and
Remark \ref{rem:ontheline} we have $\dim Y^{\delta_1,\epsilon_1}=\dim X^{\delta_0,\epsilon_0}$.

Corollary~\ref{cor:Bdiv} implies that the $B$-stable prime divisors of $Y^{\delta_1,\epsilon_1}$ are the same as the ones of $X^{\delta_0,\epsilon_0}$,
they are therefore in bijection with $I_2^{\delta_0,\epsilon_0}\sqcup J_0$. 

Let $F_J^{\delta_1,\epsilon_1}$ be a face of $P^{\delta_1,\epsilon_1}$, choose $J$ such that $(\delta_1,\epsilon_1)\in\omega_J$. If $A_J$ is not surjective, then $\omega_J$ is one-dimensional and $I\subseteq J$. Then a $B$-stable $\Qbb$-divisor of $Y^{\delta_1,\epsilon_1}$ is $\Qbb$-Cartier if and only if its coefficients satisfy the equation $\sum_{i\in I}\lambda_i^I X_i$. Also note that for $I=J$, $A_J$ is not surjective.

Then the relative Picard number of $X^{\delta_0,\epsilon_0}\longrightarrow Y^{\delta_1,\epsilon_1}$ is~1.
\end{proof}

 \begin{rem}
 
 If  $I^+=I$ or $I^-=I$, for every $i\in I$ the polytope $P^{\delta,\epsilon}$ is contained in $\mathcal{H}_i^{\delta,\epsilon}=\{x\in\Qbb^n\,\mid\,A_ix=D_i^{\delta,\epsilon}\}$. 
  There are two possible cases for $I$.
  Either $I=\{i\}$ and $A_i=0$, so that $P^{\delta,\epsilon}$ has dimension~$n$, the associated horospherical variety $X$ has the same rank as $G/H$, 
  the same lattice $M$ but its open homogeneous space has one color less ($\alpha_i$); or $A_i\neq 0$ for any $i\in I$ (and $|I|\geq 2$), so that  $P^{\delta,\epsilon}$ is of maximal dimension an 
  affine subspace directed by $\ker(A_I)$, then it has dimension $n-|I|+1$, and the associated horospherical variety $X$ has rank $n-|I|+1$, its lattice is $M\cap \ker A_I$ and the colors of its 
  open homogeneous space are the colors of $G/H$ whose index of line in not in $I$. \\
  
  If $I^+=\{i\}$ or $I^-=\{i\}$, then  the condition $A_ix\geq D_i(\delta,\epsilon)$ is superfluous in the definition of $P^{\delta,\epsilon}$.\\

 \end{rem}

 \subsubsection{If $\ell(\bar t)\in U_0$}
We assume in this paragraph that $\ell(\bar t)=(\delta_2,\epsilon_2)\in U_0$.
We want to study the morphisms from $X^{\delta_0,\epsilon_0}$ to $Z^{\delta_2,\epsilon_2}$. 

Let $L$ be such that $\{(\delta_2,\epsilon_2)\}=\omega_L$.
By Lemma~\ref{lem:inclusionclosure}, there is $I$ such that $\omega_I$ has dimension  one  and  $(\delta_2,\epsilon_2)\in \Omega_I$. 
Then all but a finite set of points of $\omega_I$ are in $\omega_I\cap U_1$.
Choose  $(\delta_1,\epsilon_1)\in \omega_I\cap U_1$ close to $(\delta_2,\epsilon_2)$.

By Lemma \ref{lem:Gmorph} there are $G$-equivariant morphisms from $X^{\delta_0,\epsilon_0}$ to $Y^{\delta_1,\epsilon_1}$ and  $Z^{\delta_2,\epsilon_2}$, and from  $Y^{\delta_1,\epsilon_1}$ to  $Z^{\delta_2,\epsilon_2}$. Moreover the morphism from  $X^{\delta_0,\epsilon_0}$ to $Z^{\delta_2,\epsilon_2}$ factorizes through $Y^{\delta_1,\epsilon_1}$. %\ora{why? maybe make this precise in Rem \ref{rem:Gmorph}?} \mage{\`a revoir encore j'essaie d'expliquer en francais: en fait les morphismes entre $G/H$-embeddings (donc qui stabilisent aussi le point choisi dans l'ouvert $G$-stable) sont uniques s'ils existent. Ainsi, ici, un morphisme entre $X^{\delta_0,\epsilon_0}$ et $Z^{\delta_2,\epsilon_2}$ existe pour deux raisons, directement par la remarque ou comme composition de deux morphismes qui existent eux aussi par la remarque. Or ces deux morphismes doivent \^etre \'egaux, car le point choisi dans la $G$-orbite ouverte est toujours ici une somme des vecteurs de plus haut poids. }
 
We suppose from now on that $\{(\delta_2,\epsilon_2)\}=\omega_L\subset\Omega_\emptyset\backslash\omega_\emptyset$. 

\begin{notc}\label{def:rel2}
Let $L$ be such that $\omega_L=\{(\bar\delta,\bar\epsilon)\}\subset\Omega_\emptyset\backslash\omega_\emptyset$. Then, for general $B$ and $B'$, the image of $A_L$ has codimension two. By Lemma~\ref{lem:inclusions},  we can suppose up to taking a subset of $L$ that for every $\ell\in L$ the image of $A_{L\setminus\{\ell\}}$ has codimension 1 in $\Qbb^{|L|}$.

We have two possible cases.
\begin{enumerate}
\item If $\omega_L$ is a vertex of $\Omega_\emptyset\backslash\omega_\emptyset$, by remark \ref{rem:2alpha} there exist $I$ and $J$ subsets of $L$ such that, for any $i\in I$ and any $j\in J$, $A_{I\backslash \{i\}}$ and $A_{J\backslash \{j\}}$ are surjective, and $\omega_I$, $\omega_J$ are segments of the Mori's polygonal chain. By Lemma~\ref{lem:conditionMori} we can fix two linearly independent equations
for the lines of $A_L$
$$
\begin{array}{lll}
(\mathcal R_I)&\sum_{i\in I}\lambda_i^I X_i=0, & \lambda_i^I>0\; \forall i\in I  \\
(\mathcal R_J)&\sum_{j\in J}\lambda_j^J X_j=0, & \lambda_j^J>0\; \forall j\in J
\end{array}
$$

\noindent We suppose that $\omega_I\subseteq \{\delta<\bar\delta\}$ and that 
$\omega_J\subseteq \{\delta>\bar\delta\}$.\\
We can moreover assume that $\sum_{i\in I}\lambda_i^IC_i=1$ and $\sum_{j\in J}\lambda_j^JC_j=1$.

\item If $\omega_L$ is not a vertex of $\Omega_\emptyset\backslash\omega_\emptyset$, then there is $I\subseteq L$ such that, for any $i\in I$ the matrix  $A_{I\backslash \{i\}}$ is surjective, and $\omega_I$ is a segment of $\Omega_\emptyset\backslash\omega_\emptyset$ containing $\omega_L$.
By Lemma~\ref{lem:conditionMori}, we can fix two linearly independent equations
for the lines of $A_L$
$$
\begin{array}{lll}
(\mathcal R_I)&\sum_{i\in I}\lambda_i^I X_i=0, & \lambda_i^I>0\; \forall i\in I  \\
(\mathcal R_L)&\sum_{j\in L}\lambda_j^L X_j=0, & \lambda_j^L\neq 0\; \forall j\in L\setminus I
\end{array}
$$

We verify that $\lambda_j^L\neq 0\; \forall j\in L\setminus I$. Indeed if the second equation has one zero coefficient $\lambda_j^L$ for some $j\in L\backslash I$, then $A_{L\backslash\{j\}}$
would also have an image of codimension~2, contradicting the minimality of $L$. 
Moreover, the coefficients with indexes in $L\backslash I$ in the relation $(\mathcal R_L)$ do not have all the same sign. 

Indeed, if not, we can suppose that that they are positive, and by adding to $(\mathcal R_L)$ a positive multiple of $(\mathcal R_I)$, 
we would obtain a second equation, linearly independent with $(\mathcal R_I)$, with positive coefficients, associated to some $J\subsetneq I$. 
By Lemma \ref{lem:conditionMori} the set $\omega_J$ would be contained in $\Omega_{\emptyset}\setminus\omega_{\emptyset}$,
contradicting the hypothesis on $\omega_L$, as it would be the vertex between the segments $\omega_I$ and $\omega_J$.

With the same argument we can also suppose that $\lambda_j^L\geq 0$ for any $j\in I$ but some are zero.

We set $(L\setminus I)_+=\{j\in L\setminus I\,\vert\;\lambda_j^L>0 \}$ and $(L\setminus I)_-=\{j\in L\setminus I\,\vert\;\lambda_j^L<0 \}$.

We can moreover assume that $\sum_{i\in I}\lambda_i^IC_i=1$ and either $\sum_{j\in L}\lambda_j^LC_j=1$, or $\sum_{j\in L}\lambda_j^LC_j=0$ and $\sum_{j\in L}\lambda_j^L (B_j'-B_j)>0$. 
Indeed if $\sum_{j\in L}\lambda_j^LC_j<0$ we take the opposite and add a positive multiple of the first equation and we obtain the wanted equation with $\sum_{j\in L}\lambda_j^LC_j>0$. And if $\sum_{j\in L}\lambda_j^LC_j=0$ and $\sum_{j\in L}\lambda_j^L (B_j'-b_j)>0$, we take the opposite and add a positive multiple of the first equation and we obtain $\sum_{i\in I}\lambda_i^IC_i>0$.

\end{enumerate}

\end{notc}

\begin{lem}\label{lem:diviseursZ1}
Let $(\delta,\epsilon)\in U_0\cap \Omega_\emptyset\backslash\omega_\emptyset$ and assume that $(\delta,\epsilon)$ is a vertex.
Let $L\subseteq I_0$ be a minimal subset such that $\omega_L=\{(\delta,\epsilon)\}$.
Let $i\in I_0\backslash L$. Then either $(\delta,\epsilon)\not\in\Omega_{L\cup \{i\}}$ or $(\delta,\epsilon)\in \omega_{L\cup \{i\}}$.
\end{lem}
 
 \begin{proof}
If $(\delta,\epsilon)\in \Omega_{L\cup \{i\}}\backslash \omega_{L\cup \{i\}}$, there exists $J\subseteq I_0$ containing $i$ and $L$ such that $(\delta,\epsilon)\in \omega_J$. By Lemma~\ref{lem:inclusions}, we have $\omega_{L\cup\{i\}}=\omega_L=\omega_J$. In particular $(\delta,\epsilon)\in \omega_{L\cup \{i\}}$.

 \end{proof}
 
\begin{lem}\label{lem:diviseursZ2}
Let $(\delta,\epsilon)\in U_0\cap \Omega_\emptyset\backslash\omega_\emptyset$ and assume that $(\delta,\epsilon)$ is not a vertex of $\Omega_\emptyset\backslash\omega_\emptyset$.
Let $L\subseteq I_0$ be a minimal subset such that $\omega_L=\{(\delta,\epsilon)\}$.
 Let $I$ be a minimal subset such that $\omega_L\subseteq\omega_I$ and $\dim \omega_I=1$.
Let $i\in I_0\backslash I$. If $(\delta,\epsilon)\in\Omega_{I\cup \{i\}}\backslash \omega_{I\cup \{i\}}$ then $i\in L\backslash I$ and either
$L\setminus I_+=\{i \}$ or $L\setminus I_-=\{i \}$.
% $\lambda_i^L$ has an opposite sign with the other $\lambda_j^L$ with $j\in L\backslash I$.
In particular, there are at most $2$ indices $i$ such that $(\delta,\epsilon)\in\Omega_{I\cup \{i\}}\backslash \omega_{I\cup \{i\}}$. If there are $2$ such indices, then $|L|=|I|+2$ and the restriction of $A_{L\backslash I}$ to $\ker(A_I)$ consists of twice the same line if $L\backslash I\subseteq I_0\backslash J_0$. 
\end{lem}

\begin{proof}

The proof is  similar to the proof of Lemma~\ref{lem:conditiondivisorielle}.% \ora{si l'article devient trop long, on peut l'enlever? Je ne sais pas ce n'est pas si simple \`a recopier}

  If $(\delta,\epsilon)\in \Omega_{I\cup \{i\}}\backslash \omega_{I\cup \{i\}}$, then there exists $J\subseteq I_0$ containing $i$ and $I$ such that $(\delta,\epsilon)\in \omega_J$
  and $\omega_J$ is one or zero-dimensional. The set $\omega_J$ cannot have dimension one, otherwise Lemma~\ref{lem:inclusions} would imply 
  that $(\delta,\epsilon)\in \omega_J\subseteq\omega_{I\cup \{i\}}$. Hence, $\omega_J=\omega_L$ and $J=L\cup\{i\}$.

 We prove that  $i\in L$. Assume that it is not the case.
 Since the coefficients $\lambda^L_j$
 do not have the same sign for $j\in L\backslash I$, there exist positive $\mu_j^L$ such that $\sum_{j\in  L\backslash I}\lambda_j^L\mu_j^L=0$. %} \mage{est-ce qu'on l'utilise? oui juste apr\`es on utilise les $\mu$}
 By Remark \ref{rem:surjminusj} there exists $y\in\ker(A_I)$ such that $A_jy=\mu_j^L>0$ for any $j\in  L\backslash I$ and  so that $A_iy=0$. 
 Since $(\delta,\epsilon)\in \omega_{L\cup\{i\}}$, there is
 $x\in M_\Qbb$ is such that $A_{L\cup\{i\}}x=D_{L\cup\{i\}}(\delta,\epsilon)$ 
and  $A_{\overline{L\cup\{i\}}}x>D_{\overline{L\cup\{i\}}}(\delta,\epsilon)$.
Thus for $t$ small, we have $A_{I\cup\{i\}}(x+ty)=D_{I\cup\{i\}}(\delta,\epsilon)$
and $A_{L\setminus I}(x+ty)>D_{L\setminus I}(\delta,\epsilon)$
proving that $(\delta,\epsilon)\in\omega_{I\cup \{i\}}$, which is a contradiction.
Thus $i\in L$.
 
We prove now that all the $\lambda_j^L$ for $j\in L\backslash(I\cup \{i\})$ have the same sign. Again, by contradiction, if the $\lambda_j^L$ for $j\in L\backslash(I\cup \{i\})$ do not have the same sign, by Remark \ref{rem:surjminusj} we can find $y\in\ker(A_I)$ such that $A_jy>0$ for any $j\in L\backslash(I\cup \{i\})$ and $A_iy=0$. Let $x\in\Qbb^n$ such that $A_Lx=D_L(\delta,\epsilon)$ and $A_{\bar{L}}x>D_{\bar{L}}(\delta,\epsilon)$. Then for any $t>0$ small enough $A_{I\cup\{i\}}(x+ty)=D_{I\cup\{i\}}(\delta,\epsilon)$ and  $A_{\overline{I\cup\{i\}}}(x+tx')>D_{\overline{I\cup\{i\}}}(\delta,\epsilon)$.  This is a contradiction, as $(\delta,\epsilon)\not\in\omega_{I\cup\{i\}}$. 

Finally, as the coefficients of $(\mathcal R_L)$ do not have all the same sign, $\lambda_i^L$ has opposite sign than $\lambda_j^L$ for $j\in L\backslash(I\cup \{i\})$.
%Conversely, suppose that $\i\in L\backslash I$ and and the $\lambda_j^L$ for $j\in L\backslash(I\cup \{i\})$ have the same sign.  Note that $(\delta,\epsilon)\in \Omega_{I\cup \{i\}}$ is obvious. Since there exists $x\in\Qbb^n$ such that $A_Lx=D_L(\delta,\epsilon)$ then we have $D_i(\delta,\epsilon)=\sum_{j\in L\backslash\{i\}}\frac{-\lambda_j^I}{\lambda_i^L} D_j(\delta,\epsilon)$. 

%Then, for any $y\in\Qbb^n$,  $A_{\bar{i}}y\geq D_{\bar{i}}(\delta,\epsilon)$ then $$A_iy=\sum_{j\in I\backslash\{i\}}\frac{-\lambda_j^I}{\lambda_i^I}A_j \geq\sum_{j\in I\backslash\{i\}}\frac{-\lambda_j^I}{\lambda_i^I}D_j(\delta,\epsilon)=D_i(\delta,\epsilon).$$ And we have equality if and only if  $A_{L\backslash\{i\}}y=D_{L\backslash\{i\}}(\delta,\epsilon)$. Then $(\delta,\epsilon)\not\in \omega_{I\cup \{i\}}$.\\

The last statment is not difficult.
 
 \end{proof}

\begin{cor}\label{cor:Bdiv2}
Let $(\delta,\epsilon)\in U_0\cap\partial\Omega_{\emptyset}$.
Let $L$ be minimal such that $\omega_L=\{(\delta,\epsilon)\}$.
\begin{enumerate}
\item If $\omega_L$ is a vertex, the set of  $G$-stable prime divisors of  $Z^{\delta,\epsilon}$
is in bijection  with $I_2^{\delta,\epsilon}\backslash (I_2^{\delta,\epsilon}\cap L)$;
\item if $\omega_L$ is not a vertex and there is $i\in I_0\setminus J_0$ such that either $L\backslash I_+$ or $L\backslash I_-$ is $\{i\}$, the set of  $G$-stable prime divisors of  $Z^{\delta,\epsilon}$
is in bijection  with $I_2^{\delta,\epsilon}\setminus\{i\}$;
\item otherwise, the set of  $G$-stable prime divisors of  $Z^{\delta,\epsilon}$
is in bijection  with $I_2^{\delta,\epsilon}$.
\end{enumerate}
\end{cor} 
 
\begin{proof}
By  Remark~\ref{rem:ontheline}, the variety $Z^{\delta,\epsilon}$ is a $G/H'$-embedding with $H'\subsetneq H$. We let $M'$ be the lattice of characters vanishing on $H$ and $N'$ its dual. 
Note that $N'_\Qbb$ is the quotient $N_\Qbb/N''_\Qbb$ where $N''_\Qbb$ is the subspace generated by the lines of $A_L$.

  The  $G$-stable prime divisors of $Z^{\delta,\epsilon}$ correspond bijectively to the facets $F_i^{\delta,\epsilon}$ with $i\in I_0\backslash J_0$ that are not equal to some $F_j^{\delta,\epsilon}$ with $j\in J_0$.
  
  Assume that $\omega_L$ is a vertex.
   The polytope $P^{\delta,\epsilon}$ generates the affine subspace $\{x\in M_\Qbb\,\mid\,A_Lx=D_L(\delta,\epsilon)\}$, so that $M'=\ker(A_L)\cap M$. Moreover, the colors of $G/H'$ are indexed by $J_0\backslash(J_0\cap L)$. If for some $i,j\in I_0$, $F_i^{\delta,\epsilon}=F_j^{\delta,\epsilon}$ is a facet, then, $i,j\not\in L$, $A_i$ and $A_j$ are colinear modulo $N''_\Qbb$. If $i\neq j$ we would have another relation other than the two given by $L$. Thus we have $i=j$.
 
  Note also that for any $i\in I_0\backslash L$, we have $F_i^{\delta,\epsilon}= F_{L\cup\{i\}}^{\delta,\epsilon}$. If the latter is non empty and if $A_i\in N''_\Qbb$ we would have another relation other than the two given by $L$ so that $A_i\not\in N''_\Qbb$.
   Hence by Lemma~\ref{lem:diviseursZ1} the $G$-stable prime divisors of the variety $Z^{\delta,\epsilon}$ are in bijection with  $$\{i\in I_0\backslash L\,\mid\,(\delta,\epsilon)\in\omega_{L\cup\{i\}}\}=\{i\in I_0\backslash L\,\mid\,(\delta,\epsilon)\in\Omega_{L\cup\{i\}}\}=I_2^{\delta,\epsilon}\backslash (I_2^{\delta,\epsilon}\cap L).$$ 
  
  Assume that $\omega_L$ is not a vertex.
   The polytope $P^{\delta,\epsilon}$ generates the affine subspace $\{x\in M_\Qbb\,\mid\,A_Ix=D_I(\delta,\epsilon)\}$, so that $M'=\ker(A_I)\cap M$. Moreover, the colors of $G/H'$ are indexed by $J_0\backslash(J_0\cap I)$. If for some $i\neq j\in I_0$, $F_i^{\delta,\epsilon}=F_j^{\delta,\epsilon}$ is a facet, then $A_i$ and $A_j$ are colinear modulo the vector space $N'_\Qbb$ generated by the lines of $A_I$, and then $L=I\cup \{i,j\}$. Hence, by Lemma~\ref{lem:diviseursZ2}, the $G$-stable prime divisors of the variety $Z^{\delta,\epsilon}$ are in bijection with  $$\{i\in I_0\backslash I\,\mid\,(\delta,\epsilon)\in\omega_{I\cup\{i\}}\}=I_2^{\delta,\epsilon}\backslash\{i\}$$ if $L\backslash I_+$ or $L\backslash I_-$ is $\{i\}$ with $i\in I_0\backslash J_0$, and else with $$\{i\in I_0\backslash I\,\mid\,(\delta,\epsilon)\in\omega_{I\cup\{i\}}\}=I_2^{\delta,\epsilon}.$$
\end{proof}

 \begin{prop}\label{prop:morphXZ}
 Let $(\delta_0,\epsilon_0)\in U_2$.
Let $\ell\colon \Qbb\to \Qbb^2$ be the parametrisation of a rational affine line 
such that $\ell(0)=(\delta_0,\epsilon_0)$.
Let $\bar t>0$ be the minimum such that $\ell(\bar t)\in\Omega_{\emptyset}\setminus U_2$.
Assume  that $\ell(\bar t)=(\delta_2,\epsilon_2)\in U_0\cap\Omega_\emptyset\backslash\omega_\emptyset$.

Let $L$ be a minimal set such that $\{(\delta_2,\epsilon_2)\}=\omega_L$.
The morphism from $X^{\delta_0,\epsilon_0}$ to $Z^{\delta_2,\epsilon_2}$ has relative Picard number at most 2.
\end{prop}

\begin{proof}

Suppose first that $\omega_L$ is a vertex. 

By Corollary~\ref{cor:Bdiv2}, the $B$-stable prime divisors of $Z^{\delta_2,\epsilon_2}$ are in bijection with $(I_2^{\delta_2,\epsilon_2}\sqcup J_0)\backslash L$. 
Let $F_J^{\delta_2,\epsilon_2}$ be a face of $P^{\delta_2,\epsilon_2}$, choose $J$ such that $(\delta_2,\epsilon_2)\in\omega_J$. 
Then $L \subseteq J$ and $\omega_J=\omega_L$, so that the restriction of $A_{J\backslash L}$ to $\ker(A_L)$ is surjective. 
This implies, by Lemma~\ref{lem:Qf} applied to the matrix of lines of  $A_{I_0\backslash L}$ restricted to $\ker(A_L)$, that $Z^{\delta_2,\epsilon_2}$ is $\Qbb$-factorial.
Then its Picard number is $$|(I_2^{\delta_2,\epsilon_2}\sqcup J_0)\backslash L|-(n-|L|+2)=|I_2^{\delta_2,\epsilon_2}\sqcup J_0|-n-2.$$ 

\noindent Since $I_2^{\delta_2,\epsilon_2}\supseteq I_2^{\delta_0,\epsilon_0}$,  the relative Picard number of $X^{\delta_0,\epsilon_0}\longrightarrow Z^{\delta_2,\epsilon_2}$ is at most~2.\\

Suppose now that $\omega_L$ is not a vertex. 
By Corollary~\ref{cor:Bdiv2},  the $B$-stable prime divisors of $Z^{\delta_2,\epsilon_2}$ are in bijection 
with a subset of $(I_2^{\delta_2,\epsilon_2}\sqcup J_0)\backslash I$ of cocardinality at most~1.
 Let $F_J^{\delta_2,\epsilon_2}$ be a face of $P^{\delta_2,\epsilon_2}$, choose $J$ such that $(\delta_2,\epsilon_2)\in\omega_J$. 
 Then $I \subseteq J$. Assume that $\omega_J$ is one-dimensional. Then the restriction of $A_{J\backslash I}$ to $\ker(A_L)$ is surjective. 
 
 Assume that $\omega_J$ has dimension zero. Then $\omega_J=\omega_L$, so that the restriction of $A_{J\backslash I}$ to $\ker(A_I)$ implies the same equation for every such $J$. 

 If  the cocardinality of a subset of $(I_2^{\delta_2,\epsilon_2}\sqcup J_0)\backslash I$ is one, then this equation does not occur for $B$-stable divisors and then $Z^{\delta_2,\epsilon_2}$ is $\Qbb$-factorial.  In any case,% by Theorem~\ref{critdiv}
  the  Picard number of $Z^{\delta_2,\epsilon_2}$ is $|(I_2^{\delta_2,\epsilon_2}\sqcup J_0)\backslash I|-(n-|I|+1)-1=|I_2^{\delta_2,\epsilon_2}\sqcup J_0|-n-2$. But $I_2^{\delta_2,\epsilon_2}$ is either $I_2^{\delta_0,\epsilon_0}$ or the union of $I_2^{\delta_0,\epsilon_0}$ with  $i$ as in Lemma~\ref{lem:diviseursZ2}, so that the relative Picard number of $X^{\delta_0,\epsilon_0}\longrightarrow Z^{\delta_2,\epsilon_2}$ is~1 or ~2.\\

\end{proof}

\begin{rem}\label{rem:morphXZ}
 If $\omega_L=\{(\delta_2,\epsilon_2)\}$ is a vertex and the relative Picard number of $X^{\delta_0,\epsilon_0}\longrightarrow Z^{\delta_2,\epsilon_2}$ is not 2, 
 then there exists $i\in I_2^{\delta_2,\epsilon_2}\backslash I_2^{\delta_0,\epsilon_0}$. In particular, $\omega_L$ is a vertex of $\Omega_i$ that does not contain $(\delta_0,\epsilon_0)$. Hence, there exists $K\subset I_0$, such that $\omega_K\subset\omega_\emptyset$ is one dimensional with one extremity equals to $\omega_L$.\\ 
 If $\omega_L$ is not a vertex, such $\omega_K$ exists by definition.
 \end{rem}

\subsection{Mori polygonal chain}

We start with the following definition.

\begin{defi}
The Mori polygonal chain of the family $P^{\delta,\epsilon}$ is

$$MPC:=(\Omega_\emptyset\backslash\omega_\emptyset)\cap\{(\delta,\epsilon)\in\Qbb^2,\mid\,0\leq \delta\leq 1\}.$$
\end{defi}

Throughout this section we suppose that $(0,0)$ and $(1,0)$ are in $\omega_\emptyset$,
that for every $i \in I_0\backslash J_0$ there exist negative $\epsilon_0$ and $\epsilon_1$ such that $(0,\epsilon_0)$ and $(1,\epsilon_1)$ are in $\omega_i$
and that $C\geq 0$. Then $\{(\delta,\epsilon)\in\Qbb^2\mid\,\epsilon\leq 0\mbox{ and }0\leq \delta\leq 1\}$ is a contained in $\omega_\emptyset$. 

Note that, with the above hypothesis, the Mori polygonal chain is contained in the half plane defined by $\epsilon>0$. It is polygonal because $\Omega_\emptyset\backslash\omega_\emptyset$ is a union of one-dimensional $\Omega_I$.

We say that $\omega_I$ is a {\it segment} of the Mori polygonal chain if 
 $\dim\omega_I=1$ and the intersection of $\omega_I$ with $MPC$ is not empty. 
 If $(\delta,\epsilon)\in MPC$, then either $P^{\delta,\epsilon}$ is not of maximal dimension or there exists $i$ such that $A_i=0$ and $D_i(\delta,\epsilon)=0$. 
 
By Lemma~\ref{lem:conditionMori}, if $\omega_I$ is a segment of MPC such that $A_{I\setminus\{j\}}$ is surjective for every $j\in I$, then there is an equation $\sum_{i\in I}\lambda_i^I X_i=0$ for the lines of $A_I$ such that 
$\lambda_i^I>0$ for every $i\in I$.

\begin{defi}
Let $K\subseteq I_0$ be such that $\dim \omega_K=1$ and the codimension of the image $A_K$ is 1. Let $\sum_{i\in K} \lambda_i^K X_i=0$ be an equation for the lines of $A_K$.
The slope of $\omega_K$ is $$sl_K=\left\{
\begin{array}{ll}
\frac{\sum_{i\in K}\lambda_i^K(B'_i-B_i)}{\sum_{i\in K}\lambda_i^K C_i}&\text{if }\sum_{i\in K}\lambda_i^K C_i\neq 0,\\
\infty&\text{if } \sum_{i\in K}\lambda_i^K C_i=0.
\end{array}
\right.
$$
\end{defi}

\begin{rem} The segment $\omega_K$ is included in the line defined by $$\epsilon\left(\sum_{i\in K}\lambda_i^K C_i\right)+\delta\left(\sum_{i\in K}\lambda_i^K(B'_i-B_i)\right)+\sum_{i\in K}\lambda_i^KB_i=0$$ and $sl_K$ is the slope of this line.
\end{rem}

\begin{prop}\label{pro:rays}
Let $(\delta_2,\epsilon_2)\in U_0\cap MPC$ and let $L$ be minimal such that 
$\omega_L=\{(\delta_2,\epsilon_2)\}$. We refer to the notation of Construction~\ref{def:rel2}.
%Let $L$ be as in in the first case of Construction~\ref{def:rel2}. 
Then there is a partition $L=K^0\sqcup\ldots\sqcup K^{r+1}$, with $r\geq 0$, such that
\begin{enumerate}
\item\label{rays1} for every $s\in\{0,\ldots,r+1\}$, for every $h,k\in K^s$
we have $\displaystyle{-\frac{\lambda_h^I}{\lambda_h^J}=-\frac{\lambda_k^I}{\lambda_k^J}}=:\nu_s\in[-\infty,0]$;
for every $s<s'\in\{0,\ldots,r+1\}$, we have $\nu_s>\nu_{s'}$.
%for every $h\in K^s$ and $k\in K^{s'}$
%we have $\displaystyle{-\frac{\lambda_h^I}{\lambda_h^J}>-\frac{\lambda_k^I}{\lambda_k^J}.}$
\item\label{rays2} The set $K_s=L\setminus K^s$ is such that the codimension of the image of $A_{K_s}$ is 1 and for every $k\in K_s$ the map $A_{K_s\setminus \{k\}}$ is surjective.
If $K\subseteq L$ is such that the codimension of the image of $A_{K}$ is 1 and for every $k\in K$ the map $A_{K\setminus \{k\}}$ is surjective, then $K=K_s$ for some $s\in \{0,\ldots,r+1\}$.
 Moreover, if $(\delta_2,\epsilon_2)$ is a vertex then $K_0=I$ and $K_{r+1}=J$ and if it is not then $K_0=I$ and $K^0=L\backslash I$.
 
Note that, by Lemma~\ref{lem:inclusionclosure}, $\omega_{K_s}$ is not empty and with an extremity equals to $\omega_L$.
\item\label{rays3} The slope of $\omega_{K_s}$ is 
$$sl_{K_s}=\frac{\sum_{i\in I}\lambda_i^I(B'_i-B_i)+\nu_s\sum_{j\in J}\lambda_j^J(B'_j-B_j)}{1+d\nu_s}$$
with $d=1$ if $(\delta_2,\epsilon_2)$ is a vertex and $d=\sum_{j\in L}\lambda_j^LC_j$ otherwise.
\item\label{rays4}  Up to a rotation, the slopes decrease when $s$ increases (see picture below).
\end{enumerate}

\end{prop}
\begin{proof}
Let $K\subseteq L$ be such that the codimension of the image of $A_{K}$ is 1 and for every $k\in K$ the map $A_{K\setminus \{k\}}$ is surjective.
Then there is an equation $$(\mathcal R_K)\;\;\sum \lambda_k^K X_k=0$$ satisfied by the lines of $A_K$.

Assume that $(\delta_2,\epsilon_2)$ is a vertex and let $I,J$ be as in Construction \ref{def:rel2}.
The relation $(\mathcal R_K)$ is a linear combination of $(\mathcal R_I)$ and $(\mathcal R_J)$.
Thus there are $\mu_I,\mu_J$ such that $(\mathcal R_K)=\mu_I(\mathcal R_I)+\mu_J(\mathcal R_J)$.
We notice that if $\mu_I=0$, then by the minimality of $K$ we have $K=J$ and if $\mu_J=0$ then by the minimality of $K$ we have $K=I$.

Assume that $(\delta_2,\epsilon_2)$ is not a vertex and let $I$ be as in Construction \ref{def:rel2}.
The relation $(\mathcal R_K)$ is a linear combination of $(\mathcal R_I)$ and $(\mathcal R_L)$.
Thus there are $\mu_I,\mu_L$ such that $(\mathcal R_K)=\mu_I(\mathcal R_I)+\mu_L(\mathcal R_L)$.
We notice that if $\mu_L=0$ then by minimality of $K$ we have $K=I$ and if $\mu_L\neq 0$ then $L\backslash I\subseteq K$.

\noindent If $(\delta_2,\epsilon_2)$ is not a vertex, we set $J=\{j\in L\,\mid\,\lambda_j^L\neq 0\}$. Notice that $J$ gives the relation $(\mathcal R_L)$.

If  $\mu_I\lambda_k^I+\mu_J\lambda_k^J\neq 0$ for every $k$, then
the image of $A_{L\setminus\{\ell\}}$ has the same codimension as the image of $A_{L}$, and this is a contradiction.
 Therefore the set $\{k\vert\;\mu_I\lambda_k^I+\mu_J\lambda_k^J=0 \}$
is non-empty and $K$ is the complement of this set in $L$. 

\noindent We set $N=\{-\lambda_k^I/\lambda_k^J\}=\{\nu_0>\ldots>\nu_{r+1}\}$ with the convention that $\nu_{r+1}=-\infty$.
%We set $\nu_K=\mu_I/\mu_J=-\lambda_k^I/\lambda_k^J$ for $k\in L\setminus K$.

\noindent We set $K_s=\{k\vert\;-\lambda_k^I/\lambda_k^J=\nu_s\}$ and $K_s=L\setminus K^s$.

This proves \ref{rays1} and \ref{rays2}.

\medskip

As for \ref{rays3}, the slope $sl_{K_s}$ is
$$\frac{\sum_{k\in K_s}\lambda_k^{K_s}(B'_k-B_k)}{\sum_{k\in K_s}\lambda_k^{K_s} C_k}.$$
Since up to a multiple $\lambda_k^{K_s}=\lambda_k^I+\nu_s\lambda_k^{\tilde I}$, we get the third part of the statement, with the assumption  that, if $(\delta_2,\epsilon_2)$ is a vertex, we have $\sum_{i\in I}\lambda_i^IC_i=1$ and $\sum_{j\in J}\lambda_j^JC_j=1$, and otherwise $\sum_{i\in I}\lambda_i^IC_i=1$ and  $\sum_{j\in L}\lambda_j^LC_j=1$ or 0.

\medskip

Set $$
\begin{array}{l}
a=\sum_{i\in I}\lambda_i^I(B'_i-B_i)=sl_I\\
\\
b=
\sum_{j\in J}\lambda_j^J(B'_j-B_j)=sl_J \,\mbox{ and }\\
\\
d=\left\{
\begin{array}{l}
1\mbox{ if }(\delta_2,\epsilon_2)\mbox{ is a vertex  and}\\
 \sum_{j\in L}\lambda_j^LC_j\mbox{ (which is 1 or 0) otherwise}.
\end{array}
\right.
\end{array}
$$

To prove \ref{rays4}, set $\displaystyle f(x)=\dfrac{a+bx}{1+dx}$.
Note that $sl_{K_s}=f(\nu_s)$. And apply a rotation $\frac{1}{1+a^2}\left(\begin{array}{cc}a & -1\\1 & a \end{array}\right)$ in order to replace $sl_I=a$ by $+\infty$ (vertical direction). This way, the slopes $sl_K$ are replaced by relative slopes $relsl_K:=\dfrac{a \,sl_K+1}{a- sl_K}$. Set $g(x)=\dfrac{a^2+1+x(ab+d)}{x(ad-b)}$ so that $relsl_{K_s}=g(\nu_s)$. 

We now prove that  $ad-b<0$. By Construction \ref{def:rel2}, if $(\delta_2,\epsilon_2)$ is a vertex we have
$\omega_I\subseteq\{\delta<\delta_2\}$ and $\omega_J\subseteq\{\delta>\delta_2\}$.
 By the convexity of $\Omega_{\emptyset}$ the slope $sl_I$ is smaller that $sl_J$, so that $ad-b<0$.
  If $(\delta_2,\epsilon_2)$ is not a vertex and $d=1$, similarily we have $sl_I< sl_{J}=\frac{b}{d}$ , so that $ad-b<0$. And if $d=0$, we have $b>0$ (still by Construction \ref{def:rel2}).
  
In all cases, since $g$ is increasing, $relsl_{K_s}$ is decreasing from $+\infty$ ($s=0$ and $\nu_s=0$) to $\frac{ab+d}{ad-b}$ ($\nu_s=-\infty$).

%where $$f(x)=\dfrac{a+bx}{1+x}=b+\dfrac{a-b}{1+x},\,\, a=\sum_{i\in I}\lambda_i^I(B'_i-B_i)=sl_I\,\mbox{ and }\,b=\sum_{j\in J}\lambda_j^J(B'_j-B_j)=sl_J.$$ 

%Then the function $f$ is increasing on $[-\infty,-1[$ and on $]-1,0]$. This implies that, as $s$ increases, $sl_{K_s}$ decreases.% from $sl_I$ to some slope that can be $-\infty$ and then decreases to $sl_J$.

%The computation if $(\delta_2,\epsilon_2)$ is not a vertex is analogous.
\end{proof}
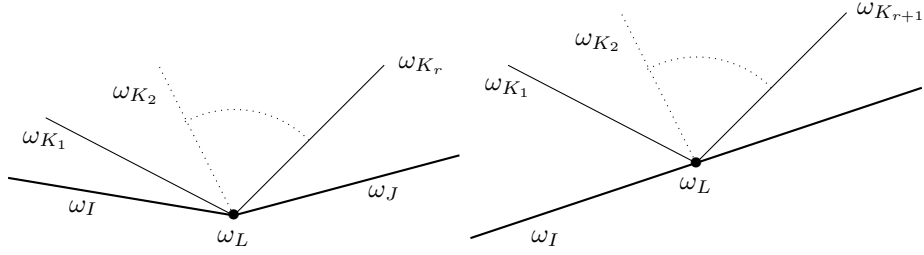
\begin{center}
\begin{tikzpicture}
\draw (0,0) node {$\bullet$};
\draw (0,-0.3) node {$\omega_L$};
\draw (-2,0.1) node {$\omega_I$};
\draw (2,0.3) node {$\omega_J$};
\draw (-2.5,1) node {$\omega_{K_1}$};
\draw (-1.3,1.6) node {$\omega_{K_2}$};
\draw (2.5,2) node {$\omega_{K_r}$};
\draw[thick] (0,0)--(-3,0.5);
\draw[thick] (0,0)--(3,0.8);
\draw (0,0)--(2,2);\draw (0,0)--(-2.5,1.3);
\draw[dotted] (0,0)--(-1,2);
\draw[dotted] (1,1) arc (45:120:1.4);
\end{tikzpicture}
\begin{tikzpicture}
\draw (0,0) node {$\bullet$};
\draw (0,-0.3) node {$\omega_L$};
\draw (-2,-1) node {$\omega_I$};
\draw (-2.5,1) node {$\omega_{K_1}$};
\draw (-1.3,1.6) node {$\omega_{K_2}$};
\draw (2.6,2) node {$\omega_{K_{r+1}}$};
\draw[thick] (-3,-1)--(3,1);
\draw (0,0)--(2,2);
\draw (0,0)--(-2.5,1.3);
\draw[dotted] (0,0)--(-1,2);
\draw[dotted] (1,1) arc (45:120:1.4);
\end{tikzpicture}
\end{center}

\section{Proof of the main theorem}
Let $G$ be a reductive group and $H\subseteq G$ be a horospherical subgroup.
We choose a basis of $M_{\mathbb Q}$ so that $M_{\mathbb Q}\cong\mathbb Q^n$.
We fix a horospherical embedding $Z$ of $G/H$ and set 
$p=r+|S\backslash R|$. 
Let $A$ be the matrix associated to the linear map $\varphi(m)=(\langle m,x_i \rangle_{i=1\ldots r},\;  \langle m,\alpha_M^{\vee} \rangle_{\alpha\in S\backslash R})$.
Denote by $J_0\subseteq I_0$ the set of indices $S\backslash R$.

Let $B=(-d_1,\ldots, -d_r, -d_{\alpha})$ and $B'=(-d'_1,\ldots, -d'_r, -d'_{\alpha})$
be such that $\sum d_i Z_i+\sum d_{\alpha} D_{\alpha}$ and $\sum d'_i Z_i+\sum d'_{\alpha} D_{\alpha}$ are ample divisors.
Let $-K_Z=\sum c_i Z_i+\sum c_{\alpha} D_{\alpha}$.
Let $C=(c_1,\ldots, c_r, c_{\alpha})$.

 Suppose now that for any $i\in I_0\backslash J_0$ there exist negative $\epsilon_0$ and $\epsilon_1$ such that $(0,\epsilon_0)$ and $(1,\epsilon_1)$ are in $\omega_i$. Then for any $\delta$, 
 the intersection $\omega_i\cap\{(\delta,\epsilon)\,\mid \, \epsilon\geq 0\}$ is an open segment (possibly empty) with one extremity at $(\delta,0)$. 
 Then the family $(P^{\delta,\epsilon})_{\epsilon\in\Qbb_{\geq 0}}$ describes a HMMP.

\begin{prop}\label{prop:HMMPisMMP}
For any $\delta\in[0,1]$ in the complement of a finite set, the HMMP described by the family $(P^{\delta,\epsilon})_{\epsilon\in\Qbb_{\geq 0}}$ is an MMP.
\end{prop} 

\begin{proof}
By Theorem \ref{HMMP}, the family $(P^{\delta,\epsilon})_{\epsilon\in\Qbb_{\geq 0}}$ describes an HMMP.
Let $U''$ be the set of $\delta$ such that there is a 1-dimensional set $\omega_I$ included in $\{\delta\}\times\mathbb Q$.
Let $p_1\colon \Qbb^2\to \Qbb$ be the projection onto the first factor and let $\delta\in[0,1]\setminus p_1(U_0\cup U'_0)\cup U''$.
By Proposition \ref{prop:morphXY},
the HMMP described by the family $(Q^{\delta,\epsilon})_{\epsilon\in\Qbb_{\geq 0}}$ consists of extremal contractions. 

\end{proof}

\begin{teo}\label{thm:mainth}
Let $G$ be a connected reductive algebraic group.
Let $X$ and $Y$ be horospherical $G$-varieties which are $G$-equivariantly birational.
Assume moreover that there are Mori fibre space structures $X/S$ and $Y/T$.
Then, there is a horospherical Sarkisov program from $X/S$ to $Y/T$.

More precisely, there are a resolution of the indeterminacy $Z$ of $X\dasharrow Y$, divisors $A_X$ and $A_Y$, a 2-dimensional polyhedron $\Omega$, points
$(\delta_X,\epsilon_X)$, $(\delta_Y,\epsilon_Y)$, $(\delta_T,\epsilon_T)$, $(\delta_S,\epsilon_S)\in\Omega$ and singletons $\omega_{L_1},\ldots,\omega_{L_\ell}\subseteq\partial\Omega$ such that
\begin{enumerate}
\item every point $(\delta,\epsilon)\in\Omega$ defines the horospherical variety of pseudo-moment polytope $Q_{(1-\delta)A_X+\delta A_Y-\epsilon K_Z}$;
\item $(\delta_X,\epsilon_X)$, $(\delta_Y,\epsilon_Y)$, $(\delta_T,\epsilon_T)$, $(\delta_S,\epsilon_S)$ correspond, respectively, to $X,Y,S,T$;
\item $(\delta_S,\epsilon_S),(\delta_T,\epsilon_T)\in\partial\Omega$;
\item if we denote $\omega_{L_i}=\{(\delta_i,\epsilon_i)\}$, we have $0<\delta_1<\cdots<\delta_l<1$;
\item every $\omega_{L_i}$ defines a Sarkisov link $W_i/V_i\dasharrow W_{i+1}/V_{i+1}$ where $W_1/V_1=X/S$ and $W_{l+1}/V_{l+1}=Y/T$.
\end{enumerate}
\end{teo}

\begin{proof}
Since the birational map $X\dasharrow Y$ is $G$-equivariant, 
$X$ and $Y$ are both $G/H$-embedding with the same horospherical homogeneous space $G/H$.

By Lemma \ref{lem:exres} and Proposition \ref{pro:resMMP} there is a smooth resolution $Z$ of the indeterminacies of $X\dasharrow Y$ and there are euclidean open sets $U_X$ and $U_Y$ of
 $WDiv(Z)_\Qbb$ such that every divisor in $U_X$ (resp. $U_Y$) is ample and for every $A\in U_X$ 
 (resp. in $U_Y$)
 the Mori fiber space $X/T$ (resp. $Y/S$) is the outcome of the HMMP from $Z$ with scaling of $A$.
 
 We pick $(A_X,A_Y)\in U_X\times U_Y$ satisfying the generality conditions of \eqref{ssec:gen}.
 Since the open set determined in \eqref{ssec:gen} is Zariski open, we can find such $(A_X,A_Y)$.
 
Let
$B=(-d_1,\ldots, -d_r, (-d_{\alpha})_{\alpha\in S\backslash R})$ and $B'=(-d'_1,\ldots, -d'_r, (-d'_{\alpha})_{\alpha\in S\backslash R})$ and $C=(c_1,\ldots, c_r, (c_{\alpha})_{\alpha\in S\backslash R})$ be such that
\begin{align*}
A_X=&\sum_{i=1}^d d_i Z_i+\sum_{\alpha\in S\backslash R} d_{\alpha} D_{\alpha}\\
A_Y=&\sum_{i=1}^d d'_i Z_i+\sum_{\alpha\in S\backslash R} d'_{\alpha} D_{\alpha}\\
-K_Z=&\sum_{i=1}^d c_i Z_i+\sum_{\alpha\in S\backslash R} c_{\alpha} D_{\alpha}.
\end{align*}
 
 Let $\Omega_I$ and $\omega_I$ be the polytopes of Definition \ref{def:omega}.
 
 Thus there are two segments $\omega_I$ and $\omega_J$ of the Mori poligonal chain such that
 every $(\delta,\epsilon)$
 in a euclidean neighborhood of $\omega_I$
  satisfies $X^{\delta,\epsilon}\cong X$ and
  if $(\delta_1,\epsilon_1)\in\omega_I$ then $Y^{\delta_1,\epsilon_1}\cong T$.
The same holds for $J$.  
  
  The segments $\omega_I$ and $\omega_J$ disconnect the chain.
  Let $\mathcal C\subseteq MPC$ be such that  $\omega_I\cup\mathcal C\cup\omega_J$ is connected.
  Let $\omega_{L_1},\ldots,\omega_{L_h}$ be the points in $U_0\cap \mathcal C$.
  We notice that $U'_0\cap \partial\Omega_{\emptyset}=\emptyset$.
  Indeed, if $(\delta,\epsilon)\in\omega_I\cap\omega_J$ and $\omega_I$, $\omega_J$ are not aligned,
  then the convex hull of $\omega_I$ and $\omega_J$ is non-empty and open (and thus of dimension 2), proving that 
   $(\delta,\epsilon)\in\omega_{\emptyset}$.
  
 \noindent We prove that every $\omega_{L_i}$ describes a link.
  We write  $\omega_L$ for simplicity.
  
  Let $L=K^0\sqcup\dots\sqcup K^{r+1}$ be the partition existing by Proposition \ref{pro:rays}.
  We fix a euclidean neighborhood $\Delta$ of $\omega_L$ and for any $i\in\{0,\dots,r\}$, let $X_i$ be the horospherical variety corresponding to a point in the open set delimited by $\omega_{K_i}$ and $\omega_{K_{i+1}}$. 
  If $\omega_L$ is not a vertex, let $X_{r+1}$ be the horospherical variety corresponding to a point in the open set delimited by $\omega_{K_{r+1}}$ and $\omega_{I}$.
  Let $t$ be $r$ if $\omega_L$ is a vertex and $r+1$ if not.
We denote by $T_0$, respectively $T_{t+1}$, the horospherical variety corresponding to a point in MPC on the left, respectively on the right, of $\omega_L$.
  
  We prove first that 
  \begin{claim}\label{cla:contr}
  the only possible divisorial contractions between two varieties $X_i$ and $X_{i\pm1}$
  are $X_1\to X_0$ and $X_{t-1}\to X_{t}$.
  \end{claim}
Recall that we have Mori fibrations $X_0\to T_0$ and $X_{t}\to T_{t+1}$.
 Let $s\in\{1,\dots,r\}$. 
 First assume that $\omega_L$ is a vertex.
 Then $K_s$ is such that $K_s^-=K^0\sqcup\dots\sqcup K^{s-1}$ and $K_s^+=K^{s+1}\sqcup\dots\sqcup K^{r+1}$. 
 By Proposition~\ref{prop:morphXY}, around $\omega_{K_s}$ we have flips except if $s=1$ and $K^0=\{i\}$ with $i\in I_0\backslash J_0$  or $s=r$ and $K^{r+1}=\{i\}$ with $i\in I_0\backslash J_0$.

 Assume now that $\omega_L$ is not a vertex.
Let $s\in\{1,\dots,t\}$, then $K_s$ is such that $K_s^-=K^0_-\sqcup K^1\sqcup\dots\sqcup K^{s-1}$ and $K_s^+=K^0_+\sqcup K^{s+1}\sqcup\dots\sqcup K^{r+1}$. Since $K^0_+$ and $K^0_-$ are non-empty 
(by Notation~\ref{def:rel2}), by Proposition~\ref{prop:morphXY} around $\omega_{K_s}$ we have flips except if $s=1$ and $K^0_-=\{i\}$ with $i\in I_0\backslash J_0$  or  $s=r+1$ and  $K^0_+=\{i\}$ with $i\in I_0\backslash J_0$. 
%Note also that from $\omega_I$ (both side) to $\omega_L$ we get  two birational maps, one of the two can be divisorial (if condition of Lemma~\ref{lem:diviseursZ2} occurs). \mage{pas utile pour la preuve de Claim?? Mais a mettre dans la remarque?}
This finishes the proof of the claim.

\smallskip

Let $R$ be the variety corresponding to $\omega_L$.
Notice that we have fibrations from $T_0\to R$ and  $T_{t+1}\to R$.
There are three cases:
\begin{enumerate}
\item $R\cong T_0\cong T_{t+1}$;
\item $R\cong T_0$ or $R\cong T_{t+1}$, and we are not in case 1;
\item $R\not\cong T_0$ and $R\not\cong T_{t+1}$.
\end{enumerate}

In case 1,  by Remark~\ref{rem:morphXZ}, we have $t\geq 1$. In particular, $sl_{K_1}<sl_{K_{0}}$, or $sl_{K_t}>sl_{K_{t+1}}$ (with the convention that $K_{r+2}=I$ if $\omega_L$ is not a vertex). A priori, $sl_{K_1}$ and $sl_{K_t}$  could be $\infty$, but the next paragraph proves that it cannot happen.

\noindent If $sl_{K_1}<sl_{K_{0}}$, then the Mori fibration $X_0\dashrightarrow R$ factors through $X_0\dasharrow Y\to R$ where $Y$ is the horospherical variety corresponding to a point of 
$\omega_{K_1}$. Then the map $X_0\dasharrow Y$ is either an isomorphism or a K-negative contraction. It cannot be a K-negative contraction because $sl_{K_1}<sl_{K_{0}}$. 
Then $X_0=Y$, which implies that $X_1\to X_0$ is a divisorial contraction. Similarly, if $sl_{K_t}>sl_{K_{t+1}}$, $X_{t-1}\to X_t$ has to be a divisorial contraction.

%Indeed $X_{s+1}\dasharrow X_s$ (resp. $X_s\dasharrow X_{s+1}$) is part of an MMP by Proposition \ref{prop:HMMPisMMP}.

 %Assume that for every $i$ we have $X_i\dasharrow X_{i+1}$ is an isomorphism in codimension 1.
 
 %\noindent \ora{Assume that $r=0$. Trivial link} Then the two Mori fibre spaces $X_0\rightarrow T_0$ and $X_{r+1}\rightarrow T_{r+1}$ are the same and we have a trivial link.
 
 %If $0<i<r+1$, we have a Mori fibration $X_i\to R$ which factors through $X_i\dasharrow X_{i+1}\to R$
%and the first map being an antiflip. Therefore there is $X_i\to Y_{i+1}\to R$, with $X_i\to Y_{i+1}$ a K-positive contraction, a contradiction with $K_{X_i}$ being relatively ample.

\noindent Thus there is $i=1$ or $i=t-1$ such that $\rho(X_i)=\rho(R)+2$.
Assume it is $i=1$. By Proposition \ref{prop:morphXZ} and Claim~\ref{cla:contr}, for every $i\in\{1,\dots,t-1\}$ we have $\rho(X_i)=\rho(X_1)$ so that $\rho(X_{t-1})=\rho(X_t)+1$ and thus $X_{t-1}\to X_{t}$
is also a divisorial contraction. We have a type II link.

In case 2, assume that $R\cong T_0$. Again by Proposition \ref{prop:morphXY} the variety $R$
 is $\Qbb$-factorial.
Thus $R\not\cong T_{t+1}$ implies $\rho( T_{t+1})=\rho(R)+1$ and $\rho( X_{t})=\rho(R)+2$.
By Proposition \ref{prop:morphXZ} and Claim~\ref{cla:contr}, for every $i\in\{1,\dots,t\}$ we have $\rho(X_i)=\rho( X_{t})=\rho(R)+2=\rho(T_0)+2$.
Since $\rho(X_0)=\rho(T_0)+1$, the map $X_1\to X_0$ is a divisorial contraction and we have a type I  link. Similarly if $R\cong T_{r+1}$ we get a type III link.

In case 3 we have $\rho(T_0)=\rho(R)+1$ and $\rho(T_{t+1})=\rho(R)+1$.
Moreover $\rho(X_0)=\rho(T_0)+1$ and $\rho(X_{t})=\rho(T_{t+1})+1$.
By Proposition \ref{prop:morphXZ} and Claim~\ref{cla:contr}, for every $i\in\{0,\dots,t-1\}$ the map $X_i\dasharrow X_{i+1}$ is an isomorphism in codimension 1 and we get a type IV link. 
\end{proof}

\begin{rem}
If $\omega_L$ is not a vertex then we cannot have a link of type IV with fibrations $T_0\longrightarrow R$ and $T_{t+1}\longrightarrow R$. Indeed from $\omega_I$ (both side) to $\omega_L$ we get two birational maps, one of the two can be divisorial (if condition of Lemma~\ref{lem:diviseursZ2} occurs), but not in case 3.
\end{rem}

\appendix

\section{ Toric Example}
Set $X:=\Pbb^1\times \Pbb^1$,  $S=\Pbb^1$,  $Y$ the projective bundle $\Pbb(\mathcal{O}\oplus\mathcal{O}(2))$ and $T=\Pbb^1$.
Fans of $X$ and $Y$ are the following:
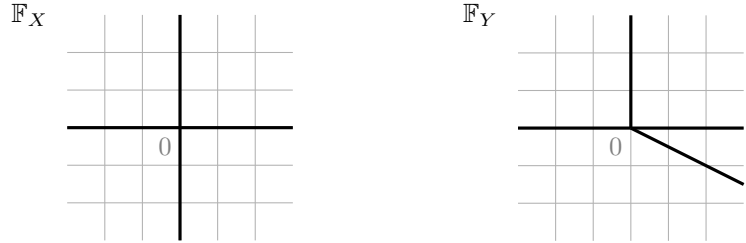
\begin{center}
\begin{tikzpicture}\node [rectangle] (a) at (0,0) {
    \begin{tikzpicture}[scale=0.5]
     \reseau2
\draw[very thick] (0,0) -- (3,0);
\draw[very thick] (0,0) -- (-3,0);
\draw[very thick] (0,0) -- (0,3);
\draw[very thick] (0,0) -- (0,-3);
 
\node at (-4,3) {$\mathbb{F}_X$};
    \end{tikzpicture}
};

 \node [rectangle] (b) at (6,0) {\begin{tikzpicture}[scale=0.5]
     \reseau2
\draw[very thick] (0,0) -- (3,0);
\draw[very thick] (0,0) -- (-3,0);
\draw[very thick] (0,0) -- (0,3);
\draw[very thick] (0,0) -- (3,-3/2);
 
\node at (-4,3) {$\mathbb{F}_Y$};
    \end{tikzpicture}
};
\end{tikzpicture}
\end{center}

Here $G=(\Cbb^*)^2$ and coincides with the Borel subgroup.

A resolution of $X$ and $Y$ is the toric variety $Z$ given by the following fan, where we index the edges as in the picture:
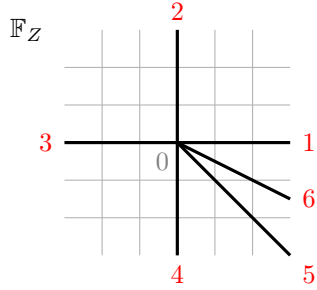
\begin{center}

    \begin{tikzpicture}[scale=0.5]
     \reseau2
\draw[very thick] (0,0) -- (3,0);
\draw[very thick] (0,0) -- (-3,0);
\draw[very thick] (0,0) -- (0,3);
\draw[very thick] (0,0) -- (0,-3);
 \draw[very thick] (0,0) -- (3,-3/2);
\draw[very thick] (0,0) -- (3,-3); 
\node at (-4,3) {$\mathbb{F}_Z$};
\node at (3.5,0) [color=red] {1};
\node at (0,3.5) [color=red] {2};
\node at (-3.5,0) [color=red] {3};
\node at (0,-3.5) [color=red] {4};
\node at (3.5,-3.5) [color=red] {5};
\node at (3.5,-3/2) [color=red] {6};
    \end{tikzpicture}
\end{center}

We let $p\colon Z\to X$ and $q\colon Z\to Y$ be the two morphisms resolving the indeterminacy of $X\dasharrow Y$.

We will consider the following matrix $A:=\left(\begin{array}{cc}
1 & 0\\
0 & 1\\
-1 & 0\\
0 & -1\\
1 & -1\\
2 & -1
\end{array}\right)$.

A $G$-stable divisor of $Z$ is of the form $D=\sum_{i=1}^6d_iD_i$. Note that $D_1$, $D_2$, $D_3$ and $D_4$ are also prime $G$-stable divisor of $X$ and $D_1$, $D_2$, $D_3$ and $D_6$ are also prime $G$-stable divisor of $Y$.

\noindent Let $D=\sum_{i=1}^6d_iD_i$ be a divisor on Z and $B$ be the column matrix associated to $(-d_1,\ldots,-d_6)$.
The divisor $D$ is ample if and only if for any $I\subseteq\{\{1,2\},\{2,3\},\{3,4\},\{4,5\},\{5,6\},\{1,6\} \}$ we have $A_{\bar{I}}(A_I^{-1}B_I)>B_I$.

This inequality system reduces to the following system, thus defining the ample cone of $Z$:

\begin{equation}\label{QZ}
 d_1+d_5>d_6,\,\,
d_2+d_6>2d_1,\,\,
 d_3+d_5>d_4,\,\,
 d_4+d_6>2d_5.
\end{equation}

The polytope $Q_D$ is the pseudo-moment polytope of $(X,p_*D)$ if and only if for every $I\subseteq\{\{1,2\},\{2,3\},\{3,4\},\{1,4\} \}$ we have $A_{\bar{I}}(A_I^{-1}B_I)>B_I$.
Equivalently, if and only if the following inequalities are satisfied 
\begin{equation}\label{QX}
 d_1+d_3>0,\,\, d_2+d_4>0,\,\, d_5> d_1+d_4,\,\, d_6>2d_1+d_4.
\end{equation}
Note that the first two inequalities correspond to the condition for $p_*D$ to be ample and the last two other correspond to the fact that the lines 5 and 6 are not necessary to define $Q_D$.

Similarly, $Q_D$ is the pseudo-moment polytope of $(Y,q_*D)$ if and only if for every $I\subseteq\{\{1,2\},\{2,3\},\{3,6\},\{1,6\} \}$ we have $A_{\bar{I}}(A_I^{-1}B_I)>B_I$.
Equivalently, if and only if the following inequalities are satisfied 
\begin{equation}\label{QY}
 d_1+d_3>0,\,\, d_2+d_6>2d_1,\,\, d_4> 2d_3+d_6,\,\, d_5>d_3+d_6.
\end{equation}
Note that the first two inequalities correspond to the condition for $q_*D$ to be ample and the last two other correspond to the fact that the lines 4 and 5 are not necessary to define $Q_D$.

\noindent To run the horospherical Sarkisov program, we have to choose $D$, $D'$ such that
\begin{enumerate}
\item\label{extorHS1} $Q_D$ and $Q_{D'}$  are pseudo-moment polytopes of $(X,p_*D)$ and $(Y,q_*D')$ respectively;
\item\label{extorHS2} there exist  $\epsilon<0$ and $\epsilon'<0$ such that $D+\epsilon K_Z$ and $D'+\epsilon' K_Z$ are ample over $Z$.
\item\label{extorHS3} the HMMP with scaling of $D$ (resp. $D'$) ends with $X/S$ (resp. $Y/T$).
\end{enumerate}

\noindent Condition \ref{extorHS1} is given by \eqref{QX} and \eqref{QY}.\\

\noindent For condition \ref{extorHS2}, note that $-K_Z=\sum_{i=1}^6D_i$ and thus  $C$ is the column matrix associated to $(1,\ldots,1)$.
Then $D+\epsilon K_Z$  is ample over $Z$ if and only if $$ d_1+d_5>d_6+\epsilon,\quad d_2+d_6>2d_1,\quad d_3+d_5>d_4+\epsilon\,\mbox{ and }\, d_4+d_6>2d_5.$$ 
Hence, there exists $\epsilon<0$ such that $D+\epsilon K_Z$  is ample over $Z$ if and only if $d_2+d_6>2d_1$ and $d_4+d_6>2d_5$. Similarly, there exists $\epsilon'<0$ such that $D'+\epsilon' K_Z$  is ample over $Z$ if and only if $d_2'+d_6'>2d_1'$ and $d_4'+d_6'>2d_5'$.\\

\noindent As for condition \ref{extorHS3}, 
note first that the HMMP from any ample divisor of $Y$ ends with $Y\to S$, because only one extremal ray of $NE(Y)$ is $K$-negative. If $d_1+d_3<d_2+d_4$ (resp. $d_1+d_3>d_2+d_4$) the HMMP from $D$ gives the first (resp. the second) projection $\Pbb^1\times\Pbb^1\to\Pbb^1$

Since $D$ and $D'$ are given up to linearly equivalence, we can choose them such that $d_1=d_2=d_1'=d_2'=0$. In particular, $Q_D$ and $Q_{D'}$ have a "south-west" vertex at 0 (but the same is not true for $Q^{\epsilon,\lambda}$ for $\epsilon\neq 0$). The conditions on $D$ and $D'$ are then

$$ d_3>0,\quad d_6>d_4>0\,\mbox{ and }\, d_4<d_5<\frac{1}{2}(d_4+d_6),\, \mbox{  with either } d_3<d_4\mbox{ or }d_3>d_4;
$$

$$ d_3'>0,\quad d_6'>0,\quad d_4'>2d_3'+d_6'\,\mbox{ and }\, d_3'+d_6'<d_5'<\frac{1}{2}(d_4'+d_6').
$$

For example, we can choose $B$ and $B'$ to be the column matrices associated to $(0,0,-1,-2,-5/2,-4)$ and $(0,0,-1,-6,-7/2,-2)$ respectively.

In Figure~\ref{fig:toric1}, we draw $\Omega_{\emptyset}$ for this choice of $B,B'$.
\begin{figure}
\caption{Scheme of $\Omega_\emptyset$ for the first toric example.}\label{fig:toric1}
\begin{center}
\begin{tikzpicture}[scale=3]
\draw (0,0) node {$\circ$};
\draw (-0.3,0) node {$(0,0)$};
\draw (0,-0.5) node {$\circ$};
\draw (-0.3,-0.5) node {$(0,1/2)$};
\draw (1,-0.5) node {$\bullet$};
\draw (1,-0.6) node {$\omega_{L_1}$};
\draw (2,-0.5) node {$\bullet$};
\draw (2,-0.6) node {$\omega_{L_2}$};
\draw (3,-0.5) node {$\circ$};
\draw (3.3,-0.5) node {$(1,1/2)$};
\draw (3,0) node {$\circ$};
\draw (3.3,0) node {$(1,0)$};
\draw[color=red] (3,-0.5)--(0,-1/2);
\draw (1,-0.5)--(0,1/2);
\draw (0.45,1/4) node {$\omega_{1,4,5}$};
\draw (1,-0.5)--(3,3/2);
\draw (1.2,-1/8) node {$\omega_{3,4,5}$};
\draw (2,-0.5)--(0,3/2);
\draw (1.95,-1/4) node {$\omega_{1,5,6}$};
\draw (2,-0.5)--(3,1/2);
\draw (2.6,-1/8) node {$\omega_{3,5,6}$};
\draw (0,-1/4) node {$X=\Pbb^1\times\Pbb^1$};
\draw (3,-1/4) node {$Y=\mathbb{F}_2$};
\draw (3/2,1) node {$Z$};
\draw (0.9,1/8) node {$\tilde{X}$};
\draw (3/2,-1/4) node {$\mathbb{F}_1$};
\draw (2.1,1/8) node {$\tilde{Y}$};
\draw[color=red] (1/2,-0.6) node {$\Pbb^1$};
\draw[color=red] (3/2,-0.6) node {$\Pbb^1$};
\draw[color=red] (5/2,-0.6) node {$\Pbb^1$};
\draw[color=red] (1,-0.7) node {$\Pbb^1$};
\draw[color=red] (2,-0.7) node {$\Pbb^1$};
\end{tikzpicture}
\end{center}
\end{figure}
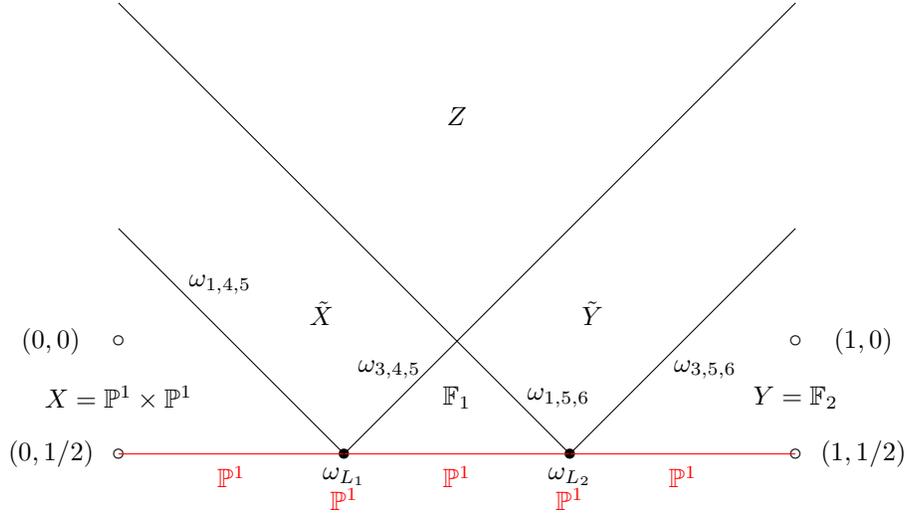
We denote by $\tilde{X}$ is the blow-up of $X$ at a point and $\tilde{Y}$ is the blow-up of $\mathbb{F}_1$ at a point in the $(-1)$-section.
Note that $L_1=\{1,3,4,5\}$ and $L_2=\{1,3,5,6\}$. Also note that $(1/2,0)$ is in $U_0'$.

In the above example, we have $d_3<d_4$. If we choose $B$ to be the column matrices associated to $(0,0,-6,-1,-3/2,-3)$
 (and same $B'$ as before), then $d_3>d_4$. Thus we obtain the horospherical Sarkisov program illustrated in Figure~\ref{fig:toric2}. Note that the first Sarkisov link 
is the type IV link from the first to the second projection $\Pbb^1\times\Pbb^1\longrightarrow\Pbb^1$.
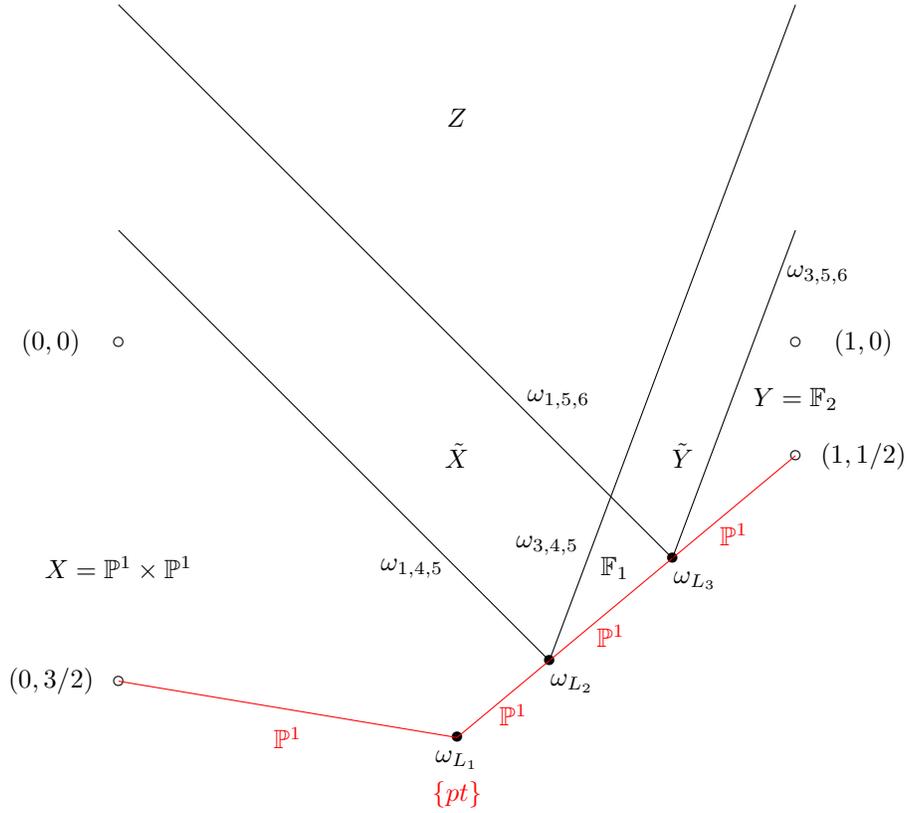
\begin{figure}
\caption{Scheme of $\Omega_\emptyset$ for the second toric example.}\label{fig:toric2}
\begin{center}
\begin{tikzpicture}[scale=3]
\draw (0,0) node {$\circ$};
\draw (-0.3,0) node {$(0,0)$};
\draw (0,-1.5) node {$\circ$};
\draw (-0.3,-1.5) node {$(0,3/2)$};
\draw (3/2,-7/4) node {$\bullet$};
\draw (3/2,-7/4-0.1) node {$\omega_{L_1}$};
\draw[color=red] (0,-1.5)--(3/2,-7/4);

\draw (21/11,-31/22) node {$\bullet$};
\draw (21/11+0.1,-31/22-0.1) node {$\omega_{L_2}$};
\draw (27/11,-21/22) node {$\bullet$};
\draw (27/11+0.1,-21/22-0.1) node {$\omega_{L_3}$};
\draw (3,-0.5) node {$\circ$};
\draw (3.3,-0.5) node {$(1,1/2)$};
\draw (3,0) node {$\circ$};
\draw (3.3,0) node {$(1,0)$};
\draw[color=red] (3/2,-7/4)--(3,-1/2);
\draw (21/11,-31/22)--(0,1/2);
\draw (1.3,-1) node {$\omega_{1,4,5}$};
\draw (21/11,-31/22)--(3,3/2);
\draw (1.9,-0.9) node {$\omega_{3,4,5}$};
\draw (27/11,-21/22)--(0,3/2);
\draw (1.95,-1/4) node {$\omega_{1,5,6}$};
\draw (27/11,-21/22)--(3,1/2);
\draw (3.1,0.3) node {$\omega_{3,5,6}$};
\draw (0,-1) node {$X=\Pbb^1\times\Pbb^1$};
\draw (3,-1/4) node {$Y=\mathbb{F}_2$};
\draw (3/2,1) node {$Z$};
\draw (1.5,-0.5) node {$\tilde{X}$};
\draw (2.2,-1) node {$\mathbb{F}_1$};
\draw (2.5,-0.5) node {$\tilde{Y}$};
\draw[color=red] (3/4,-1.75) node {$\Pbb^1$};
\draw[color=red] (7/4,-1.65) node {$\Pbb^1$};
%\draw[color=red] (21/11+0.1,-1.5) node {$\Pbb^1$};
\draw[color=red] (24/11,-1.3) node {$\Pbb^1$};
%\draw[color=red] (27/11+0.1,-1.1) node {$\Pbb^1$};
\draw[color=red] (30/11,-0.85) node {$\Pbb^1$};
\draw[color=red] (3/2,-2) node {$\{pt\}$};
\end{tikzpicture}
\end{center}
\end{figure}
The fifth line of $A$ does not correspond to a divisor of $X$ or $Y$, but is added in order to get $Z$ to be smooth. The strategy given in the paper also works by chosing a (not necessarily smooth) resolution of indeterminacies $Z'$, but non-terminal varieties can appear.
In this example, if we forget the fifth line of $A$, we can obtain a unique  Sarkisov link of type II with $X'=Y'=Z'$.

\section{A rank one horospherical example}

Choose a connected algebraic group $G$ such that $$S\backslash R=\{\alpha_1,\alpha_2,\alpha_3,\alpha_4,\alpha_5\}$$ with the notation of section~\ref{sec:horo}. Let $\chi=\varpi_1+\varpi_2-\varpi_3-\varpi_4$ and $H\subset P$ be the kernel in $P$ of the character $\chi$. In particular, $M=\Zbb\chi$. We identify $\chi$ with 1 in $M\simeq \Zbb$. 
Then $N\simeq \Zbb$, $\alpha_{1M}^\vee=\alpha_{2M}^\vee=1$, $\alpha_{3M}^\vee=\alpha_{4M}^\vee=-1$ and  $\alpha_{5M}^\vee=0$.

Since $G/H$ is of rank one, projective $G/H$-embedding are  uniquely determined by their colors, that is, by subsets of $\{\alpha_1,\alpha_2,\alpha_3,\alpha_4\}$. Denote by $X_K$ the projective $G/H$-embedding  such that $\mathcal{F}_{X_K}=\{\alpha_k\mid\,k\in K\}$. In particular, $Z:=X_\emptyset$ is a common resolution of all $X_K$.

Denote by $P_K$ the parabolic subgroup of $G$ containing $B$ whose the set of simple roots is $\{\alpha_k\mid\,k\in K\}\cup R$.

Let $H'\subset P_5$ be the kernel in $P_5$ of the character $\chi$. It is a horospherical subgroup associated to the same lattice $M$. Denote by $Y_K$ the projective $G/H'$-embedding $Y$ such that $\mathcal{F}_{Y_K}=\{\alpha_k\mid\,k\in K\}$.

Suppose that the coefficients $a_{\alpha_i}$'s are respectively 2, 3, 2, 3 and 2 for any $i=1,\dots,5$. We consider $A$ to be the column matrix associated to $$(\alpha_{1M}^\vee,\alpha_{2M}^\vee,\alpha_{3M}^\vee,\alpha_{4M}^\vee,\alpha_{5M}^\vee, 1,-1)=(1,1,-1,-1,0,1,-1).$$ Then $C$ is the column matrix associated to $(2,3,2,3,2,1,1)$. 

Let $B$ and $B'$ be the column matrices associated to $-(0,1,7,6,5,2,2)$ and $-(2,0,6,7,1,3,7)$ respectively. Then the associated horospherical Sarkisov program 
gives a link of type IV\textsubscript m, followed by a link of type III, a link of type IV\textsubscript m and a last one of type IV\textsubscript s.

We then obtain the horospherical Sarkisov program illustrated in Figure~\ref{fig:horo} (we only draw $\Omega_\emptyset$ for $\epsilon\geq 0$ and $0\leq \delta\leq 1$). In red we draw the one dimensional $\omega_I$'s giving fibrations, in blue the ones giving divisorial contractions and in black the ones giving flips.\\

\begin{figure}
\caption{Scheme of $\Omega_\emptyset$ for the rank one horospherical  example.}\label{fig:horo}
\begin{center}
\begin{tikzpicture}[scale=3]
\draw (0,0) node {$\circ$};
\draw (-0.3,0) node {$(0,0)$};
\draw (3,0) node {$\circ$};
\draw (3.3,0) node {$(1,0)$};
\draw (0,-2/3) node {$\circ$};
\draw (-0.3,-2/3) node {$(0,2/3)$};
\draw (3*1/16,-13/16) node {$\bullet$};
\draw (3*1/16,-13/16-0.1) node {$\omega_{L_1}$};
\draw[color=red] (0,-2/3)--(3*1/16,-13/16);

\draw (3*5/12,-7/6) node {$\bullet$};
\draw (3*5/12,-7/6-0.1) node {$\omega_{L_2}$};
\draw[color=red] (3*5/12,-7/6)--(3*1/16,-13/16);

\draw (3*2/3,-7/6) node {$\bullet$};
\draw (3*2/3,-7/6-0.1) node {$\omega_{L_3}$};
\draw[color=red] (3*5/12,-7/6)--(3*2/3,-7/6);

\draw (3*7/8,-3/4) node {$\bullet$};
\draw (3*7/8+0.05,-3/4-0.1) node {$\omega_{L_4}$};
\draw[color=red] (3,-1/2)--(3*2/3,-7/6);
\draw (3,-0.5) node {$\circ$};
\draw (3.3,-0.5) node {$(1,1/2)$};

\draw (3*1/16,-13/16)--(1,0);
\draw[color=blue] (3*5/12,-7/6)--(3*3/4,-1/2);
\draw (3*3/4,-1/2) node {$\bullet$};
\draw[color=blue] (3*5/6,0)--(3*3/4,-1/2);

\draw (3*3/4,-1/2)--(3*7/8,-3/4);

\draw (0.7,-0.15) node {$\omega_{1,2}$};

\draw[color=blue] (1.6,-0.8) node {$\omega_{4,7}$};

\draw[color=blue] (2.25,-1/4) node {$\omega_{3,7}$};

\draw (2.55,-0.6) node {$\omega_{3,4}$};

\draw (0,-1/3) node {$X=X_1$};
\draw (3,-1/4) node {$Y=X_{2,3}$};
\draw (1.5,-0.5) node {$X_{2}$};
\draw (2,-0.9) node {$X_{2,4}$};

\draw[color=red] (-0.1,-0.9) node {$G/P_1$};
\draw[color=red] (0.7,-1.1) node {$G/P_2$};
\draw[color=red] (1.6,-1.3) node {$G/P_{2,4}$};
\draw[color=red] (2.45,-1.05) node {$Y_{2,4}$};
\draw[color=red] (2.95,-0.7) node {$Y_{2,3}$};
%\draw[color=red] (7/4,-1.65) node {$\Pbb^1$};
%\draw[color=red] (21/11+0.1,-1.5) node {$\Pbb^1$};
%\draw[color=red] (24/11,-1.3) node {$\Pbb^1$};
%\draw[color=red] (27/11+0.1,-1.1) node {$\Pbb^1$};
%\draw[color=red] (30/11,-0.85) node {$\Pbb^1$};
%\draw[color=red] (3/2,-2) node {$\{pt\}$};
\end{tikzpicture}
\end{center}
\end{figure}
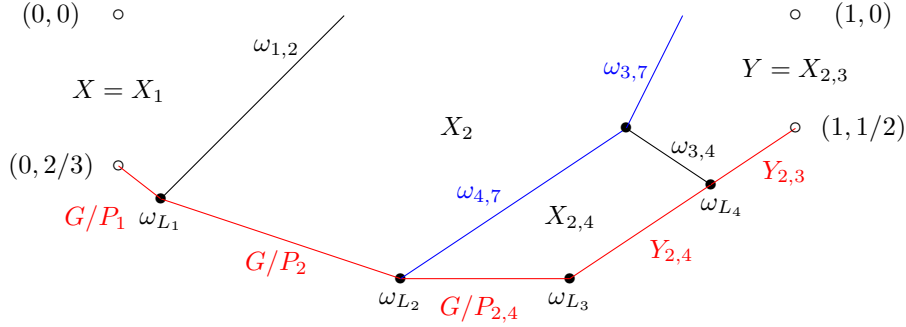
%Here is a scheme of the polytopes corresponding to points in $\Omega_{\emptyset}$ for this choice of $B,B'$. 

%\begin{tabular}{cccc}
%\includegraphics[width=0.3\textwidth]{extor0-m2.png}&&&\includegraphics[width=0.3\textwidth]{extor1-m2.png}\\
%\includegraphics[width=0.2\textwidth]{extor0-0.png}&\includegraphics[width=0.2\textwidth]{extor05-0.png}&\includegraphics[width=0.2\textwidth]{extor07-0.png}&\includegraphics[width=0.2\textwidth]{extor1-0.png}\\
%&\includegraphics[width=0.2\textwidth]{extor03-02.png}&\includegraphics[width=0.2\textwidth]{extor05-02.png}&\\
%\includegraphics[width=0.2\textwidth]{extor0-05.png}&&&\includegraphics[width=0.2\textwidth]{extor1-05.png}
%\end{tabular}

\bibliographystyle{amsalpha}
\bibliography{HoroSarkisovProgram.bib}

\end{document}

%% file: horovar.tex
\section{Horospherical varieties}\label{sec:horo}

We begin by recalling very briefly the Luna-Vust theory of horospherical embeddings and by setting the notation used in the rest of the paper. For more details on horospherical varieties, we refer the reader to \cite{Fanohoro}, and for basic results on Luna-Vust theory of spherical embeddings, we refer to \cite{knop}.

\noindent Let $G$ be connected reductive algebraic group.
A closed subgroup $H$ of $G$ is said to be {\it horospherical} if it contains the unipotent radical $U$ of a Borel subgroup $B$ of $G$. This is equivalent to say that (\cite[Prop. and Rem. 2.2]{Fanohoro}), there exists a parabolic subgroup $P$ containing $H$ such that the map $G/H\longrightarrow G/P$ is a torus fibration; or to say that there exists a parabolic subgroup $P$ containing $H$ such that $H$ is the kernel of finitely many characters of $P$.

\noindent Note that $B\subset P=N_G(H)$. 
We also fix a maximal torus $T$ of $B$. Then we denote by $S$ the set of simple roots of $(G,B,T)$. Also denote by $R$ the subset of $S$ of simple roots of $P$.
Let $X(T)$ (resp. $X(T)^+$) be the lattice of characters of $T$ (resp. the set of dominant characters). Similarly, we define $X(P)$ and $X(P)^+=X(P)\cap X(T)^+$. Note that the lattice $X(P)$ and the dominant chamber $X(P)^+$ are generated by the fundamental weights $\varpi_\alpha$ with $\alpha\in S\backslash R$ and the weights of the center of $G$.

We denote by $M$ the sublattice of $X(P)$ consisting of characters of $P$ vanishing on $H$. The rank of $M$ is called the \textit{rank of} $G/H$ and denoted by $n$. Let $N:=\Hom_\Zbb(M,\Zbb)$.

For any free lattice $\mathbb{L}$, we denote by $\mathbb{L}_\Qbb$ the $\Qbb$-vector space $\mathbb{L}\otimes_\Zbb\Qbb$. 

For any simple root $\alpha\in S\backslash R$, the restriction of the coroot $\alpha^\vee$ to $M$ is a point of $N$, which we denote by $\alpha^\vee_M$. 

\begin{defi}
A $G/H$-embedding is a couple $(X,x)$, where $X$ is a normal algebraic $G$-variety and $x$ a point of $X$ such that $G\cdot x$ is open in $X$ and isomorphic to $G/H$.

\noindent The variety $X$ is called a \textit{horospherical variety}.
\end{defi}

\noindent By abuse of notation, we often forget the point $x$, so that we call $X$ a $G/H$-embedding.
But there are several non-isomorphic $G/H$-embeddings $(X,x)$ for the same horospherical variety $X$. Two points $x_1$ and $x_2$ differ by an element of the torus $P/H$, which acts on the right on $G/H$. 
Similarly to toric varieties (which are $(\Cbb^*)^n$-embeddings with the above defintion), $G/H$-embeddings are classified by colored fans in $N_\Qbb$.
For example, for the toric variety $\Pbb^2$ we have different non-isomorphic $(\Cbb^*)^2$-embeddings, whose fans are the same up to the action of $\operatorname{SL_2}(\Zbb)$. In this paper, since we define horospherical varieties by their colored fans, or equivalently by some of their moment polytopes, we are
implicitly fixing a $G/H$-embedding up to isomorphism.

\begin{defi}\label{defi:fan}
\begin{enumerate}
\item A colored cone of $N_\Qbb$ is an couple $(\mathcal{C}, \mathcal{F})$
 where $\mathcal{C}$ is a convex cone of $N_\Qbb$ and $\mathcal{F}$ is a set of colors (called the set of colors of the colored cone), such that
\begin{enumerate}[(i)]
\item $\mathcal{C}$ is generated by finitely many elements of $N$ and contains $\{\alpha^\vee_M\,\mid\,\alpha\in\mathcal{F}\}$,
\item $\mathcal{C}$ does not contain any line and
$\mathcal{F}$ does not contain any $\alpha$ such that $\alpha^\vee_M$ is zero.
\end{enumerate}
\item A colored face of a colored cone $(\mathcal{C}, \mathcal{F})$ is a couple $(\mathcal{C}', \mathcal{F}')$ such that $\mathcal{C}'$ is a face of $\mathcal{C}$ and $\mathcal{F}'$ is the set of $\alpha\in\mathcal{F}$ satisfying $\alpha^\vee_M\in\mathcal{C}'$.
A colored fan is a finite set $\mathbb{F}$ of colored cones such that
\begin{enumerate}[(i)]
\item any colored face of a colored cone of $\mathbb{F}$ is in $\mathbb{F}$,
\item and any element of $N_\Qbb$ is in the interior of at most one colored cone of $\mathbb{F}$.
\end{enumerate}\end{enumerate}
\end{defi}

The main result of Luna-Vust Theory of spherical embeddings is the one-to-one correspondence between colored fans and isomorphic classes of $G/H$-embeddings (see for example \cite{knop}). It generalizes the classification of toric varieties in terms of fans, case where $G=(\Cbb^*)^n$ and $H=\{1\}$. We will rewrite this result in Section~\ref{sec:divpoly} for projective horospherical varieties in terms of polytopes and describe explicitly the correspondence.

%\begin{teo}\label{th:Luna-Vust}\marginpar{\ora{enlever?}}
%There is an explicit one-to-one correspondence between colored fans and isomorphic classes of $G/H$-embeddings.
%
%Complete $G/H$-embeddings correspond to complete fans i.e. to fans such that $N_\Qbb$ is the union of its colored cones.
%\end{teo}

%If $G=(\Cbb^*)^n$ and $H=\{1\}$, we recover the well-known classification of toric varieties.

\noindent If $X$ is a $G/H$-embedding, we denote by $\mathbb{F}_X$ the {\it colored fan} of $X$ in $N_\Qbb$ and we denote by $\mathcal{F}_X$ the subset $\cup_{(\mathcal{C},\mathcal{F})\in\mathbb{F}_X}\mathcal{F}$ of $S\backslash R$, which we call the {\it set of colors} of $X$.

%A consequence of Theorem~\ref{th:Luna-Vust} is the following description of $G$-orbits of $G/H$-embeddings.

%\begin{prop}\label{prop:bijorbits} The set of $G$-orbits of a \marginpar{\ora{enlever?}} $G/H$-embedding $X$ is naturally in bijection with the set of colored cones of $\mathbb{F}_X$ (reversing usual orders). In particular, the $G$-stable irreducible divisors correspond to the colored edges of $\mathbb{F}_X$ of the form $(\mathcal{C},\varnothing)$.
%\end{prop}

\smallskip

We now recall the description of divisors of horospherical varieties.

\noindent We denote by $X_1,\dots,X_r$ the $G$-stable irreducible divisors of a $G/H$-embedding $X$. For any $i\in\{1,\dots,r\}$, we denote by $x_i$ the primitive element in $N$ of the colored edge associated to $X_i$. 
The $B$-stable and not $G$-stable irreducible divisors of a $G/H$-embedding $X$ are the closures in $X$ of $B$-stable irreducible divisors of $G/H$, which are the inverse images by the torus fibration $G/H\longrightarrow G/P$ of the 
Schubert divisors of the flag variety $G/P$.
% Recall that the Schubert divisors of $G/P$ are indexed  \marginpar{\ora{enlever?}} by the subset of simple roots $S\backslash R$ and of the form $\overline{Bw_0s_\alpha P/P}$, where $w_0$ is the longest element in the Weyl group of $(G,T)$ and $s_\alpha$ is the simple reflection associated to a simple root $\alpha$. 
The $B$-stable irreducible divisors of $G/H$ are indexed by simple roots of $S\backslash R$, we write them $D_\alpha$ with $\alpha\in S\backslash R$.
 
 %of the form  $\overline{Bw_0s_\alpha P/H}$, we denote them by $D_\alpha$ for any $\alpha\in S\backslash R$. \ora{Si on enleve la partie en haut, reformuler.}

We can now recall the characterization of Cartier, $\Qbb$-Cartier and ample divisors of horospherical varieties due to M.~Brion in the more general case of spherical varieties (\cite{briondiv}). This will permit to define a polytope associated to a divisor of a horospherical variety.

\begin{teo}(Section 3.3, \cite{briondiv}) \label{th:divcrit}
Let $G/H$ be a horospherical homogeneous space. Let $X$ be a $G/H$-embbeding.  
Then every divisor of $X$ is equivalent to a linear combination of $X_1,\dots,X_r$ and $D_\alpha$ with $\alpha\in S\backslash R$.
Now, let $D=\sum_{i=1}^r a_i X_i +\sum_{\alpha\in S\backslash R} a_\alpha D_\alpha$ be a $\Qbb$-divisor of $X$.
\begin{enumerate}
\item $D$ is $\Qbb$-Cartier if and only if there exists a piecewise linear function $h_D$, linear on each colored cone of $\mathbb{F}_X$, such that for any $i\in\{1,\dots,r\}$, $h_D(x_i)=a_i$ and for any $\alpha\in\mathcal{F}_X$, $h_D(\alpha^\vee_M)=a_\alpha$.
\item Suppose that $D$ is a divisor (i.e. $a_1,\dots,a_r$ and the $a_\alpha$ with $\alpha\in S\backslash R$ are in $\Zbb$). Then $D$ is Cartier if moreover, for any colored cone $(\mathcal{C},\mathcal{F})$ of $\mathbb{F}_X$, the linear function $(h_D)_{|\mathcal{C}}$, can be define as an element of $M$ (instead of $M_\Qbb$ for $\Qbb$-Cartier divisors).
\item Suppose that $D$ is $\Qbb$-Cartier. Then $D$ is ample, resp. nef if and only if the piecewise linear function $h_D$ is strictly convex, resp. convex, and for any $\alpha\in (S\backslash R)\backslash\mathcal{F}_X$, we have $h_D(\alpha^\vee_M)<a_\alpha$, resp. $\leq a_\alpha$.
\item Suppose that $D$ is Cartier. 
Let $\tilde{Q}_D$ be the polytope in $M_\Qbb$ defined by the following inequalities, where $\chi\in M_\Qbb$: for any colored cone $(\mathcal{C},\mathcal{F})$ of $\mathbb{F}_X$, $(h_D)+\chi\geq 0$ on $\mathcal{C}$, and for any $\alpha\in(S\backslash R)\backslash\mathcal{F}_X$, $\chi(\alpha^\vee_M)+a_\alpha\geq 0$.
Note that here the weight of the canonical section of $D$ is $v^0:=\sum_{\alpha\in S\backslash R}a_\alpha\varpi_\alpha$. Then the $G$-module $H^0(X,D)$ is the direct sum, with multiplicity one, of the irreducible $G$-modules of highest weights $\chi+v^0$ with $\chi$ in $\tilde{Q}_D\cap M$.    
\end{enumerate}
\end{teo}
In all the paper, a divisor of a horospherical variety is always supposed to be $B$-stable, i.e. of the form $\sum_{i=1}^r a_i X_i +\sum_{\alpha\in S\backslash R} a_\alpha D_\alpha$.

\begin{cor}\label{cor:pic}
A $G/H$-embbeding $X$ is $\Qbb$-factorial if and only if all the cones in $\mathbb{F}_X$ are simplicial and for any $\alpha\in \mathcal{F}_X$, $\alpha^\vee_M$ generates a ray of $\mathbb{F}_X$.

\noindent Moreover, in that case, 
the Picard number of $X$ is the number of their $B$-stable prime divisors minus the rank $n$ (Equality (4.1.1) \cite{Fanohoro}).
\end{cor}

In what follows we denote by $WDiv(X)_\Qbb$ (resp. $WDiv_0(X)_\Qbb$) the vector space of $B$-invariant (linearly equivalent to zero)
$\Qbb$-Weil divisors on a horospherical variety $X$ 
and by $Amp(X)$ (resp. $Nef(X)$) the cones in $WDiv(X)_\Qbb$ of ample (resp. nef) divisors.
Both cones are polyhedral and, if $X$ is $\Qbb$-factorial, full-dimensional.

\subsection{Projective horospherical varieties and polytopes}\label{sec:divpoly}

In this section we recall how many properties of $G/H$-embeddings can be formulated in terms of moment polytopes.
Theorem~\ref{th:divcrit} gives the following 
\begin{prop}(Corollary 2.8, \cite{HMMP})\label{cor:divcrit}
Let $X$ be a projective $G/H$-embedding and $D=\sum_{i=1}^r a_i X_i +\sum_{\alpha\in S\backslash R} a_\alpha D_\alpha$ be a $\Qbb$-divisor of $X$. Suppose that $D$ is $\Qbb$-Cartier and ample.
\begin{enumerate}
\item\label{divcrit1} The polytope $\tilde{Q}_D$ defined in Theorem~\ref{th:divcrit} is of maximal dimension in $M_\Qbb$ and we have $$\tilde{Q}_D=\{m\in M_\Qbb\,\mid\,\langle m,x_i\rangle\geq -a_i,\,\forall i\in\{1,\dots,r\}\mbox{ and }\langle m,\alpha^\vee_M\rangle\geq -a_\alpha,\,\forall \alpha\in\mathcal{F}_X\}.$$
\item\label{divcrit2} Let $v^0:=\sum_{\alpha\in S\backslash R}a_\alpha\varpi_\alpha$. The polytope $Q_D:=v^0+\tilde{Q}_D$ is contained in the dominant chamber $X(P)^+$ of $X(P)$ and it is not contained in any wall of the dominant chamber. 
%$W_{\alpha,P}$, for $\alpha\in S\backslash R$. 
%Moreover, $Q_D\cap W_{\alpha,P}\neq\emptyset$ if and only if $\alpha\in\mathcal{F}_X$.
\item\label{divcrit3} Let $(\mathcal{C},\mathcal{F})$ be a maximal colored cone of $\mathbb{F}_X$, then the element $v^0-(h_D)_{|\mathcal{C}}$ of $M_\Qbb$ is a vertex of $Q_D$. In particular, if $D$ is Cartier, then $Q_D$ is a lattice polytope (i.e. has its vertices in $v^0+M$).
\item\label{divcrit4} Conversely, let $v$ be a vertex of $Q_D$. We define $\mathcal{C}_v$ to be the cone of $N_\Qbb$ generated by inward-pointing normal vectors of the facets of $Q_D$ containing $v$. 
We set $\mathcal{F}_v=\{\alpha\in S\backslash R\,\mid\,v\, \text{the corresp. wall of the dominant chamber}\}$. Then $(\mathcal{C}_v,\mathcal{F}_v)$ is a maximal colored cone of $\mathbb{F}_X$. 
%\item With a natural construction, similar to the previous two items, we obtain a \marginpar{\ora{enlever 3 4 5 6?}} bijection between faces of $Q_D$ and colored cones of the colored fan $\mathbb{F}_X$.
%\item The divisor $D$ can be computed from the pair $(Q,\tilde{Q})$ as follows: the coefficients $a_\alpha$ with $\alpha\in S\backslash R$ are given by the translation vector in $X(P)^+$ that maps $\tilde{Q}$ to $Q$; and for any $i\in\{1,\dots,r\}$, the coefficient $a_i$ is given by $-\langle v_i,x_i\rangle$ for any element $v_i\in M_\Qbb$ in the facet of $\tilde{Q}$ for which $x_i$ is an inward-pointing normal vector.
\end{enumerate}
\end{prop}

The polytope $Q_D$ is called the {\it moment polytope} of $(X,D)$ (or of $D$), and the polytope $\tilde{Q}_D$ the {\it pseudo-moment} polytope of $(X,D)$ (or of $D$).

\noindent The projective $G/H$-embeddings are classifed in terms of $G/H$-polytopes (defined below in Definition~\ref{def:G/H-eq}), and we can give an explicit construction of a $G/H$-embedding from a $G/H$-polytope.

\begin{defi}\label{def:G/H-eq}
Let $Q$ be a polytope in $X(P)^+_\Qbb$ (not necessarily a lattice polytope). We say that $Q$ is a $G/H$-polytope, if its direction is $M_\Qbb$ and if it is contained in no wall $W_{\alpha,P}$ with $\alpha\in S\backslash R$.

Let $Q$ and $Q'$ be two $G/H$-polytopes in $X(P)^+_\Qbb$. Consider any polytopes $\tilde{Q}$ and $\tilde{Q'}$ in $M_\Qbb$ obtained by translations from $Q$ and $Q'$ respectively.   We say that $Q$ and $Q'$ are equivalent $G/H$-polytopes if the following conditions are satisfied.
\begin{enumerate}
\item\label{G/H-eq0} There exist an integer $j$ and $2j$ affine half-spaces $\mathcal{H}_1^+,\dots,\mathcal{H}_j^+$ and $\mathcal{H'}_1^+,\dots,\mathcal{H'}_j^+$ of $M_\Qbb$ (respectively delimited by the affine hyperplanes $\mathcal{H}_1,\dots,\mathcal{H}_j$ and $\mathcal{H'}_1,\dots,\mathcal{H'}_j$) such that $\tilde{Q}$ is the intersection of the $\mathcal{H}_i^+$, $\tilde{Q'}$ is the intersection of the $\mathcal{H'}_i^+$, and for all $i\in\{1,\dots,j\}$, $\mathcal{H}_i^+$ is the image of $\mathcal{H'}_i^+$ by a translation.
\item\label{G/H-eq1} With the notation of \ref{G/H-eq0}, for all subset $J$ of $\{1,\dots,j\}$, the intersections $\cap_{i\in J}\mathcal{H}_i\cap Q$ and $\cap_{i\in J}\mathcal{H'}_i\cap Q'$ have the same dimension.
\item\label{G/H-eq2} $Q$ and $Q'$ intersect exactly the same walls of the dominant chamber. %$W_{\alpha,P}$ of $X(P)^+$ with $\alpha\in S\backslash R$.
\end{enumerate}
\end{defi}

Remark that this definition does not depend on the choice of $\tilde{Q}$ and $\tilde{Q'}$.
We now give a classification of projective horospherical varieties in terms of polytopes.

\begin{prop}(Proposition 2.10, \cite{HMMP})\label{prop:classprojembpoly}
The correspondence between moment polytopes and colored fans gives a bijection between the set of equivalence classes of $G/H$-polytopes and (isomorphism classes of) projective $G/H$-embeddings.

%Moreover, the set of $G$-orbits of a projective $G/H$-embedding is in bijection with the set of faces of one of its moment polytope \ora{enlever la 2nde phrase?} (preserving the respective orders). 
\end{prop}

\begin{prop}(Proposition~2.11 and Remark~2.12, \cite{HMMP})\label{prop:veryample}
Let $Q$ be a $G/H$-polytope. Up to multiplying $Q$ by an integer, there exists a very ample Cartier divisor $D$ of the corresponding $G/H$-embedding $X$ such that $Q=Q_D$.  More precisely, $X$ is isomorphic to the closure of the $G$-orbit $G\cdot [\sum_{\chi\in (v^0+M)\cap Q}v_\chi]$ in $\Pbb(\oplus_{\chi\in (v^0+M)\cap Q}V(\chi))$.
\end{prop}

\begin{lem}\label{lem:Gmorph}
Let $X$ be a $G/H$-embedding and
let $D$ be a nef divisor on  $X$.
Then there is a horospherical subgroup $H\subsetneq H'\subseteq G$
such that $Q_D$ is the polytope of a $G/H'$-embedding $X'$
and a suitable multiple of $D$ defines a morphism $X\to X'$.
\end{lem}
\begin{proof}
If $D$ is nef, the polytope $Q_D$ can be defined as in Proposition~\ref{cor:divcrit}. The colored fan that we can then construct is not necessarily $\mathbb{F}_X$, but the colored fan of a horospherical $G$-variety $Y$ (not necessarily a $G/H$-embedding), where the $G$-equivariant map $\phi_D$ associated to the nef divisor goes from $X$ to $Y$.

Indeed, up to multiplying $D$ or $Q$ by an integer, the proof of \cite[Proposition~2.11]{HMMP} and \cite[Remark~2.12]{HMMP} gives the projective $G$-equivariant map  $$\begin{array}{cccc}
\phi_D: & X & \longrightarrow & \Pbb(H^0(X,D)^\vee)\\
 & x & \longmapsto & [s\mapsto s(x)]
\end{array}$$ whose image is the closure of the $G$-orbit $G\cdot [\sum_{\chi\in (v^0+M)\cap Q}v_\chi]$ in $\Pbb(\oplus_{\chi\in (v^0+M)\cap Q}V(\chi))$. This closure is the variety $Y$ by applying Proposition~\ref{prop:veryample} to $Q=Q_D$. Thus $Q$ is a $G/H'$-polytope with $H\subseteq H'$, and $Y$ is the $G/H'$-embedding corresponding to $Q$.
\end{proof}

By the duality between colored fans and moment polytopes we easily get that if $X$ is a $\Qbb$-factorial $G/H$-embedding then for every ample $B$-invariant divisor $D$ the polytope $Q_D$ is simple. We can go further and give the following result, which is a translation of Corollary~\ref{cor:pic} in terms of polytopes, by using Proposition~\ref{cor:divcrit} \ref{divcrit3} and \ref{divcrit4}.
The matrix inequality $Ax\geq B$ comes from Proposition~\ref{cor:divcrit} \ref{divcrit1}.

\begin{lem}\label{lem:Qf}
Let $X$ be a $G/H$-embedding and let $D$ be an ample $B$-invariant divisor. The polytope $\tilde{Q}_D$ can be defined by a matrix inequality $Ax\geq B$, where the lines of $A$ are given by the $x_i$'s and the $\alpha^\vee_M$'s, and the column matrix $B$ is given by minus the coefficients of $D$. For any vertex $v$ of $\tilde{Q}_D$, denote $I_v$ the set of lines of $A$ such that $A_{I_v}x=B_{I_v}$. 

Then  $X$ is $\Qbb$-factorial
if and only if $A_{I_v}$ is surjective for any vertex $v$ of $\tilde{Q}_D$.

\end{lem}

Similarly, we have the following result. 

\begin{lem}\label{lem:Qcart}
Let $X$ be a $G/H$-embedding and let $D$ be an ample $B$-invariant divisor. Let $D'$ be a $B$-invariant divisor of $X$ and denote by $B'$ the column matrix is given by minus the coefficients of $D'$. 

Then $D'$ is $\Qbb$-Cartier if and only if for any vertex $v$ of $\tilde{Q}_D$, $B'$ is in the image of $A_{I_v}$.
\end{lem}

\begin{rem}
The existence of $G$-equivariant morphisms between horospherical varieties can be characterized in terms of colored fans \cite{knop} or equivalently in terms of moment polytopes \cite[section~2.4]{HMMP}.
In this text, we will rather use Lemma \ref{lem:Gmorph}.
\end{rem}

\section{Minimal Model Program}

In this section, we recall the results in \cite{HMMP} where the second-named author describes a minimal model program from a horospherical variety in terms of a one-parameter family of polytopes. 

\noindent We start by recalling some standard terminology for the minimal model program for projective varieties. We refer to \cite{KM98} for the basic notions on the minimal model program.

Let $X$ be a projective variety with terminal $\Qbb$-factorial singularities such that $K_X$ is not nef.
Then by the cone and contraction theorem there is a morphism $\varphi\colon X\to Y$ such that $\rho(X)=\rho(Y)+1$ and for every curve $C$ in $X$ which is contracted to a point by $\varphi$
we have $K_X\cdot C<0$. Moreover
\begin{itemize}
\item if $\dim Y<\dim X$ then $\varphi$ is called a Mori fibre space;
\item if $\dim Y=\dim X$ and $Exc(\varphi)$ has codimension 1 in $X$ then $\varphi$ is said to be divisorial;
\item if $\dim Y=\dim X$ and $Exc(\varphi)$ has codimension at least 2 in $X$ then $\varphi$ is said to be small.
\end{itemize}

In particular, the total space of a Mori fiber space is always assumed to be terminal $\Qbb$-factorial.

In the last case, by \cite{HMcK07} and \cite{HMcK10}, there is $\varphi^+\colon X^+\to Y$ such that $Exc(\varphi^+)$ has codimension at least 2 in $X^+$ and for every curve $C$ in $X^+$ contracted by $\varphi^+$
we have $K_{X^+}\cdot C>0$.
The data of $\varphi$ and $\varphi^+$ is called a {\it flip}.

In the second and third case $Y$ is birational to $X$. We set $X_1=Y$ and $\varphi_1=\varphi$ in the second case and $X_1=X^+$ and $\varphi_1=(\varphi^+)^{-1}\circ \varphi$ in the third case.
If $K_{X_1}$ is not nef, then by the cone and contraction theorem there is again a morphism $X_1\to Y_1 $.
An {\it MMP}, or minimal model program, is a sequence of birational maps $\varphi_i\colon X_{i-1}\dasharrow X_i$ obtained as above.
We say that it {\it terminates} if there is an integer $k$ such that $K_{X_k}$ is nef or there is a Mori fibre space $X_k\to T$.
A proper morphism with connected fibres is called a {\it contraction}. 
A birational morphism $\varphi\colon X\to Y$ such that $\rho(X)=\rho(Y)+1$ is called an {\it extremal contraction}. We say that a morphism $\varphi$ is {\it $K$-negative} (resp. {\it $K$-positive}) if for every curve $C$ in $X$ which is contracted to a point by $\varphi$
we have $K_X\cdot C<0$ (resp. $K_X\cdot C>0$).

%\mage {\noindent Therefore, for every $i$ the birational map $\varphi_i$ is an extremal contraction.A enlever ??????}

\subsection{HMMP scaled by an ample divisor}

Let $(X,D)$ be a polarized horospherical variety: $X$ is a $G/H$-embedding and $D$ is a $B$-stable ample  $\Qbb$-divisor of $X$. Write $D=\sum_{i=1}^rb_iX_i+\sum_{\alpha\in S\backslash R}b_\alpha D_\alpha$. An anticanonical divisor of $X$ is $-K_X=\sum_{i=1}^rX_i+\sum_{\alpha\in S\backslash R}a_\alpha D_\alpha$, where $a_\alpha$ are integers greater or equal to 2, and given by an explicit formula 
\cite[Th.~4.2]{BrionCurves}.

We consider the one-parameter family of polytopes $(\tilde{Q}^\epsilon)_{\epsilon\geq 0}$ defined  by $\langle x,x_i\rangle\geq -b_i+\epsilon$ for all $i\in\{1,\dots,r\}$ and $\langle x,\alpha^\vee_M\rangle\geq -b_\alpha+\epsilon a_\alpha$ for any $\alpha\in S\backslash R$. 
Equivalently, 
let $A$ be the matrix associated to the linear map $\varphi(m)=(\langle m,x_i \rangle_{i=1\ldots r},\;  \langle m,\alpha_M^{\vee} \rangle_{\alpha\in\mathcal F})$.
%We can also defined with matrices $A$, $B$ and $C$ by $Ax\geq B+\epsilon C$ where : the lines of $A$ correspond to the $x_i$'s and the $\alpha^\vee_M$'s; 
Let $B$ be the column matrix whose coordinates are minus the coefficents of $D$; and $C$ the column matrix whose coordinates are the coefficents of $-K_X$.

We define $(Q^\epsilon)_{\epsilon\geq 0}$  by $Q^\epsilon=\tilde{Q}^\epsilon+\sum_{\alpha\in S\backslash R}b_\alpha-\epsilon a_\alpha \varpi_\alpha$.

\begin{rem}
For small $\epsilon\in \Qbb$, $\tilde{Q}^\epsilon=\tilde{Q}_{D+\epsilon K_X}$ and $Q^\epsilon=Q_{D+\epsilon K_X}$ is the moment polytope associated to the ample divisor $D+\epsilon K_X$.
\end{rem}

\begin{teo}(Corollary 3.16 and Section~4, \cite{HMMP})\label{HMMP}
For every ample divisor $D$ over a $\Qbb$-Gorenstein horospherical variety $X$, there exists $\epsilon_{max}>0$, %such that 
%the family $(Q^\epsilon)_{\epsilon\in [0,\epsilon_{max}[}$ gives of finitely many $G/H$-embeddings describing an MMP from $X$ and terminates with a Mori fibration into a horospherical $G$-variety $T$ associated to the polytope $Q_{\epsilon_{max}}$.
and, there exist non-negative integers $k,j_0,\dots,j_k$, rational numbers $\alpha_{i,j}$ for $i\in\{0,\dots,k\}$ and $j\in\{0,\dots,j_i\}$ and $\alpha_{k,j_k+1}\in\Qbb_{>0}\cup\{+\infty\}$ ordered as follows with the convention that $\alpha_{i,j_i+1}=\alpha_{i+1,0}$ for any $i\in\{0,\dots,k-1\}$:
\begin{enumerate}
\item $\alpha_{0,0}=0$;
\item for any $i\in\{0,\dots,k\}$, and for any $j<j'$ in $\{0,\dots,j_i+1\}$ we have $\alpha_{i,j}<\alpha_{i,j'}$;
\end{enumerate}
and such that the different $G/H$-embedding associated to the polytopes in  the family $(Q^\epsilon)_{\epsilon\in\Qbb_{\geq 0}}$ are given by the following intervals:

\begin{enumerate}
\item $X_{i,0}$ when $\epsilon\in [\alpha_{i,0},\alpha_{i,1}[$, with $i\in\{0,\dots,k\}$;
\item  $X_{i,j}$ when $\epsilon\in ]\alpha_{i,j},\alpha_{i,j+1}[$, with $i\in\{0,\dots,k\}$ and $j\in\{1,\dots,j_i\}$;
\item $Y_{i,j}$ when $\epsilon=\alpha_{i,j}$ with $i\in\{0,\dots,k\}$ and $j\in\{1,\dots,j_i\}$;
\item $T$ such that $\dim T<\dim X$ when $\epsilon=\alpha_{k,j_k+1}=\epsilon_{max}$.
 \end{enumerate}
%Also $T$ is the  horospherical $G$-variety associated to the polytope $Q_{\epsilon}$ when $\epsilon=\alpha_{k,j_k+1}=\epsilon_{max}$.\\

%It also gives a projective horospherical $G$-variety $Z$ associated to the moment polytope $Q^{\alpha_{max}}$. Indeed,  let $M^1_\Qbb$ be the minimal vector subspace containing $\tilde{Q}^{\alpha_{max}}$ and let $M^1:=M^1_\Qbb\cap M$. Let $R^1$ be the union of $R$ with the set of $\alpha\in S\backslash R$ such that $Q^{\alpha_{max}}$ is contained in the wall $W_{\alpha,P}$. Then we define the subgroup $H^1$ of $P^1:=P_{R^1}$ to be  the intersection of kernels of characters of $P^1$ in $M^1$. Then $Q^{\alpha_{max}}$ is a $G/H^1$-polytope and corresponds to a $G/H^1$-embedding $Z$.

%Remark that, by definition, $M^1\subset M$ and $R\subset R^1$ so that we have a projection $\pi:G/H\longrightarrow G/H^1$.\\

Moreover we get dominant $G$-equivariant morphisms: 
\begin{enumerate} \item $\phi_{i,j}:X_{i,j-1}\longrightarrow Y_{i,j}$ for any $i\in\{0,\dots,k\}$ and $j\in\{1,\dots,j_i\}$;
\item $\phi_{i,j}^+:X_{i,j}\longrightarrow Y_{i,j}$ for any $i\in\{0,\dots,k\}$ and $j\in\{1,\dots,j_i\}$; \item $\phi_i:X_{i,j_i}\longrightarrow X_{i+1,0}$ for any $i\in\{0,\dots,k-1\}$;
\item and $\phi:X_{k,j_k}\longrightarrow T$.
\end{enumerate}
For every $i,j$ the morphism  $\phi_{i,j}$ is $K$-negative,  the morphism $\phi_{i,j}^+$ is $K$-positive and their exceptional loci have codimension at least 2.
For every $i$ the morphism $\phi_i$ is a $K$-negative divisorial contraction and $\phi$ is a fibration.
%Moreover the morphisms $\phi_{i,j}$ and $\phi_{i,j}^+$ are flips, the morphisms $\phi_i$ are divisorial contractions and that the morphism $\phi$ is a Mori fibration.

%\textcolor{red}{Y a-t-il besoin de pr\'eciser que $\phi_i$ et $\phi_{i,j}$ sont extr\'emales et contractres des courbes n\'egatives le long du canonique, et $\phi_{i,j}+$ extr\'emales et contractres des courbes positives le long du canonique??}
\end{teo}

We will refer to the series of maps of Theorem \ref{HMMP} as {\it HMMP}, for horospherical minimal model program.
\begin{rem}
We notice that the algorithm described in Theorem \ref{HMMP} is not necessarily an MMP, as the morphisms involved are not extremal contractions. Nevertheless, if $X$ is terminal $\Qbb$-factorial,  it is an MMP for a general choice of $D$, in particular $\phi$ is a Mori fiber space.
\end{rem}

%\begin{rem}
%In \cite{HMMP}, the family $(Q^\epsilon)_{\epsilon\in\Qbb_{\geq 0}}$ is constructed in two steps, in order to remove an inequality associated to a $G$-stable divisor, as soon as there is a divisorial contraction. This is for a technical reason to prove the theorem, but we can keep these useless inequalities to simplify the construction/definition of the family $(Q^\epsilon)_{\epsilon\in\Qbb_{\geq 0}}$.\marginpar{\ora{enlever?}}
%\end{rem}

\begin{rem}\label{rem:notpseudo}
In Theorem~\ref{HMMP}, we can also use the family $(\tilde{Q}^\epsilon)_{\epsilon\in\Qbb_{\geq 0}}$. By replacing the walls of the dominant chamber by the inequalities in $Ax\geq B+\epsilon C$ coming from an $\alpha\in S\backslash R$, the study of the family and associated horospherical variety is simpler.
\end{rem}

%% file: resolHMMP.tex
\section{Resolutions and Horospherical MMP}

The main goal of this section is to prove that if $X/T$ is a Mori fiber space and $Z\to X$ a resolution of the singularites of $X$, then 
there is a euclidian open set in $WDiv(Z)_{\Qbb}$ of ample divisors $A$ such that 
the HMMP from $Z$ with scaling of $A$ ends with  $X/T$.

%Notation (sans doute \`a donner avant): $WDiv(X)_\Qbb$ is the $\Qbb$-vector space of $B$-stable $\Qbb$-divisor of $X$, and $WDiv_0(X)_\Qbb$ is the subsapce of the ones that are linearly equivalent to 0.

\begin{lem}\label{lem:exres}
 Let $X$ and $Y$ be $G/H$-embeddings. Then there is a smooth $G/H$-embedding $Z$ and $Z\to X\times Y$ a resolution of the indeterminacy of $X\dasharrow Y$. 
\end{lem}

\begin{proof}
The existence of a smooth resolution of the indeterminacy of $X\dasharrow Y$ 
 is a consequence of the same result for toric varieties.
Indeed, we can first unpick colors both for $X$ and $Y$ to obtain toroidal varieties. Then, as for toric varieties we choose a common smooth subdivision of the two fans.
\end{proof}

%The next result allows to identify a moment polytope of $X$ with a polytope of the family of polytopes associates to an HMMP of a resolution of singularities of $X$.

\begin{lem}\label{lem:polyZX}
Let $X$ and $Z$ be terminal and $\Qbb$-factorial $G/H$-embeddings such that $Z$ is a resolution of the singularities of $X$.

%It is enough to prove the result for $X$ : there exists an ample divisor $A_Z$ of $Z$ such that the HMMP from $Z$ scaled by $A_Z$ ends with $X\longrightarrow T$. 

Let $E_1,\dots,E_k$ be the exceptional divisors of $\phi\colon Z\longrightarrow X$. 
Then, for any non-negative rational numbers $d_1,\dots,d_k$, we have $Q_D=Q_{\phi^*(D)+\sum_{i=1}^kd_iE_i}$.

\end{lem}

\begin{proof}
Since $X$ has terminal singularites, we have $K_Z=\phi^*K_X+\sum_{i=1}^k a_iE_i$ with  $a_i>0$ for every $i$.

We denote by $X_1,\dots,X_r$ the $G$-stable irreducible divisors of $X$. For any $i\in\{1,\dots,r\}$, we denote by $x_i$ the primitive element 
in $N$ of the colored edge associated to $X_i$. Then the $G$-stable irreducible divisors of  $Z$ are $X_1,\dots,X_r,E_1,\dots,E_k$.
We denote by $e_i$ the  primitive element in $N$ of the colored edge associated to $E_i$.
Let $D=\sum_{i=1}^r b_iXi+\sum_{\alpha\in S\backslash R}b_\alpha D_\alpha$ be an ample $B$-stable divisor of $X$. 
Denote by $c_1,\dots,c_k$ the rational numbers such that $\phi^*(D)= \sum_{i=1}^r b_iXi+\sum_{j=1}^k c_jE_j+\sum_{\alpha\in S\backslash R}b_\alpha D_\alpha$. 
Then, for any non-negative rational numbers $d_1,\dots,d_k$, the polytope $Q_D$ coincides with

\begin{align*}
Q_{\phi^*(D)+\sum_{i=1}^kd_iE_i}:=&\\
\{m\in M_\Qbb\mid&\,\langle m,x_i\rangle \geq -b_i, \langle m,e_j\rangle \geq -c_j-d_j,\langle m,\alpha^\vee_M\rangle\geq -a_\alpha\,\mbox{for all } i,j,\alpha\}.
\end{align*}

By \cite{briondiv}, the divisors $D$ and $\phi^*(D)$ are the divisors associated to the same piecewise linear function $h_D$ defined in Theorem~\ref{th:divcrit}.
Therefore, the coefficients $d_j$ are the values $h_D(e_j)$.
In terms of polytopes, since $D$ is ample, this means that $\mathcal{H}_j:=\{m\in M_\Qbb\,\mid\,\langle m,e_j\rangle=-c_j\}$
intersects $Q_D$ along a face of $Q_D$. In particular, intersecting $Q_D$ with the halfspace $\{m\in M_\Qbb\,\mid\,\langle m,e_j\rangle\geq -c_j-d_j\}$ does not change the polytope.
\end{proof}

\begin{lem}\label{lem:HMMPendsXT}
Let $X/T$ be a Mori fibre space such that $X$ is a $G/H$-embedding.
The set of ample $B$-stable $\Qbb$-divisors $D$ of $X$ such that the HMMP from $X$ scaled by $D$ gives $X\longrightarrow T$ is an open cone $C_1$ in $WDiv(X)_\Qbb$. 
\end{lem}

\begin{proof}
Let $R$ be the extremal ray of $NE(X)$ corresponding to the Mori fibration $X\longrightarrow T$.
Let $F$ be the dual facet of $R$ in $Nef(X)_\Qbb$.
By abuse of notation we denote by $F$ the facet $F+WDiv_0(X)_\Qbb$ of the inverse image of $Nef(X)_\Qbb$ in $WDiv(X)_\Qbb$.

Then the HMMP from $X$ scaled by $D$ gives $X\longrightarrow T$ if and only if the half line $D+\Qbb_+K_X$ intersects the boundary of $Nef(X)_\Qbb+WDiv_0(X)_\Qbb$ 
in the relative interior $\mathring{F}$ of $F$. Equivalently, $D$ is in $C_1':=\mathring{F}-\Qbb_+K_X=\Qbb_+(\mathring{F}-K_X)$.

Note that $-K_X\cdot C>0$ for any $[C]\in R$, so that $-K_X$ is not in $F$ and then $C_1'$ is of maximal dimension.
This implies that $D'-\epsilon K_X$ is ample for any $D'\in \mathring{F}$ and $\epsilon>0$ small enough 
and then $C_1'$ intersects $Amp(X)_\Qbb+WDiv_0(X)_\Qbb$ non trivially.

Finally, the set $C_1=C_1'\cap Amp(X)_\Qbb+WDiv_0(X)_\Qbb$ is the sought open cone.
\end{proof}

\begin{prop}\label{pro:resMMP}
 Let $X/T$ be a Mori fibre space such that $X$ is a $G/H$-embedding.
 Let $Z\to X$ be a resolution of singularities in the category of $G/H$-embeddings.
 Then there is an euclidian open neighborhood $U_X$ of $WDiv(Z)_\Qbb$ such that every divisor in $U_X$ is ample and for every $A\in U_X$ 
 the Mori fiber space $X/T$ is the outcome of the HMMP from $Z$ with scaling of $A$.
\end{prop}
\begin{proof}
Let $C_1$ be the open cone in $WDiv(X)_\Qbb$ of Lemma \ref{lem:HMMPendsXT}.
Let $D\in C_1$. Let $\eta$ be small enough such that $D-\eta K_X$ ample.

We have $WDiv(Z)_\Qbb=\phi^*(WDiv(X)_\Qbb)\oplus Vect(E_1,\dots,E_k)$, and $\phi^*(Nef(X))$ is a face of $Nef(Z)=\overline{Amp(Z)}$. 
Set $$A:=\phi^*(D-\eta K_X)+\sum_{i=1}^k b_iE_i.$$
There exists an open polyhedron $Pol_D$ of $\Qbb^k$ containing 0 in its boundary, such that for any $(b_1,\dots,b_k)\in Pol_D$
the divisor $A$ is ample and $d_j:=b_j+\eta a_j>0$ for 
all $j\in\{1,\dots,k\}$. 

Since $A+\eta K_Z=\phi^*(D)+\sum_{i=1}^k d_iE_i$, we have $Q_{A+\eta K_Z}=Q_D$ by Lemma \ref{lem:polyZX}. This, together with Lemma~\ref{lem:HMMPendsXT}, proves that the HMMP scaled by $A$  ends with the Mori fiber space $X/T$. Indeed, the inequalities coming from the exceptional divisors $E_i$ are necessary to define the polytope $Q_A$ but not the polytope $Q_{A+\eta K_Z}$ and the polytopes $Q_{A+\epsilon K_Z}$ with $\epsilon\geq \eta$.

Choose $U_X$ to be the set of  divisors $A=\phi^*(D-\eta K_X)+\sum_{i=1}^k b_iE_i$, with $D\in C_1$, $\eta>0$ such that $D-\eta K_X$ is ample and $(b_1,\dots,b_k)\in Pol_{D,\eta}$. Remark that by construction, since $Amp(X)$ and $Amp(Z)$ are polyhedral, the condition $D-\eta K_X$  ample is given by an inequality depending linearly on the coefficents of $D$, and the polyhedron $Pol_D$ is given by inequalities depending linearly on the coefficents of $D$ and $\delta$. In particular, $U_X$ is also an open polyhedron.

%Let $A$ be a $B$-stable divisor defined by $$A:=\phi^*(D-\eta K_X)+\sum_{i=1}^k b_iE_i,$$
%with  $\eta>0$, and $b_1,\dots,b_k$ rational numbers.

%The strategy is the following : choose $D$ such that the HMMP from $X$ scaled by $D$ gives $X\longrightarrow T$, 
%then pick $\eta$ small enough to get $D-\eta K_X$ ample, and finally choose $b_1,\dots,b_k$ such that $A_Z$ is ample and for 
%all $j\in\{1,\dots,k\}$, $d_j:=b_j+\eta a_j>0$, so that the polytope $Q^\eta_{A_Z}$ corresponds to the moment polytope $Q_D$ of $X$. 

%Let $D$ be in $C_1$. Let $\eta>0$ such that $D-\eta K_X$ is ample. 
%We have $WDiv(Z)_Q=\phi^*(WDiv(X)_\Qbb)\oplus Vect(E_1,\dots,E_k)$. And $\phi^*(Nef(X))$ is a face of $Nef(X)=\overline{Amp(X)}$. 
%Then there exist a polyhedron $Pol$ of $\Qbb^k$ containing 0 in its boundary, such that for any $(b_1,\dots,b_k)\in Pol$, we have the divisor $A_Z:=\phi^*(D-\eta K_X)+\sum_{i=1}^kb_iE_i$ is ample. 

\end{proof}